\documentclass[cjm]{ipart}

\pubyear{2014}
\volume{0}
\issue{0}
\firstpage{1}
\lastpage{1}

\makeatletter
\let\c@author\relax
\makeatother
\usepackage{etex}
\usepackage[style=alphabetic,doi=false,isbn=false,url=false,maxbibnames=99,backend=bibtex]{biblatex}
\DeclareFieldFormat[article,incollection,unpublished]{title}{#1}\renewbibmacro{in:}{\ifentrytype{article}{}{\printtext{\bibstring{in}\intitlepunct}}}
\usepackage{fixltx2e}

\bibliography{breuil-schneider}

\RequirePackage{amsmath}
\RequirePackage{amssymb}
\RequirePackage{amsxtra}
\RequirePackage{amsfonts}
\RequirePackage{latexsym}
\RequirePackage{euscript}
\RequirePackage{amscd}
\RequirePackage{amsthm}
\RequirePackage{xypic}
\usepackage{stmaryrd}
\usepackage{mathtools}
\usepackage{bm}
\xyoption{all}

\newcommand{\BS}{\operatorname{BS}}
\newcommand{\ybar}{\overline{y}}

\newcommand{\Rinfty}{{R_\infty}}

\let\emptyset\varnothing

\def\A{\mathbb A}

\def\C{\mathbb C}
\def\F{\mathbb F}

\def\Q{\mathbb{Q}}

\def\T{\mathbb{T}}

\def\Z{\mathbb{Z}}

\def\Qbar{\overline{\Q}}

\def\Zbar{\overline{\Z}}

\def\m{\mathfrak m}

\def\Mod{\mathrm{Mod}}

\def\unif{\varpi}
\def\pr{\mathrm{pr}}

\def\alg{\mathrm{alg}}
\def\lalg{\mathrm{l.alg}}
\def\ladm{\mathrm{l.adm}}

\def\rig{\mathrm{rig}}
\def\ab{\mathrm{ab}}

\def\sm{\mathrm{sm}}

\def\cont{\mathrm{cont}}

\def\GL{\operatorname{GL}}

\def\Gal{\mathrm{Gal}}

\def\End{\mathrm{End}}

\def\Art{\mathop{\mathrm{Art}}\nolimits}

\def\Hom{\mathop{\mathrm{Hom}}\nolimits}

\def\Spec{\mathop{\mathrm{Spec}}\nolimits}
\def\Spf{\mathop{\mathrm{Spf}}\nolimits}
\def\Frob{\mathop{\mathrm{Frob}}\nolimits}

\def\Ind{\mathop{\mathrm{Ind}}\nolimits}
\newcommand{\cInd}{{\operatorname{c-Ind}}}

\def\rhobar{\overline{\rho}}

\def\cotimes{\widehat{\otimes}}

\def\WD{\mathrm{WD}}
\def\St{\mathrm{St}}

\def\m{\mathfrak{m}}

\def\iso{\xrightarrow{\sim}}

\def\fkZ{\mathfrak{Z}}
\DeclareMathOperator{\MaxSpec}{MaxSpec}
\DeclareMathOperator{\Tor}{Tor}

\swapnumbers

\newcommand{\onto}{\twoheadrightarrow}
\newcommand{\ra}{\rightarrow}

\newcommand{\into}{\hookrightarrow}

\newcommand{\isoto}{\xrightarrow{\sim}}

\newlength{\ownl}

\newcommand{\ad}{{\operatorname{ad}\,}}

\newcommand{\Ann}{{\operatorname{Ann}\,}}

\newcommand{\diag}{{\operatorname{diag}}}

\renewcommand{\Im}{{\operatorname{Im}\,}}

\newcommand{\Irr}{{\operatorname{Irr}}}

\newcommand{\nind}{\mathop{i}}

\newcommand{\rec}{{\operatorname{rec}}}

\newcommand{\Res}{{\operatorname{Res}}}

\newcommand{\tr}{{\operatorname{tr}\,}}

\newcommand{\wt}[1]{\widetilde{#1}}

\newcommand{\Sp}{\operatorname{Sp}}

\newcommand{\loc}{{\operatorname{loc}}}

\newcommand{\op}{{\operatorname{op}}}

\newcommand{\semis}{{\operatorname{ss}}}

\newcommand{\univ}{{\operatorname{univ}}}

\newcommand{\CO}{{\mathcal{O}}}

\newcommand{\CX}{{\mathcal{X}}}

\newcommand{\cA}{\mathcal{A}}

\newcommand{\cG}{\mathcal{G}}
\renewcommand{\cH}{\mathcal{H}}

\newcommand{\cM}{\mathcal{M}}
\newcommand{\cN}{\mathcal{N}}
\newcommand{\cO}{\mathcal{O}}

\newcommand{\cP}{\mathcal{P}}
\newcommand{\cQ}{\mathcal{Q}}
\newcommand{\cS}{\mathcal{S}}

\newcommand{\cX}{\mathcal{X}}

\newcommand{\ga}{{\mathfrak{a}}}
\newcommand{\gb}{{\mathfrak{b}}}

\newcommand{\gd}{{\mathfrak{d}}}

\newcommand{\gp}{{\mathfrak{p}}}

\newcommand{\barR}{\overline{{R}}}

\newcommand{\tF}{\widetilde{{F}}}
\newcommand{\tG}{\widetilde{{G}}}

\newcommand{\tQ}{\widetilde{{Q}}}

\newcommand{\tT}{\widetilde{{T}}}

\newcommand{\tZ}{\widetilde{{Z}}}

\newcommand{\tp}{\widetilde{\mathfrak{p}}}

\newcommand{\tv}{{\widetilde{{v}}}}

\newcommand{\alphabar   }{\overline{\alpha  }}

\newcommand{\epsilonbar    }{\overline{\epsilon}}
\newcommand{\varepsilonbar  }{\overline{\varepsilon}}

\newcommand{\psibar   }{\overline{\psi}}

\def\RCS$#1: #2 ${\expandafter\def\csname RCS#1\endcsname{#2}}
\RCS$Revision: 85 $
\RCS$Date: 2011-12-16 18:54:17 -0600 (Fri, 16 Dec 2011) $

 \newcommand{\bigO}{\mathcal{O}}

 \newcommand{\p}{\mathfrak{p}}

\newcommand{\bb}{\mathbb}
\newcommand{\mc}{\mathcal}
\newcommand{\mf}{\mathfrak}

\newcommand{\rbar}{{\bar{r}}}
\newcommand{\sbar}{{\bar{s}}}

\newcommand{\HT}{\operatorname{HT}}

 \newcommand{\Qp}{{\Q_p}}

\newcommand{\Qpbar}{{\overline{\Q}_p}}
\newcommand{\Qpbartimes}{{\overline{\Q}_p^\times}}

 \newcommand{\OO}{\mathcal O}
\DeclareMathOperator{\wtimes}{\widehat{\otimes}}
\newcommand{\mm}{\mathfrak m}
\newcommand{\qqq}{\mathfrak q}
\newcommand{\pp}{\mathfrak p}
\newcommand{\Zfin}{ Z\mathrm{-fin}}
\newcommand{\proaug}{\mathrm{pro.aug}}

\usepackage{hyperref}

\newtheorem{thm}[subsection]{Theorem}

\newtheorem{lem}[subsection]{Lemma}

\newtheorem{ithm}{Theorem}

\newtheorem{iconj}[ithm]{Conjecture}

\newtheorem{cor}[subsection]{Corollary}
\newtheorem{conj}[subsection]{Conjecture}

\newtheorem{prop}[subsection]{Proposition}

\theoremstyle{definition}

\theoremstyle{remark}
\newtheorem{remark}[subsection]{Remark}
\newtheorem{rem}[subsection]{Remark}

\newtheorem{condition}[subsection]{Condition}

\def\numequation{\addtocounter{subsection}{1}\begin{equation}}
\def\nummultline{\addtocounter{subsubsection}{1}\begin{multline}}
\def\anumequation{\addtocounter{subsection}{1}\begin{equation}}

\begin{document}

\begin{frontmatter}

\title{Patching and the $p$-adic local Langlands correspondence\protect\thanksref{T1}}
\runtitle{$p$-adic local Langlands correspondence}
\thankstext{T1}{A.C.\ was partially
  supported by the NSF Postdoctoral Fellowship DMS-1204465 and NSF Grant DMS-1501064. M.E.\ was
  partially supported by NSF grants DMS-1003339, DMS-1249548, and DMS-1303450. T.G.\ was
  partially supported by a Marie Curie Career Integration Grant, and by an
  ERC Starting Grant. D.G.\ was partially supported by NSF grants
  DMS-1200304 and DMS-1128155. V.P.\ was partially supported by the DFG,
  SFB/TR45. S.W.S.\ was partially supported by NSF grant DMS-1162250
  and a Sloan Fellowship.
}

\begin{aug}
 \author{Ana Caraiani\ead[label=e1]{caraiani@princeton.edu}},
 \address{Department of Mathematics, Princeton University, Fine Hall,
\\Washington Rd., Princeton, NJ 08544, USA \\ 
 \printead{e1}}
 \author{Matthew Emerton\ead[label=e2]{emerton@math.uchicago.edu}},
\address{Department of Mathematics, University of Chicago,\\
5734 S.\ University Ave., Chicago, IL 60637, USA \\
\printead{e2}}
\author{Toby Gee\ead[label=e3]{toby.gee@imperial.ac.uk}},
 \address{Department of
  Mathematics, Imperial College London,\\
  London SW7 2AZ, UK \\ \printead{e3}}
 \author{David Geraghty\ead[label=e4]{david.geraghty@bc.edu}},
 \address{Department of Mathematics, 301 Carney Hall,
  Boston College, \\Chestnut Hill, MA 02467, USA \\ \printead{e4}} 
\author{Vytautas
  Pa\v{s}k\=unas\ead[label=e5]{paskunas@uni-due.de}}\address{
   Fakult\"at f\"ur Mathematik, Universit\"at Duisburg-Essen,  \\45117 Essen, Germany \\ \printead{e5}}
\and
\author{Sug Woo Shin\ead[label=e6]{sug.woo.shin@berkeley.edu}} \address{Department of Mathematics, UC Berkeley, Berkeley, CA 94720, USA \\Korea Institute for Advanced Study, 85 Hoegiro,
Dongdaemun-gu, \\Seoul 130-722, Republic of Korea \\ \printead{e6}}

 \runauthor{A. Caraiani et al.}

\end{aug}
\begin{abstract}We use the patching method of Taylor--Wiles and Kisin
  to construct a candidate for the $p$-adic local Langlands
  correspondence for $\GL_n(F)$, $F$ a finite extension of $\Qp$. We
  use our construction to prove many new cases of the
  Breuil--Schneider conjecture.
 \end{abstract}

\tableofcontents

\end{frontmatter}

\setcounter{tocdepth}{1}
\section{Introduction}\label{sec:introduction}
Our goal in this paper is to use global methods (specifically, the
Taylor--Wiles--Kisin patching method) to construct a candidate for the
$p$-adic local Langlands correspondence for $\GL_n(F)$, where $F$ is an
arbitrary finite extension of $\Qp$, and $p\nmid 2n$. At present, the existence of such
a correspondence is only known for $\GL_1(F)$ (where it is given by local
class field theory), and for $\GL_2(\Qp)$ (\emph{cf.}\ \cite{MR2642409,paskunasimage}). We do not prove that our construction gives a purely local correspondence (and
it would perhaps be premature to conjecture that it should),
but we are able
to say enough about our construction to prove many new cases of the
Breuil--Schneider conjecture, and to reduce the general case of the
Breuil--Schneider conjecture (under some mild technical hypotheses) to standard conjectures related to automorphy lifting
theorems.

The idea that global methods could be used to
construct the correspondence is a natural one; the only proofs at present of the
classical local Langlands correspondence~\cite{ht,MR1738446, scholze_local_LL} are by global means,
and indeed the first proofs of local class field theory were global. The basic
idea is to embed a local situation into a global one, apply a global
correspondence (for example, the association of Galois representations to
certain automorphic forms),
and then to prove that the construction is
independent of the choice of global situation. In this paper, we carry out the
first half of this idea (although, in contrast to the constructions of~\cite{ht}, the direction of the correspondence we construct is from
representations of the local Galois group $G_F$
to representations of $\GL_n(F)$). We intend to return to the second half (investigating
the question of independence of the global situation) in subsequent
work. (See Section~\ref{sec:LL} for a discussion of the relationship of our
construction to conjectural extensions of the existing $p$-adic Langlands
correspondence for $\GL_2(\Qp)$.)

\subsection{A candidate for $p$-adic local Langlands}
Let $p\nmid 2n$ be prime, and let $\F$ be the residue field of the ring
of integers $\cO$ in some fixed field of coefficients $E$, a finite
extension of~$\Qp$. The main construction of the paper associates to any representation
$\rbar:G_F \to \GL_n(\F)$
a compact, torsion free $\cO$-module
$M_{\infty}$ equipped with commuting actions
of $\GL_n(F)$ and of a local $\cO$-algebra $R_{\infty}$, which is a formal
power series ring
over the universal lifting $\cO$-algebra $R_{\rbar}^{\square}$. 

If $r: G_F \to \GL_n(E)$ is a continuous lifting of $\rbar$ (with respect
to some suitable integral structure on $r$), and $y :R_{\infty} \to \cO$
is a homomorphism  compatible with the homomorphism $x:R_{\rbar}^{\square} \to
E$ arising from (an appropriate choice of integral structure on) $r$,
then we define $V(r) := (M_{\infty} \otimes_{R_{\infty},y}\cO)^d[1/p]$
(where $d$ denotes the Schikhof dual); this is a continuous, unitary
$E$-Banach space representation of $\GL_n(F)$, which we show is furthermore
admissible
(Proposition~\ref{prop: V(r) is admissible Banach} below).\footnote{In the
definition of $V(r)$ given there, certain additional restrictions are
placed on the choice of extension of $x$ to $y$; we suppress this technical
point in the present discussion.}

This construction is the key point of this paper,
and the most optimistic of us expect that $r\mapsto V(r)$ will realize the $p$-adic Langlands correspondence for $\GL_n(F)$.
At the moment we do not have enough control over the representations $V(r)$ for an arbitrary Galois representation $r$ to say anything definitive in this
direction in general. (Indeed,
{\em a priori}, $V(r)$ depends not just on $r$, but on the choices of $x$
and of $y$; furthermore, it is not evidently non-zero.)
However,
if $r$ satisfies the assumptions of Theorem \ref{main thm on BS,
  abstract intro version} below, namely is potentially crystalline and generic with regular Hodge--Tate weights,
then we can show (Theorem ~\ref{computation of locally algebraic vectors}) that
the subspace of locally algebraic vectors in $V(r)$  is isomorphic to
the locally algebraic representation
$\BS(r)$ associated to $r$ by Breuil and Schneider \cite{MR2359853}.
This result is very much in the spirit of a Langlands correspondence,
as $\BS(r)$
is a tensor product of an algebraic representation, which encodes the information about the Hodge--Tate weights of $r$, and a smooth representation, which corresponds to the Weil--Deligne representation of $r$ by the classical Langlands correspondence.

\subsection{The Breuil--Schneider conjecture}\label{subsec: intro BS}
Before giving more details of our construction, we discuss an
application of it to the Breuil--Schneider
conjecture~\cite{MR2359853}; this conjecture predicts that locally algebraic
representations of $\GL_n(F)$ admit invariant norms (and thus nonzero
completions to unitary Banach representations) if and only if they
arise from regular de Rham Galois representations by applying (a
generic version of) the classical local Langlands correspondence to
the corresponding Weil--Deligne representations. (This conjecture is
motivated by the case of $\GL_2(\Qp)$, where it is an immediate
consequence of known properties of the $p$-adic local Langlands
correspondence.) In one direction, Hu~\cite{MR2560407} showed that if
such a norm exists, the locally algebraic representation necessarily comes from a
regular de Rham representation.

The converse direction is largely open. We recall the conjecture in
more detail in Section~\ref{sec:breuilschneider} below, to which the
reader should refer for any unfamiliar notation or terminology. As we
remarked above, given
a de Rham representation $r:G_F\to\GL_n(\Qpbar)$ of regular weight,
in \cite{MR2359853} there is associated to $r$
a locally algebraic $\Qpbar$-representation
$\BS(r)$ of $\GL_n(F)$. The following is \cite[Conjecture
4.3]{MR2359853} (in the open direction).

\begin{iconj}\label{introconj: BS} If $r:G_F\to\GL_n(\Qpbar)$ is de
  Rham and has regular weight, then $\BS(r)$ admits a
  nonzero unitary Banach completion.
  \end{iconj}In fact, the known properties of the $p$-adic local Langlands correspondence
  for $\GL_2(\Qp)$ suggest that there should even be a nonzero admissible
  completion. Conjecture~\ref{introconj: BS} was proved in the case that
  $\pi_\sm(r)$ is supercuspidal in~\cite[Theorem 5.2]{MR2359853}, and in the more general case that $\WD(r)$ is indecomposable in~\cite{sorensenBS}. (The notation is defined in \S \ref{subsec:notation}). The
  argument of~\cite{sorensenBS} is global. It makes use of a strategy of one of
  us (M.E.) who observed that if $r$ arises as the local Galois representation
  coming from an automorphic representation, then one can obtain an admissible
  completion from the completed cohomology of~\cite{MR2207783}, \emph{cf.}\
  Proposition 4.6 of~\cite{MR2181093} (and also~\cite{sorensenBS2}). However, as
  there are only countably many automorphic representations, it is not possible
  to say anything about most principal series representations in this way; indeed,
  as was already remarked in~\cite{MR2359853} (see the discussion before Remark
  5.7), the principal series case seems to be the deepest case of the conjecture.

  Other than for $\GL_2(\Qp)$, the only previous results in the
  general principal series case that we are aware of are those
  of~\cite{1302.3060,1207.4517,MR2872444,MR2441702}, which prove the conjecture for certain
  principal series cases for $\GL_2(F)$, under additional restrictions
  on the Hodge filtration of $r$. The methods of these papers do not
  seem to shed any light on the stronger question of the existence of
  admissible completions. Under the assumption that $p\nmid 2n$, which we make from now on, we
have associated an admissible unitary Banach
representation $V(r)$ to any continuous representation
$r:G_F\to\GL_n(\Qpbar)$. In order to prove Conjecture~\ref{introconj: BS},
it would be enough to establish that $V(r)$ contains a copy of
$\BS(r)$ when $r$ is de Rham of regular weight. We expect this to be true in general, and we are able to show that it is equivalent to proving a certain
automorphy lifting theorem.  The following is our main result in this direction. (See
sections~\ref{sec:patching} and~\ref{sec:breuilschneider} for any
unfamiliar terminology; note in particular that the hypothesis that $r$ lies on an
automorphic component does not imply that $r$ arises from the Galois
representation associated to an automorphic representation, but is
rather the much weaker condition that it lies on the same component
of a local deformation ring as some such representation. It is a
folklore conjecture (closely related to the problem of deducing the
Fontaine--Mazur conjecture from generalisations of Serre's conjecture
via automorphy lifting theorems) that every de Rham representation of
regular weight satisfies this condition. Note also that we call a potentially crystalline representation $r:G_F\to\GL_n(\Qpbar)$ \emph{generic} if
the smooth representation of $\GL_n(F)$ corresponding
via the classical local Langlands correspondence
to the Weil--Deligne representation underlying $r$
is  generic, i.e.\ admits a Whittaker model; see Section 2.3 of~\cite{kudla} for more details on this notion.)

\begin{ithm}[Theorem~\ref{thm: BS for points in support of patched
    modules} and Remark~\ref{rem: definition of an automorphic component}]
  \label{main thm on BS, abstract intro version}Suppose that $p\nmid
  2n$, that $r:G_F\to\GL_n(\Qpbar)$ is potentially crystalline of
  regular weight, and that $r$ is generic. Suppose further that $r$ lies on an
  automorphic component of the corresponding potentially crystalline
  deformation ring. Then $\BS(r)$ admits a nonzero unitary admissible
  Banach completion.
\end{ithm}

By taking known automorphy lifting theorems, in particular those proved in
~\cite{BLGGT}, we are able to deduce new cases of
Conjecture~\ref{introconj: BS}. In particular, we deduce the following
result (Corollary~\ref{cor: specifc cases of BS}).

\begin{ithm}\label{main thm on BS, explicit intro version}Suppose that $p>2$, that
    $r:G_F\to\GL_n(\Qpbar)$ is de Rham of regular weight, and that $r$ is
    generic. Suppose further that either
    \begin{enumerate}
    \item $n=2$, and $r$ is potentially Barsotti--Tate, or
    \item $F/\Qp$ is unramified and $r$ is crystalline with
      Hodge--Tate weights in the extended Fontaine--Laffaille
      range, and $n\ne p$.\end{enumerate}
Then  $\BS(r)$ admits a nonzero unitary admissible
    Banach completion.
\end{ithm}
Actually, we prove a more general result (Corollary~\ref{cor: our result
  on BS}) which establishes the conjecture for potentially diagonalisable
representations; conjecturally, every potentially crystalline
representation is potentially diagonalisable. We remark that while we expect
these results to extend to potentially
semistable (rather than just potentially crystalline) representations, and to
non-generic representations,
we have restricted to the potentially crystalline case for
two reasons: we can use the main theorems
of~\cite{BLGGT} without modification, and we do not have to consider issues
related to the possible reducibility of $\BS(r)$.

\subsection{The patching construction}
\label{subsec:intro overview}
In the proof of the classical local Langlands
correspondence~\cite{ht,MR1738446}, the globalisation argument
uses a reduction to the supercuspidal case (via the classification of irreducible
smooth representations of $\GL_n(F)$ given
in~\cite{BZ77ENS,Zel80}), and then uses
trace formula methods to realise supercuspidal representations as the
local components of cuspidal automorphic representations. No such
argument is possible in our setting; there are only countably many
automorphic representations, but already for $\GL_2(\Qp)$ there are
uncountably many irreducible $p$-adic Galois representations (even up to twist).

It is natural to hope that in the $p$-adic setting, one could carry
out an analogous globalisation using ``$p$-adic automorphic
representations'', such as those arising from the completed cohomology
of~\cite{MR2207783}. However, since the locally algebraic vectors in
completed cohomology are computed by classical automorphic
representations, one cannot expect to see any regular de Rham Galois
representations in completed cohomology other than those arising from
classical automorphic representations.

The globalisation argument in
the proof of classical local Langlands is effectively a result showing
the Zariski-density of automorphic points in the Bernstein spectrum;
the analogous result for $p$-adic local Langlands (or rather, for the
part of it pertaining to regular de Rham representations) would be a
Zariski-density result for automorphic points in the corresponding
local Galois deformation rings. This is not known in general, but
strong results in this direction follow from the Taylor--Wiles--Kisin
patching method, which provides a Zariski-density result for a
non-empty collection of components of a local deformation ring (and in
general shows that each component either contains no automorphic
points, or a Zariski-dense set of points;
as mentioned above, the
problem of showing that each component contains an automorphic point
is closely related to the problem of deducing the Fontaine--Mazur
conjecture from generalisations of Serre's conjecture, \emph{cf.}
Remark~5.5.3 of~\cite{emertongeerefinedBM}).

The Taylor--Wiles--Kisin method patches together spaces of
automorphic forms with varying tame level. Traditionally, the weight and the
$p$-part of the level of these forms is fixed, and one obtains a
patched module for a certain universal local deformation ring corresponding
to de Rham representations of fixed Hodge--Tate weights, and a fixed
inertial type. In the present paper, we instead vary over all weights and levels at $p$,
obtaining a module $M_\infty$ over the unrestricted local deformation ring
(with some power series variables adjoined). By
construction, $M_\infty$ naturally has an action of $\GL_n(\cO_F)$;
by keeping track of the action of the Hecke operators at $p$, we are
able to promote this to an action of $\GL_n(F)$. Dualising the fibre of
this patched module at the point corresponding to a particular Galois
representation $r$, and inverting $p$, gives the unitary admissible
Banach representation $V(r)$ that we seek. The condition that $p\nmid 2n$ is
needed to employ the Taylor--Wiles--Kisin method (for example, this condition is
necessary in order to be able to appeal to various results from~\cite{BLGGT}), but
we suspect that it is not ultimately needed to carry out variants of these
constructions.

We do not know whether it is reasonable to expect that our construction is purely local, and thus defines
a $p$-adic local Langlands correspondence; this amounts to the problem of
showing  that the
patched modules that we construct are purely local objects. For
some weak evidence in this direction, see~\cite{emertongeesavitt}, which proves a related result for lattices
corresponding to certain 2-dimensional tamely potentially Barsotti--Tate
representations. It can also be shown that our construction recovers the
known correspondence for $\GL_2(\Qp)$, without needing to use the full strength of the $p$-adic local Langlands correspondence. We are currently writing a paper which will explain this and for which the main steps are: showing that our module $M_\infty$ is projective as a $\GL_2(\mathbb{Q}_p)$-representation, computing its cosocle via the weight part of Serre's conjecture, and exploiting the density of crystabelline points in local deformation rings.
Finally, we refer to Section~\ref{sec:LL} below for a slightly more detailed discussion of how the construction
of this paper might relate to a hypothetical $p$-adic local Langlands correspondence in the general case of $\GL_n(F)$.

\subsection{Inertial local Langlands, the Bernstein centre, and local-global
compatibility}
In Sections~\ref{sec:types}
and~\ref{sec:local-global compatibility}
we relate the theory of the Bernstein centre and the so-called inertial local Langlands correspondence to the theory of potentiallly crystalline deformation
rings, and the Taylor--Wiles--Kisin patched modules which lie over them.

More precisely, in Section~\ref{sec:types} we synthesise and expand on results
of Bernstein and Bernstein--Zelevinsky, Bushnell--Kutzko, Schneider--Zink, and Dat, to draw the following conclusions: for any {\em inertial type} $\tau$
(i.e.\ a representation $\tau: I_F \to \GL_n(\Qbar_p)$ with open kernel
which extends to the Weil group of $F$),
there is an associated {\em smooth type} $\sigma(\tau)$,
which is a smooth $\Qpbar$-representation of $\GL_n(\cO_F)$,
with the following properties:

\begin{itemize}
\item[(i)] The Hecke algebra $\cH\bigl(\sigma(\tau)\bigr):=
\End\bigr(\cInd_{\GL_n(\cO_F)}
^{\GL_n(F)} \sigma(\tau)\bigr)$ is commutative.
\item[(ii)] $\sigma(\tau)$ appears as a $\GL_n(\cO_F)$-subrepresentation
of an irreducible smooth $\GL_n(F)$-representation $\pi$ if and only
if the Weil--Deligne representation attached to $\pi$ via the local
Langlands correspondence is isomorphic to $\tau$ when restricted
to~$I_F$, and in addition satisfies $N = 0$.  Furthermore, for such $\pi$,
the representation $\sigma(\tau)$ appears in $\pi$ with multiplicity
one.
\end{itemize}

Suppose that
$\pi$ is an irreducible smooth $\GL_n(F)$-representation whose
associated Weil--Deligne representation is isomorphic to $\tau$ when restricted
to~$I_F$, and in addition satisfies $N = 0$, so that $\sigma(\tau)$
appears with multiplicity one in $\pi$, by (ii) above.
By Frobenius reciprocity, the Hecke algebra $\cH\bigl(\sigma(\tau)\bigr)$
then acts on the associated one-dimensional multiplicity space via
a character $\chi_{\pi}: \cH\big(\sigma(\tau)\bigr) \to \Qbar_p,$
and hence there is an induced surjection
$\cInd_{\GL_n(\cO_F)}^{\GL_n(F)} \sigma(\tau) \otimes_{\cH(\sigma(\tau)
),\chi_{\pi}} \Qbar_p \to \pi.$   In this context, we establish one
further result.

\begin{itemize}
\item[(iii)]
If $\pi$ is generic then
the preceding surjection is an isomorphism.
\end{itemize}

From these results, we deduce that
the connected components of the $\Spec$ of the Bernstein centre for $\GL_n(F)$
are identified with the various $\Spec \cH\bigl(\sigma(\tau)\bigr)$,
as $\tau$ ranges over all (isomorphism classes of) inertial
types. Over any such component we have the universal $\GL_n(F)$-representation
$\cInd_{\GL_n(\cO_F)}^{\GL_n(F)}\sigma(\tau)$,
whose fibre over the point $\chi_{\pi}$
arising from a generic $\pi$ whose associated Weil--Deligne representation
satisfies $N = 0$ is isomorphic to the representation~$\pi$.

For both the comparison with Galois deformation rings that we make in
Section~\ref{sec:local-global compatibility}, and for the connections
that we draw with the theory of Banach space representations of $\GL_n(F)$,
it is technically important to work over a finite extension of $\Q_p$
rather than over $\Qbar_p$, so we also explain how to descend the preceding
results from $\Qbar_p$ to such finite extensions.

In Section~\ref{sec:local-global compatibility},
we consider so-called {\em locally algebraic types},
which are representations of $\GL_n(\cO_F)$, defined over some finite
extension $E$ of $\Qbar_p$,  of the form $\sigma_{\sm}\otimes\sigma_{\alg}$,
where $\sigma_{\sm}$ is the smooth type attached to some inertial type $\tau$,
and $\sigma_{\alg}$ is an irreducible algebraic representation of
$\Res_{F/\Q_p}\GL_n$.
Attached to $\sigma$ we have a Hecke algebra
$\cH(\sigma) := \End(\cInd_{\GL_n(\cO_F)}^{\GL_n(F)} \sigma);$
it is isomorphic to $\cH(\sigma_{\sm}).$
Attached to $\sigma$ and any continuous representation
$\rbar:G_F \to \GL_n(\F)$ as above, there is a universal lifting ring
$R_{\rbar}^{\square}(\sigma)$, parameterising potentially crystalline lifts of
$\rbar$ whose associated inertial type coincides with $\tau$,
and whose Hodge--Tate weights match with the highest weight of
$\sigma_{\alg}$ after applying the usual $\rho$-shift.

One of the main results of
Section~\ref{sec:local-global compatibility},
which may be of independent interest,
is the existence of a homomorphism
$\eta:\cH(\sigma) \to R_{\rbar}(\sigma)^{\square}[1/p]$
which interpolates the local Langlands correspondence. (This gives an
algebraic extension of an analogous rigid-analytic result proved in~\cite{chenevier}).  Namely,
if $x: R_{\rbar}(\sigma)^{\square}[1/p] \to \Qbar_p$ corresponds
to a crystalline lift $r_x$ of $\rbar$,
if $\pi$ is the irreducible smooth representation of $\GL_n(F)$ associated
to the Weil--Deligne representation underlying $r_x$ via the local Langlands
correspondence, and if $\chi_{\pi}: \cH(\sigma) \cong \cH(\sigma_{\sm})
\to \Qbar_p$ is the character of $\cH(\sigma)$ associated to $\pi$,
then we have the equality $\chi_{\pi} = x\circ \eta.$

The second main result of this section is a key reciprocity law
related to the Hecke action on locally algebraic vectors in $M_{\infty}$,
which we refer to as {\em local-global compatibility} (in analogy
with the classical local-global compatibility results for cohomology of
Shimura varieties~\cite{1202.4683}).
Given a locally algebraic type $\sigma$, we may form the $R_{\infty}$-module \[M_{\infty}(\sigma) :=
\Hom^{\cont}_{\GL_n(\cO_F)}(M_{\infty}, \sigma^*)^*,\] where $^*$ denotes continuous dual.
We show that the
$R^{\square}_{\rbar}$-action on this module
factors through $R^{\square}_{\rbar}(\sigma)$.
It is thus equipped with two natural $\cH(\sigma)$-actions: one via its
very definition together with Frobenius reciprocity (and so related
to the structure of $M_{\infty}$ as a $\GL_n(F)$-representation),
and one via the homomorphism $\eta$ (and so related to the structure
of $M_{\infty}$ as an $R^{\square}_{\rbar}$-module);
local-global compatibility is the statement that these two actions coincide.

This reciprocity law is crucial to our analysis of the locally algebraic
vectors in the representations $V(r)$, and to our study of the Breuil--Schneider
conjecture.

\subsection{The relationship with \cite{scholze}}
We briefly discuss how our work relates to some other recent progress
in the field.
In~\cite{scholze}, Scholze provides more evidence for the existence of a purely local $p$-adic local Langlands correspondence for $\GL_n(F)$, by studying the cohomology of the Lubin--Tate tower. In the classical case~\cite{ht}, when $l\not =p$, it is known that the Lubin--Tate tower simultaneously realizes the local Langlands correspondence and the Jacquet-Langlands correspondence, between representations of $\GL_n(F)$ and representations of $D^\times$, where $D/F$ is the central division algebra of invariant $1/n$. Scholze uses the Lubin--Tate tower to construct a purely local functor \[\pi\mapsto F(\pi)\] from admissible, smooth $\mathbb{F}_p$-representations of $\GL_n(F)$ to admissible representations of $D^\times$ equipped with an action of $G_F$. This functor goes in the opposite direction from our construction.

However, when $n=2$, Scholze proves that it is compatible with our
patching construction (Corollary 9.3 of~\cite{scholze}), in the
following sense. As in the case of unitary groups and $\GL_n(F)$, one
can patch the cohomology of locally symmetric spaces coming from a
quaternion algebra (which is split at $p$ and ramified at all infinite
places) to obtain a representation $\pi_\infty$ of $\GL_2(F)$. One can
also patch the cohomology of certain Shimura curves (corresponding to
a quaternion algebra which is ramified at $p$, but split at one
infinite place) to get a representation $\rho_\infty$ of
$D^\times\times G_F$. Then Scholze shows
that \[F(\pi_\infty)=\rho_\infty.\]
(In fact, Scholze employs a variant of the patching construction used in
this paper, making use of ultrafilters to reduce the amount of
bookkeeping needed to obtain the action of Hecke operators at~$p$.) We
remark that it should be possible to adapt his strategy to $p$-adic
representations and to general $n$ (using Shimura varieties of
Harris--Taylor type for the latter step). It seems reasonable to expect
that this will lead to a proof that one can recover the
$G_F$-representation $r$ from the Banach space $V(r)$ that we
associate to it (at least in cases where it can be shown that $V(r)$
is nonzero; for example, this will be the case for the representations
considered in Theorems~\ref{main thm on BS, abstract intro version}
and~\ref{main thm on BS, explicit intro version}, where we even prove
that the locally algebraic vectors in $V(r)$ are nonzero.)

\subsection{Outline of the paper}In Section~\ref{sec:patching}, we
carry out our patching construction. Section~\ref{sec:types} contains
an introduction to the results of Bernstein--Zelevinsky,
Bushnell--Kutzko, Schneider--Zink and Dat on types and the local
Langlands correspondence for~$\GL_n$. We then refine
some of these results, as described above,
and explain how to descend them from
algebraically closed coefficient fields to finite extensions of~$\Qp$.
In Section~\ref{sec:local-global compatibility} we begin by
establishing our interpolation of
the classical local Langlands correspondence
over a (potentially crystalline) local deformation
ring, and then apply this to establish local-global compatibility for our
patched modules.
Finally, in Section~\ref{sec:breuilschneider} we combine our local-global compatibility result with
automorphy lifting theorems to prove our results on the
Breuil--Schneider conjecture.

\subsection{Acknowledgements}\label{subsec:acknowledgements}
The constructions made in this paper originate in work
carried out at a focused research group on ``The $p$-adic Langlands program for
non-split groups'' at the Banff Centre in August 2012; we would like to thank
BIRS for providing an excellent working atmosphere, and for its financial
support. We would also like to thank AIM for providing financial support and an
excellent working atmosphere towards the end of this project. The idea that
patching could be used to prove the Breuil--Schneider conjecture goes back to
conversations between two of us (M.E.\ and T.G.) at the Harvard Eigen-semester
in 2006, and we would like to thank the mathematics department of Harvard
University for its hospitality. Finally, we thank Brian Conrad and
Florian Herzig for their helpful remarks on an earlier draft of this paper,
and the anonymous referees for their many helpful comments,
corrections, and questions.
\subsection{Notation}\label{subsec:notation}We fix a prime $p$, and an
algebraic closure $\Qpbar$ of $\Qp$. Throughout the paper we work with a
finite extension $E/\Qp$ in $\Qpbar$, which will be our coefficient field.
We write $\cO=\cO_E$ for the ring of integers in $E$,
$\varpi=\varpi_E$ for a uniformiser, and $\F:=\cO/\varpi$ for the
residue field.
At any
particular moment $E$ is fixed, but we allow ourselves to modify $E$
(typically via an extension of scalars) during the course of our arguments.
Furthermore, we will often assume without further comment
that $E$ and $\F$ are sufficiently large, and in particular that if we
are working with representations of the absolute Galois group of a $p$-adic
field $F$, then the images of all embeddings $F\into \overline{\Q}_p$ are
contained in $E$.

If $F$ is a field, we let $G_F$ denote its absolute Galois group. Let
$\varepsilon$ denote the $p$-adic cyclotomic character, and
$\varepsilonbar$ the mod $p$ cyclotomic character. If $F$ is a finite
extension of $\bb{Q}_p$ for some $p$, we write $I_F$ for the inertia
subgroup of $G_F$, and $\varpi_F$ for a uniformiser of the ring of
integers $\cO_F$ of $F$. If $\tF$ is a number field and $v$ is a
finite place of $\tF$ then we let $\Frob_v$ denote a geometric Frobenius
element of $G_{\tF_v}$.

If $F$ is a $p$-adic field, $W$ is a de Rham
representation of $G_F$ over $E$, and  $\kappa:F \into E$,
then we will write $\HT_\kappa(W)$ for the multiset of Hodge--Tate
weights of $W$ with respect to $\kappa$.  By definition, the multiset $\HT_\kappa(W)$ contains $i$ with multiplicity
$\dim_{E} (W \otimes_{\kappa,F} \widehat{\overline{F}}(i))^{G_F}
$. Thus for example $\HT_\kappa(\varepsilon)=\{ -1\}$.

We say that $W$ has \emph{regular} Hodge--Tate weights if for each
$\kappa$, the elements of $\HT_\kappa(W)$ are pairwise distinct. Let $\Z^n_+$
denote the set of tuples $(\xi_1,\dots,\xi_n)$ of integers
with $\xi_1\ge \xi_2\ge\dots\ge \xi_n$. Then if $W$ has
regular Hodge--Tate weights, there is a
$\xi=(\xi_{\kappa,i})\in(\Z^n_+)^{\Hom_\Qp(F,E)}$ such that for each $\kappa:F\into
E$, \[\HT_{\kappa}(W)=\{\xi_{\kappa,1}+n-1,\xi_{\kappa,2}+n-2,\dots,\xi_{\kappa,n}\},\]
and we say that \emph{$W$ is regular of weight $\xi$}. For any
$\xi\in\Z^n_+$, view $\xi$ as a dominant weight (with respect to the
upper triangular Borel subgroup) of the
algebraic group $\GL_{n}$ in the usual way, and let $M'_\xi$
be the algebraic $\cO_F$-representation of $\GL_n$ given
by \[M'_\xi:=\Ind_{B_n}^{\GL_n}(w_0\xi)_{/\cO_F}\] where $B_n$
is the Borel subgroup of upper-triangular matrices of $\GL_n$, and
$w_0$ is the longest element of the Weyl group (see \cite{MR2015057}
for more details of these notions, and note that $M'_\xi$ has
highest weight $\xi$). Write $M_\xi$ for the
$\cO_F$-representation of $\GL_n(\cO_F)$ obtained by evaluating
$M'_\xi$ on $\cO_F$. For any $\xi\in(\Z^n_+)^{\Hom_\Qp(F,E)}$
we write $L_\xi$ for the $\cO$-representation of $\GL_n(\cO_F)$
defined by \[L_{\xi}:=\otimes_{\kappa:F\into
  E}M_{\xi_\kappa}\otimes_{\cO_F,\kappa}\cO.\]

If $F$ is a $p$-adic field, then an \emph{inertial type} is   a representation
$\tau:I_F\to\GL_n(\Qpbar)$ with open kernel  which extends
to the Weil group $W_F$.

  Then we say that a de Rham representation
$\rho:G_F\to\GL_n(E)$ has inertial type $\tau$ if the
restriction to $I_F$ of the Weil--Deligne representation $\WD(\rho)$ associated to
$\rho$ is equivalent to $\tau$. Given an inertial type $\tau$, there
is a finite-dimensional smooth irreducible $\Qpbar$-representation
$\sigma(\tau)$ of $\GL_n(\cO_F)$ associated to $\tau$ by the
``inertial local Langlands correspondence''; see Theorem~\ref{thm:
  inertial local Langlands, N=0} below. (Note that by the results of
Section~\ref{god_save_the_queen} below, we will be able to replace
$\sigma(\tau)$ by a model defined over a finite extension of $E$ in
our main arguments.)

Let $F$ be a finite extension of $\Qp$, and let $\rec$ denote the local
Langlands correspondence from isomorphism classes of irreducible
smooth representations of $\GL_n(F)$ over $\C$ to isomorphism classes
of $n$-dimensional Frobenius semisimple Weil--Deligne representations
of $W_F$ as in the introduction to \cite{ht}. Fix once and for all an isomorphism $\imath:\Qpbar\xrightarrow{\sim}\C$. We define the
local Langlands correspondence $\rec_p$ over $\Qpbar$ by $\imath \circ
\rec_p = \rec \circ \imath$. This depends only on
$\imath^{-1}(\sqrt{p})$, and if we define
$r_p(\pi):=\rec_p(\pi\otimes|\det|^{(1-n)/2})$, then $r_p$ is
independent of the choice of $\imath$. Furthermore, if $V$ is a
Frobenius semisimple
Weil--Deligne representation of $W_F$ over $E$, then $r_p^{-1}(V)$ is
also defined over $E$ by \cite[Prop 3.2]{Clo90} and the fact that $r_p$ commutes
with automorphisms of $\C$. (The claims about the dependence of
$\rec_p$ and $r_p$ on the choice of $\imath$ follow from the main
theorem of~\cite{MR1228128}, together with a study of the behaviour of
 $\varepsilon$-factors under automorphisms of~$\C$. In the case $n=2$,
 this is explained in~\cite[\S35]{MR2234120}, and the same argument
 goes through in general, with the required input on
 $\varepsilon$-factors being provided by~\cite[Thm.\ 3.2]{MR1739725}.\footnote{Alternatively, perhaps more conceptually, the claims are also implied by the geometric realisation of $r_p$ (up to dualising $\pi$) for supercuspidal representations in the cohomology of the Lubin-Tate tower. See Lemma VII.1.6 and the definition of $\rec_l$ on page 237 of \cite{ht}.}) 

Recall that a linear-topological $\cO$-module is a topological
$\cO$-module (that is, it has a topology for which both addition and the action of $\cO$ are continuous) which also has a fundamental system of open neighborhoods of the identity which are $\cO$-submodules. If $A$ is a linear-topological $\cO$-module, we write $A^\vee$
for its Pontrjagin dual $\Hom_\cO^\cont(A,E/\cO)$, where $E/\cO$ has
the discrete topology, and we give $A^\vee$ the compact open topology.

By the proof of Theorem 1.2 of~\cite{MR1900706}, the functor given by $A\mapsto
A^d:=\Hom_\cO^\cont(A,\cO)$ induces an anti-equivalence of categories
between the category of compact, $\cO$-torsion-free linear-topological
$\cO$-modules $A$ and the category of $\varpi$-adically complete and
separated $\cO$-torsion-free $\cO$-modules. A quasi-inverse is given
by $B\mapsto B^d:=\Hom_\cO(B,\cO)$, where the target is given the weak
topology of pointwise convergence. We refer to this duality as
Schikhof duality. Note that if $A$ is an $\cO$-torsion free profinite linear-topological $\cO$-module, then $A^d$ is the unit ball in the $E$-Banach space $\Hom_\cO(A,E)$.

If $r:G_F\to\GL_n(E)$ is de Rham of regular weight $a$, then we
write $\pi_\alg(r):=L_a^d \otimes_{\cO}E$, and
$\pi_\sm(r):=r_p^{-1}(\WD(r)^{F-\semis})$, both of which are
$E$-representations of $\GL_n(F)$. (The $\GL_n(\cO_F)$-action on $L_a^d$ extends linearly to a $\GL_n(F)$-action on $\pi_\alg(r)$.)
As the names suggest, $\pi_\alg(r)$
is an algebraic representation, and $\pi_\sm(r)$ is a smooth
representation. Note that $\pi_\alg(r)=L_\xi\otimes_{\cO} E$ for $\xi_{\kappa,i}:=-a_{\kappa,n+1-i}$.

We let $\Art_F:F^\times\xrightarrow{\sim} W_F^{\ab}$ be the isomorphism
provided by local class field theory, which we normalise so that
uniformisers correspond to geometric Frobenius elements.

 We
write all matrix transposes on the left; so ${}^tg$ is the transpose
of $g$. We let $\cG_n$ denote the group scheme over $\Z$ defined to be the semidirect
product of $\GL_n\times\GL_1$ by the group $\{1,j\}$, which acts on
$\GL_n\times\GL_1$ by \[j(g,\mu)j^{-1}=(\mu\cdot{}^tg^{-1},\mu).\]We have a
homomorphism $\nu:\cG_n \to \GL_1,$ sending $(g,\mu)$ to $\mu$ and $j$
to $-1$.

Further notation is introduced in the course of our arguments; we mention
just some of it here, for the reader's convenience.

From Subsection~\ref{subsec:patching} on, we will have fixed a particular
finite extension $F$ of $\Q_p$, with ring of integers $\cO_F$ and uniformiser
$\unif_F$. To ease notation we will typically write
$G := \GL_n(F)$, $K := \GL_n(\cO_F)$,  and $Z := Z(G)$. For each $m \geq 0$,
we write $\Gamma_m = \GL_n(\cO_F/\varpi_F^m)$
and $K_m :=\ker\bigr(\GL_n(\cO_F)\to\GL_n(\cO_F/\varpi_F^m)\bigl)$,
so that $K/K_m\iso \Gamma_m$.

Furthermore, throughout Section~\ref{sec:patching}, a large amount of notation is introduced
related to automorphic forms on a definite unitary group, and Taylor--Wiles--Kisin
patching.  Here we merely signal that the main construction of this section,
and our major object of study in the paper, is a patched $G$-representation that we
denote by $M_{\infty}$.    Beginning in Section~\ref{sec:local-global compatibility},
we will also write
\[M_\infty(\sigma^\circ):=\left(\Hom^{\mathrm{cont}}_{\cO[[K]]}(M_\infty,(\sigma^\circ)^d)\right)^d,\]
when $\sigma^{\circ}$ is a $K$-invariant $\cO_E$-lattice in a
$K$-representation $\sigma$ of finite dimension over $E$.

We use $\Ind$ to denote induction, and $\cInd$ to denote induction
with compact supports.
In Sections~\ref{sec:types} and~\ref{sec:local-global compatibility},
we use $\nind_P^G$ to denote normalised parabolic induction.

If $\sigma$ is a representation of $K$, then we will write
$\cH(\sigma) := \End_G(\cInd_K^G \sigma)$ to denote the Hecke algebra of $G$ with respect
to $\sigma$. Sometimes, when it is helpful to emphasise the role of $G$,
we will write $\cH(G,\sigma)$ instead.  We also use obvious variants
with $G$ and $K$ replaced by another $p$-adic group and compact open subgroup.

\section{The patching argument}\label{sec:patching}

In this section we will carry out our patching argument on definite
unitary groups.  The key difference between the construction presented
here and previous patching constructions is that the object we end up
with is not simply a module over a certain Galois deformation ring,
but rather a $\GL_n(F)$-representation over that ring;
we refer to the discussion at the
beginning of Subsection~\ref{subsec:patching} below for a more detailed
account of this difference.

Our construction uses the
same general framework as that used
in section 5 of~\cite{emertongeerefinedBM} (which in
turn is based on the approach of \cite{cht}, \cite{blgg}
and~\cite{jack}); we recall the key elements of this framework
in the first several subsections that follow.
We follow the notation of~\cite{emertongeerefinedBM}
as closely as possible, and we indicate explicitly where we deviate
from it.

The construction itself is the subject of Subsection~\ref{subsec:patching},
and the key fact that it actually produces a $\GL_n(F)$-representation is
verified in Proposition~\ref{prop:Minfty is projective and has a G-action}.
In Subsection~\ref{subsec:admissible unitary Banach} we explain how
our patched representation of $\GL_n(F)$ gives rise to admissible unitary
Banach representations attached to local Galois representations.

\subsection{Globalisation}\label{subsec:globalization}Let $F/\Qp$ be a
finite extension, and fix a continuous representation
$\rbar:G_F\to\GL_n(\F)$.  Our goal in this subsection is to give
a criterion for $\rbar$ to be obtained as the restriction of
a global Galois representation that is automorphic, in a
suitable sense (and satisfies some additional convenient properties).

We will assume that the following
hypotheses are satisfied:
\begin{itemize}
\item $p\nmid 2n$, and
\item $\rbar$ admits a potentially crystalline lift of regular
  weight, which is potentially diagonalisable.\footnote{Recall that, as in~\cite{BLGGT}, a potentially crystalline representation $r$ of $G_F$ is \emph{potentially diagonalisable} if there exists a finite extension $F'/F$ such that $r|_{G_{F'}}$ is crystalline and lies on the same irreducible component of the universal crystalline lifting ring of $\rbar|_{G_{F'}}$ (with fixed Hodge--Tate weights) as a sum of characters lifting $\rbar|_{G_{F'}}$. } \end{itemize}
Conjecturally, the second hypothesis is always satisfied; this is
Conjecture A.3 of~\cite{emertongeerefinedBM}. In this direction, we
note the following result.
\begin{lem}
  \label{lem: existence of pot diagonal lift}After possibly making a
  finite extension of scalars, the second hypothesis is
  satisfied if either $n=2$ or $\rbar$ is semisimple.
\end{lem}
\begin{proof}
  If $n=2$, this is Remark A.4 of~\cite{emertongeerefinedBM}. If
  $\rbar$ is semisimple, then after extending scalars, we may write it as a sum of inductions
  of characters, and it is easy to see that by lifting these
  characters to crystalline characters, we can find a potentially
  crystalline lift which has regular weight, and is a sum of
  inductions of characters. Such a lift is obviously potentially
  diagonalisable (indeed, after restriction to some finite extension, it
  is a sum of crystalline characters).
\end{proof}
Having assumed these hypotheses, Corollary A.7
of~\cite{emertongeerefinedBM} (with $K$ our $F$, and $F$ our $\tF$) provides us with an imaginary CM field
$\tF$ with maximal totally real subfield $\tF^+$, and a continuous irreducible representation
$\rhobar:G_{\tF^+}\to\cG_n(\F)$ such that $\rhobar$ is a \emph{suitable
globalisation} of $\rbar$ in the sense of Section 5.1
of~\cite{emertongeerefinedBM}. Here we say that $\rhobar$ is irreducible if $\rhobar |_{G_{\tF}}$, which is
regarded as a representation valued in $\GL_n(\F)$, is irreducible.
 We recall the properties that $(\tF,\tF^+,\rhobar)$ need to satisfy for this definition:
\begin{itemize}
 \item each place $v\mid p$ of $\tF^+$ splits in $\tF$, and has $\tF^+_v\cong
    F$; we fix a choice of such isomorphisms.
 \end{itemize}
and
  \begin{itemize}
  \item $\rhobar$ is automorphic (see, for example, Definition 5.3.1 of~\cite{emertongeerefinedBM}) and unramified at primes $v\nmid p$.
  \item the inverse image of $\GL_n(\F)\times \GL_1(\F)$ under $\rhobar$ is $G_{\tF}$.
  \item $\rhobar(G_{\tF(\zeta_p)})$ is adequate in the sense of Definition 2.3 of~\cite{jack}.\footnote{We will not need the precise definition of an \emph{adequate subgroup} of $\GL_n(\mathbb{\bar F}_p)$; we will only need to know that this property is satisfied in order to apply the machinery developed in~\cite{jack}. See Section~\ref{subsec:aux primes} for more details.}
  \item For each place $v\mid p$ of $\tF^+$, there is a place $\tv$ of $\tF$
    lying over $v$ with $\rhobar|_{G_{\tF_\tv}}$ isomorphic to $\rbar$.
  \item $\overline{\tF}^{\ker\ad\rhobar|_{G_{\tF}}}$ does not contain
  $\tF(\zeta_p)$.
\end{itemize}
Corollary A.7 of~\cite{emertongeerefinedBM} guarantees that all these properties can be satisfied simultaneously. (The only difference is that the last property is replaced by the fact that $\overline{\tF}^{\ker \rhobar}$ does not contain $\tF(\zeta_p)$, which is stronger, since $\ker \rhobar\subset G_{\tF}$ by the second property above.) In order to arrange that our patched modules have a certain
multiplicity one property, we will also demand that:
\begin{itemize}
\item $\rhobar(G_{\tF})=\GL_n(\F')$ for some subfield $\F'\subseteq \F$
  with $\#\F'>3n$.
\end{itemize}
To see that we can arrange this, note that the proof of Proposition A.2
of~\cite{emertongeerefinedBM} (which is the main input to Corollary
A.7 of \emph{op.\ cit.}, together with the potential automorphy results of~\cite{BLGGT}) allows us to arrange that $\rhobar(G_{\tF})=\GL_n(\F_{p^m})$ for any sufficiently large $m$.

Finally, after making a solvable base change, we can and do assume that $\tF/\tF^+$ is
unramified at all finite places.

\subsection{Unitary groups}\label{subsec:unitary groups}
We now use the globalisation $\rhobar$ of our local Galois
representation $\rbar$ to carry out the
Taylor--Wiles--Kisin patching argument as in Section~5
of~\cite{emertongeerefinedBM}. The definitions of Hecke algebras, the
choices of auxiliary primes and so on are essentially identical to the
arguments made in \cite{emertongeerefinedBM}, and rather than
repeating them verbatim, we often refer the reader to
\cite{emertongeerefinedBM} for the details of these definitions,
indicating only the differences in our
construction. 

  As in Sections 5.2 and 5.3 of~\cite{emertongeerefinedBM}, we fix a certain
  definite unitary group $\tG/\tF^+$ together with a model (which we will also
  denote by $\tG$) over $\cO_{\tF^+}$. (The group $\tG$ is denoted $G$ in~\cite{emertongeerefinedBM}, but
  we will later use $G$ to denote $\GL_n(F)$.) This model has the property that for each place $v$ of $\tF^+$ which
  splits as $ww^c$ in $\tF$, there is an isomorphism
$\iota_w:\tG(\cO_{\tF^+_v})\isoto\GL_n(\cO_{\tF_w})$; we fix a choice of such
isomorphisms.  We also choose a finite
place $v_1$ of $\tF^+$ which is prime to $p$, with the properties that
  \begin{itemize}
  \item $v_1$ splits in $\tF$, say as $v_1=\tv_1\tv_1^c$,
  \item $v_1$ does not split completely in $\tF(\zeta_p)$, and
  \item $\rhobar(\Frob_{\tF_{\tv_1}})$ has distinct $\F$-rational eigenvalues, no two of
    which have ratio $(\mathbf{N}v_1)^{\pm 1}$.
  \end{itemize}
  (It is possible to find
  such a place $v_1$ by the Cebotarev density theorem, and our assumptions that
  $\rhobar(G_{\tF})=\GL_n(\F')$ with $\#\F'>3n$, and that
  $\overline{\tF}^{\ker\ad\rhobar|_{G_{\tF}}}$ does not contain
  $\tF(\zeta_p)$. Indeed, choosing a conjugacy class
  in~$\Gal(\tF(\zeta_p)/\tF)$, we have a positive density set of
  places~$v_1$ satisfying the first two conditions, and with
  $\mathbf{N}v_1$ taking a fixed value $\lambda$ modulo $p$; if we then choose
  a diagonal matrix in $\GL_n(\F')$ with distinct diagonal entries,
  none of whose ratios are $\lambda^{\pm 1}$, then another
  application of the Cebotarev density theorem produces the required
  place~$v_1$.

 Note that this differs slightly from the choice of place $v_1$ in
  the first paragraph of Section 5.3 of~\cite{emertongeerefinedBM}, where the third condition is replaced by the requirement that $\mathrm{ad}\ \rhobar|(\Frob_{\tF_{\tv_1}})=1$.
  However, it is still the case that any deformation of
  $\rhobar|_{G_{\tF_{\tv_1}}}$ is unramified (see Lemma~\ref{lem: at v1, we have smooth unramified
    lifts} below).  We have made this choice in order to be able to
  arrange that our patched modules satisfy multiplicity one.

Let $S_p$ denote the set of primes of $\tF^+$ dividing $p$. We now fix a place $\p\mid p$ of $\tF^+$,
and for each
  integer $m\ge 0$ we consider the compact open subgroup
  $U_m=\prod_vU_{m,v}$ of $\tG(\A^\infty_{\tF^+})$, where
\begin{itemize}
\item $U_{m,v}= \tG(\bigO_{\tF^+_v})$ for all $v$ which split in $\tF$ other
  than $v_1$ and $\p$;

   \item  $U_{m,v_1}$ is  the preimage of the upper triangular  matrices under
\[\tG(\cO_{\tF^+_{v_1}})\to \tG(k_{v_1}) \underset{\bar{\iota}_{\tv_1}}\iso \GL_n(k_{v_1})\]
\item $U_{m,\p}$ is  the kernel of the map
$\tG\bigl(\cO_{\tF^+_{\p}})\to \tG(\cO_{\tF^+_\p}/\varpi_{\tF^+_\p}^m\bigr); $
  \item $U_{m,v}$ is a hyperspecial maximal compact subgroup of $\tG(\tF_v^+)$
    if $v$ is inert in~$\tF$.
  \end{itemize}By the choice of $v_1$ and $U_{m,v_1}$ we see that
  $U_m$ is sufficiently small (in the sense of Section 5.2 of~\cite{emertongeerefinedBM}). Write $U:=U_0$. In order to make the patching
  argument, we will need to consider certain compact open subgroups of
  the $U_m$ corresponding to choices of sets of auxiliary primes
  $Q_N$ that will be introduced in Section~\ref{subsec:aux primes}. Specifically, for each integer
  $N\geq 1$, we will have a finite set of primes $Q_N$ of $\tF^+$
  disjoint from $S_p \cup \{v_1\}$ as
  well as open compact subgroups $U_i(Q_N)_v$ of $\tG(\cO_{\tF^+_v})$
  for each $v \in Q_N$ and $i=0,1$. We then define subgroups $U_i(Q_N)_m =
  \prod_v U_i(Q_N)_{m,v}\subset
  U_m$, for $i=0,1$ by setting $U_0(Q_N)_{m,v}=U_1(Q_N)_{m,v}=U_{m,v}$
  for $v\not \in Q_N$, and $U_i(Q_N)_{m,v}=U_i(Q_N)_v$ for $v\in
  Q_N$. 

  By assumption, $\rbar$ has a potentially diagonalisable lift of
  regular weight, say $r_{\textrm{pot.diag}}:G_F\to\GL_n(\cO)$. Suppose that $r_{\textrm{pot.diag}}$ has
  weight $\xi$ and inertial type $\tau$ (in the sense of
  Section~\ref{subsec:notation}). Extending $E$ if necessary, we may
  assume that the $\GL_n(\cO_F)$-representation $\sigma(\tau)$ is defined over $E$. Then we have two
  representations $L_\xi$ and $L_{\tau^\vee}$ of $\GL_n(\cO_F)$ on finite
  free $\cO$-modules in the following way: the representation
  $L_\xi$ is the one defined in the notation section, and $L_{\tau^\vee}$
  is a choice of $\GL_n(\cO_F)$-stable lattice in
  $\sigma(\tau)^\vee$. Set $L_{\xi,\tau}:=L_{\tau^\vee}\otimes_\cO
  L_\xi$, a finite free $\cO$-module with an action of
  $\GL_n(\cO_F)$.

  Returning to our
  global situation, let $W_{\xi,\tau}$ denote the finite free $\cO$-module with an
  action of $\prod_{v\in S_p\setminus\{\p\}}U_{m,v}$ given by
  $W_{\xi,\tau}=\otimes_{v\in S_p\setminus\{\p\},\cO} L_{\xi,\tau}$
  where $U_{m,v}$ acts on the factor corresponding to $v$ via
  $U_{m,v}=\tG(\cO_{\tF_v^+})\underset{\iota_{\tv}}\iso\GL_n(\cO_{\tF_{\tv}})\isoto\GL_n(\cO_F)$. In
  order to avoid duplication of definition, we allow $Q_N=\emptyset$
  in the definitions we now make. For any finite $\cO$-module $V$ with
  a continuous action of $U_{m,\p}$, we have spaces of algebraic
  modular forms $S_{\xi,\tau}(U_i(Q_N)_m,V)$; these are just the
  functions \[f:\tG(\tF^+)\backslash \tG(\A_{\tF^+}^\infty)\to
  W_{\xi,\tau}\otimes_{\cO} V\]with the property that if $g\in
  \tG(\A_{\tF^+}^\infty)$ and $u\in U_i(Q_N)_m$ then
  $f(gu)=u^{-1}f(g)$, where $U_i(Q_N)_m$ acts on $W_{\xi,\tau}\otimes_{\cO}
  V$ via projection to $\prod_{v\in S_p}U_{m,v}$. (For example: when $V=\cO$ is the
  trivial representation, then after extending scalars from $\cO$ to
  $\C$ via $\cO \subset \Qpbar \underset{\imath}{\iso} \C$, this space corresponds to
  classical automorphic forms of fixed type
  $\sigma(\tau)$ at the places in $S_p\setminus\{\p\}$, full level
  $\p^m$ at $\p$, and whose weight (via our fixed isomorphism
  $\imath:\Qpbar\to\C$) is $0$ at places above $\p$, and given by
  $\xi$ at each of the places in $S_p\setminus\{\p\}$.)

  We let
  $\mathbb{T}^{S_p\cup Q_N,\univ}$ be the commutative
  $\bigO$-polynomial algebra generated by formal variables $T_w^{(j)}$
  for all $1\le j\le n$, $w$ a place of $\tF$ lying over a place $v$
  of $\tF^+$ which splits in $\tF$ and is not contained in $S_p\cup
  Q_N\cup\{v_1\}$, together with formal variables $T_{\tv_1}^{(j)}$
  for $1\le j\le n$. The algebra $\mathbb{T}^{S_p\cup Q_N,\univ}$ acts on
  $S_{\xi,\tau}(U_i(Q_N)_m,V)$ via the Hecke operators
  \[ T_{w}^{(j)}:=   \left[ U_{m,w}\ \iota_{w}^{-1}\left( \begin{matrix}
      \varpi_{w}1_j & 0 \cr 0 & 1_{n-j} \end{matrix} \right)
U_{m,w} \right]
\] where $\varpi_w$ is a fixed uniformiser in
$\mc{O}_{\tF_w}$. 

Choose an ordering
$\delta_1,\dots,\delta_n$ of the (distinct) eigenvalues of
$\rhobar(\Frob_{\tv_1})$. Since $\rhobar$ is a suitable globalisation of $\rbar$, it is in
particular automorphic in the sense of Definition 5.3.1
of~\cite{emertongeerefinedBM}, and we let $\m_{Q_N}$ be the maximal
ideal of $\mathbb{T}^{S_p\cup Q_N,\univ}$ corresponding to $\rhobar$,
and containing each of the elements
$$\bigl(T_{{\tv_1}}^{(j)}-(\mathbf{N}v_1)^{j(1-j)/2}(\delta_1\cdots\delta_j)\bigr),$$ for $1\le j\le n$. We will write $\m$ for $\m_\emptyset$.

\subsection{Galois deformations}\label{subsec:galois deformations}Let $S$ be a set of places of $\tF^+$ which split in $\tF$, with
 $S_p\subseteq S$. As in
\cite{cht}, we will write $\tF(S)$ for the maximal extension of $\tF$
unramified outside $S$, and from now on we will write
$G_{\tF^+,S}$ for $\Gal(\tF(S)/\tF^+)$. We will freely make use of the terminology (of
liftings, framed liftings etc.) of Section 2 of
\cite{cht}.

Let $T=S_p\cup \{v_1\}$. For each $v\in S_p$, we let $\tv$ be a choice of a place of $\tF$ lying
over $v$, with the property that
$\rhobar|_{G_{\tF_\tv}}\cong\rbar$. (Such a choice is possible by our
assumption that $\rhobar$ is a suitable globalisation of $\rbar$.) We let $\tT$ denote the set of places $\tv$, $v\in T$. For
each $v\in T$, we let $R_\tv^\square$ denote the maximal reduced and
$p$-torsion free quotient of the universal $\cO$-lifting ring of
$\rhobar|_{G_{\tF_\tv}}$. For each $v\in S_p\setminus\{\p\}$, we write
$R_\tv^{\square,\xi,\tau}$ for the reduced and $p$-torsion free
quotient of $R_\tv^\square$ corresponding to potentially crystalline
lifts of weight $\xi$ and inertial type $\tau$. (Such a quotient of $R_\tv^{\square}$ exists by Corollary 2.7.7 of~\cite{kisindefrings}. It has the property that for any $E$-algebra $A$, an $E$-algebra map $R_\tv^{\square}[1/p]\to A$ factors through $R_\tv^{\square,\xi,\tau}$ if and only if the pullback of the universal lifting along this map is potentially crystalline of weight $\xi$ and inertial type $\tau$.)

Consider (in the terminology of \cite{cht}) the deformation problem
\[\cS:=\big(\tF/\tF^+,T,\tT,\cO,\rhobar,\varepsilon^{1-n}\delta_{\tF/\tF^+}^n,\{R_{\tv_1}^\square\}\cup\{R_{\tp}^\square\}\cup\{R_\tv^{\square,\xi,\tau}\}_{v\in
  S_p\setminus\{\p\}}\bigr).\]
There is a corresponding universal
deformation $\rho_\cS^{\univ}:G_{\tF^+,T}\to\cG_n(R_\cS^{\univ})$ of~$\rhobar$. In addition, there is a universal $T$-framed deformation ring
$R_\cS^{\square_T}$ in the sense of Proposition 2.2.9 of~\cite{cht},
which parameterises deformations of $\rhobar$ of type $\cS$ together
with particular local liftings for each $\tv\in\tT$.
\begin{lem}
  \label{lem: at v1, we have smooth unramified
    lifts}$R_{\tv_1}^\square$ is formally smooth over~$\cO$, and all
  of the corresponding Galois representations are unramified.
\end{lem}
\begin{proof}
  By our assumptions on~$v_1$, this is immediate from Lemma 2.4.9 and Corollary 2.4.21 of~\cite{cht}.
\end{proof}
\subsection{Auxiliary primes}\label{subsec:aux primes}
Recall that the globalisation $\rhobar$ constructed in Section~\ref{subsec:globalization} satisfies the property that $\rhobar(G_{\tF(\zeta_p)})$ is adequate. This property is needed in order to apply
the version of the Taylor--Wiles patching argument given in~\cite{jack} (see also Section 5.5 of~\cite{emertongeerefinedBM}). More precisely, Proposition 4.4 of~\cite{jack} allows us to  choose an integer $q\ge[\tF^+:\Q]n(n-1)/2$ and for each $N\ge 1$ sets
  of primes $Q_N,\tQ_N$ with the following properties (as well as a crucial property about the
  generation of global Galois deformation rings over local ones that
  we will recall below): \begin{itemize}
\item $Q_N$ is a finite set of finite places of $\tF^+$ of cardinality $q$ which is disjoint
  from $T$ and consists of places which split in $\tF$;
\item $\tQ_N$ consists of a single place $\tv$ of $\tF$ above each place
  $v$ of $Q_N$;
\item $\mathbf{N}v \equiv 1 \mod p^N$ for $v \in Q_N$;
 \item for each $v\in Q_N$,
   $\rhobar|_{G_{\tF_\tv}}\cong\sbar_\tv\oplus\psibar_\tv$ where $\psibar_\tv$ is
   an eigenspace of Frobenius
on which Frobenius acts semisimply.
\end{itemize}
We remark that, for each $v$, any (generalized) eigenspace $\psibar_{\tv}$ of Frobenius on which the action is in fact semisimple can be chosen. Part of the content of Proposition 4.4 of~\cite{jack} is that the adequacy of $\rhobar(G_{\tF(\zeta_p)})$ implies the existence of such an eigenspace for appropriately chosen primes $v$.

For each $N$, $v\in Q_N$ and $i=0,1$, we let $U_i(Q_N)_v\subset
\tG(\cO_{\tF^+_v})$ denote the parahoric open
compact subgroups defined in Section 5.5
of~\cite{emertongeerefinedBM}, following~\cite{jack}. (We briefly recall their definition here: they are the inverse images under $\iota_{\tv}$ of certain parahoric subgroups $\mathfrak{p}^{\tv}_N,\mathfrak{p}^{\tv}_{N,1}$ of $\GL_n(\cO_{\tF_{\tv}})$. These parahoric subgroups correspond to the partition $n=(n-d^{\tv}_N)+d^{\tv}_N$, where $d^{\tv}_N$ is the dimension of the eigenspace $\psibar_{\tv}$: $\mathfrak{p}^{\tv}_N$ is the standard parahoric and $\mathfrak{p}^{\tv}_{N,1}$ is the kernel of the map \[\mathfrak{p}^{\tv}_N\to \GL_{d^{\tv}_N}(k_{\tv})\to k_{\tv}^\times\to k^\times_{\tv}(p).\] The first map in the sequence is given by projection to the $d^\tv_N$-block and reduction to $k_{\tv}$, the second map is taking the determinant and the last one is projection onto the maximal $p$-power order quotient of $k_{\tv}^\times$. These parahoric subgroups are roughly supposed to be analogous to levels $\Gamma_0(\tv),\Gamma_1(\tv)$ in the case of modular curves and modular forms.)

For each $v\in Q_N$, a quotient $R^{\psibar_\tv}_{\tv}$ of $R^\square_{\tv}$
is defined in Section 5.5 of~\cite{emertongeerefinedBM}
(following~\cite{jack}). We let $\cS_{Q_N}$ denote the
deformation problem
\begin{multline*}
\cS_{Q_N} := \big(\tF/\tF^+,\, T\cup Q_N,\, \tT\cup
\tQ_N,\, \cO,\, \rhobar,\, \varepsilon^{1-n}\delta_{\tF/\tF^+}^n, \\
\{R_{\tv_1}^\square\}\cup\{R_{\tp}^\square\}\cup\{R_\tv^{\square,\xi,\tau}\}_{v\in S_p\setminus\{\p\}}
 \cup\{R_{\tv}^{\psibar_\tv}\}_{v\in Q_N}\bigr).
\end{multline*}
We let $R_{\cS_{Q_N}}^{\univ}$
denote the corresponding universal deformation ring, and we let
$R^{\square_T}_{\cS_{Q_N}}$ denote the corresponding universal $T$-framed deformation
ring.
We define \[ R^{\loc} := R^\square_{\tp}\widehat{\otimes}\left(\widehat{\otimes}_{v \in
  S_p\setminus\{\p\}}R_\tv^{\square,\xi,\tau}\right)\widehat{\otimes}{R}^{\square}_{\tv_1} \]
where all completed tensor products are taken over $\mc{O}$. By the choice of the sets of primes $Q_N$, we also know that
\begin{itemize}
\item the ring $R^{\square_T}_{\mc{S}_{Q_N}}$ can be topologically
  generated over $R^{\loc}$ by\\ $q-[\tF^+:\Q]n(n-1)/2$ elements.
\end{itemize}

For each $v\in Q_N$ we choose a uniformiser
$\varpi_{\tv}\in\cO_{\tF_{\tv}}$, so that we have the projection operator
$\pr_{\varpi_{\tv}}\in\End_{\cO}(S_{\xi,\tau}(U_i(Q_N)_m,\cO/\varpi^r)_{\m_{Q_N}})$
defined as in Proposition 5.9
of~\cite{jack}. We briefly recall that, at level $U_0(Q_N)_v$,
$\pr_{\varpi_{\tv}}$ is defined as (the pullback along $\iota_{\tv}$
of) a polynomial in the Hecke operators corresponding to the block
$\GL_{d^\tv_N}$ inside the parahoric subgroup
$\mathfrak{p}^{\tv}_{N}\subset \GL_n(\cO_{\tF_\tv})$. The same formula
also gives rise to an element of the Hecke algebra for
$\mathfrak{p}^\tv_{N,1}\subset \GL_n(\cO_{\tF_\tv})$. This gives a
compatible projection operator at level $U_1(Q_N)_v$, which we will
also call $\pr_{\varpi_{\tv}}$ by abuse of notation. In the case of
level $U_0(Q_N)_v$, $\pr_{\varpi_{\tv}}$ is the projection onto a
one-dimensional subspace, which is identified with the spherical
vector, according to Proposition 5.9 of \emph{op.\ cit.} At level
$U_1(Q_N)_v$, $\pr_{\varpi_\tv}$ is best understood in terms of the
associated Galois representation: it only allows tamely ramified
deformations of the subrepresentation of $\rhobar$ corresponding to
$\psibar_\tv$. See Proposition 5.12 of \emph{op.\ cit.} for more details.

We define $\pr$ to be the composite of the projections $\pr_{\varpi_{\tv}}$.
(These projections commute among themselves, and so it doesn't matter in
which order we compose them. Whenever we use $\pr$ it will be clear
from the context what the underlying set $Q_N$ is.) Then, as in
Section 5.5 of \cite{emertongeerefinedBM}:
\begin{enumerate}
\item\label{gamma-to-gamma0} The map
\[ \pr:  S_{\xi,\tau}(U_m, \cO/\varpi^r)_\m \to
\pr\left(S_{\xi,\tau}(U_0(Q_N)_m,\cO/\varpi^r)_{\m_{Q_N}}\right) \]
is an isomorphism. Moreover, since $\pr$ is defined using Hecke operators at
places in $Q_N$, it commutes with the action of $\tG(\tF^+_{\gp})$ on the spaces
of algebraic automorphic forms. More precisely, if $g\in \tG(\tF^+_{\gp})$
satisfies $g^{-1}U_{m',\gp}g\subseteq U_{m,\gp}$ for some positive integers
$m\leq m'$, then we have a commutative diagram
\[
\begin{CD}
 S_{\xi,\tau}(U_m, \cO/\varpi^r)_\m
  @>\pr >>
\pr\left(S_{\xi,\tau}(U_0(Q_N)_m,\cO/\varpi^r)_{\m_{Q_N}}\right)
  \\
@VV g V  @VV g V \\
S_{\xi,\tau}(U_{m'}, \cO/\varpi^r)_\m
  @>\pr >>\pr\left(S_{\xi,\tau}(U_0(Q_N)_{m'},\cO/\varpi^r)_{\m_{Q_N}}\right)
\end{CD}
\]
where both vertical arrows are induced by the action of $g\in
\tG(\tF^+_{\gp})$ (namely: $(gf)(g') = f(g'g)$).
\item\label{projectivity} Let
$$\Gamma_m = \GL_n(\cO_F/\varpi_F^m) \cong U/U_m$$ and
  $$\Delta_{Q_N} = \prod_{v\in Q_N}U_0(Q_N)_v/U_1(Q_N)_v.$$ Then
  $U_0(Q_N)_0$
  acts on
  $S_{\xi,\tau}(U_1(Q_N)_m,\cO/\varpi^r)$ via
  $(gf)(g')=g_pf(g'g)$, and this action factors through $\Delta_{Q_N}\times \Gamma_m$.
  With respect to this action,
  $ \pr \left(S_{\xi,\tau}(U_1(Q_N)_m,\cO/\varpi^r)_{\m_{Q_N}}\right)$
  is
  a projective $(\cO/\varpi^r)[\Delta_{Q_N}][\Gamma_m]$-module,
  and there is a natural $\Gamma_m$-equivariant isomorphism
\[
\pr\left(S_{\xi,\tau}(U_1(Q_N)_m,\cO/\varpi^r)_{\m_{Q_N}}\right)^{\Delta_{Q_N}} \iso S_{\xi,\tau}(U_m,\cO/\varpi^r)_{\m}.\]
(The projectivity follows from the proof of Lemma 3.3.1
of~\cite{cht}. The fact that there is a $\Gamma_m$-equivariant isomorphism follows immediately from
point~\eqref{gamma-to-gamma0} and the definitions. We shall not need
the analogous statement about coinvariants which is recalled in
\cite{emertongeerefinedBM} and proved in \cite{jack}; see Remark~\ref{rem:duality explanation} below
for an indication as to why not. We remark that, by the explicit construction of $U_i(Q_N)_v$ above, $\Delta_{Q_N}$ is a finite abelian group of $p$-power order.)
  \item\label{gal-reps} Let $\mathbb{T}^{S_p\cup
    Q_N}_{\xi,\tau}(U_i(Q_N)_m,\cO/\varpi^r)$ be the image of
$\T^{S_p\cup Q_N,\univ}$ in the ring \\ $\End_\cO\left(\pr\left(
S_{\xi,\tau}(U_i(Q_N)_m,\cO/\varpi^r)\right)\right)$. As in Proposition 5.3.2 of~\cite{emertongeerefinedBM}, there exists a
deformation
\[ G_{\tF^+,T\cup Q_N} \to \cG_n\left(\mathbb{T}^{S_p\cup
    Q_N}_{\xi,\tau}(U_1(Q_N)_m,\cO/\varpi^r)_{\m_{Q_N}}\right) \]
of $\rhobar$ which is of type $\cS_{Q_N}$. In particular, $\pr
\left(S_{\xi,\tau}(U_i(Q_N)_m,\cO/\varpi^r)_{\m_{Q_N}}\right)$ is a
finite $R^{\univ}_{\cS_{Q_N}}$-module.
\end{enumerate}

\begin{remark}
The construction of the above deformation of $\rhobar$ follows the outline of the proof of Proposition 3.4.4 of~\cite{cht}, but one can appeal to Corollaire 5.3 of~\cite{labesse} for the necessary base change results and to the main results of~\cite{caraiani, 1202.4683} for local-global compatibility in the conjugate-self-dual case, which will show that the deformation is of type $\cS_{Q_N}$. (For the places in $S_p\setminus\{\mathfrak{p}\}$, the argument is similar to the one in the proof of Lemma~\ref{lem:classical lg compatibility} (1), which works at the place $\mathfrak{p}$. In particular, the argument relies on Theorem~\ref{thm: inertial local Langlands, N=0}.)
\end{remark}

As in Section 5.5 of~\cite{emertongeerefinedBM}, there is a
homomorphism $\Delta_{Q_N}\to(R_{\cS_{Q_N}}^{\univ})^\times$ obtained
by identifying $\Delta_{Q_N}$ with the product of the inertia
subgroups in the maximal abelian $p$-power order quotient of
$\prod_{v\in Q_N}G_{\tF_{\tv}}$, and thus a homomorphism
$\cO[\Delta_{Q_N}]\to R_{\cS_{Q_N}}^{\univ}$. The
$R_{\cS_{Q_N}}^{\univ}$-module structure coming from the existence of Galois representations thus induces an action of $\cO[\Delta_{Q_N}]$
on $\pr
\left(S_{\xi,\tau}(U_1(Q_N)_m,\cO/\varpi^r)_{\m_{Q_N}}\right)$,
which agrees with
the one in~\eqref{projectivity} above.

\subsection{Patching}\label{subsec:patching}
We now make our patching construction, by applying the Taylor--Wiles--Kisin method.
Before doing so, we provide a brief comparison and contrast with the patching
constructions in some previous papers, such as \cite{emertongeerefinedBM}
and~\cite{emertongeesavitt}.  In the latter paper, we employ Taylor--Wiles--Kisin
patching to construct what we call {\em patching functors}, which are (essentially)
certain exact functors
from the category of continuous $\GL_n(\cO_F)$-representations on finitely generated $\cO$-modules
to the category of coherent sheaves on an appropriate deformation space
of local Galois representations (perhaps with some auxiliary patching variables added).
Although this is not discussed in~\cite{emertongeesavitt},
such a functor can be (pro-)represented by an object $M_{\infty}$, which is
a continuous $\GL_n(\cO_F)$-representation over the local deformation ring
(again, perhaps with patching variables added).  More precisely, in terms
of such a $\GL_n(\cO_F)$-representation $M_{\infty}$, the patching functor can
be defined as $\Hom^{\cont}_{\cO[[\GL_n(\cO_F)]]}(M_{\infty},V^{\vee})^{\vee},$
if $V$ is a continuous representation of $\GL_n(\cO_F)$ on a finitely generated
$\cO$-module.  The exactness of the patching functor can be encoded in the requirement
that $M_{\infty}$ be a projective $\cO[[\GL_n(\cO_F)]]$-module.

In this paper our approach is to construct the representing object $M_{\infty}$ directly,
and (most importantly) to promote it from being merely a $\GL_n(\cO_F)$-representation
to being a representation of the full $p$-adic group $\GL_n(F)$.  (In terms of patching
functors, one can somewhat loosely
think of this as extending the patching functor from the category
of $\GL_n(\cO_F)$-representations to a category that we might call the {\em Hecke category},
whose objects are the same, but in which the morphisms between
any two $\GL_n(\cO_F)$-representations $U$ and $V$ are defined to be
$\Hom_{\GL_n(F)}\bigl(\cInd^{\GL_n(F)}_{\GL_n(\cO_F)} U, \cInd_{\GL_n(\cO_F)}^{\GL_n(F)} V\bigr).$)

Obtaining this additional structure on $M_{\infty}$ requires us to
keep track of additional data (``partial actions'' of the non-compact
directions in $G$) in the course of the patching process.
Before presenting the details of our construction,
we remark that Scholze
has simplified this aspect of our construction
by a reinterpretation of patching in terms of ultraproducts
\cite[\S\S 8, 9]{scholze},
which obviates the need for keeping track of this extra data.
We have chosen to keep
the original form of our argument here, however.

From now on, to ease notation we write $K=\GL_n(\cO_{\tF_{\tp}})=\GL_n(\cO_F)$,
$G=\GL_n(\tF_{\tp})=\GL_n(F)$, and $Z=Z(G)$. For each integer $N\ge 0$, we set
$K_N :=\ker\bigl(\GL_n(\cO_F)\to\GL_n(\cO_F/\varpi_F^N)\bigr)$, so that
$K/K_N\iso \Gamma_N$. We have the Cartan decomposition $G=KAK$, where $A$
is the set of diagonal matrices whose diagonal entries are powers of
the uniformiser $\varpi_F$, and we let $A_N$ be the subset of $A$
consisting of matrices with the property that the ratio of any two
diagonal entries is of the form $\varpi_F^r$ with $|r|\le N$,
and set $G_N=KA_NK$.  Note that $G_N$ is not a subgroup of $G$ unless
$N=0$, but that each $K\backslash G_N/KZ$ is finite, and $G=\cup_{N\ge
  0}G_N$.

If $(\sigma,W)$ is a representation of $KZ$, then we write
$\Ind_{KZ}^{G_N}\sigma$ for the space of functions $f:G_N\to W$ with
$f(kg)=\sigma(k)f(g)$ for all $g\in G_N$, $k\in KZ$; this is naturally
a $KZ$-representation via $(kf)(g):=f(gk)$. We define
$\Ind_{KZ}^G\sigma$ in the same way; then $\Ind_{KZ}^G\sigma$ is a representation
of $G$ via $(gf)(g'):=f(g'g)$.

 For each $N$, we set \[M_{i,Q_N}:= \pr\bigl( S_{\xi,\tau}(U_i(Q_N)_{2N},\cO/\varpi^N)_{\m_{Q_N}}\bigr)^\vee.\] Note that $M_{i,Q_N}$ depends on the integer $N$ as well as on the set of primes $Q_N$ (it could happen that $Q_M=Q_N$ for $M\neq N$), but we will only include $Q_N$ in the notation for the sake of simplicity. Note also that we could have equivalently defined \[M_{i,Q_N}:= \pr^\vee\bigl( S_{\xi,\tau}(U_i(Q_N)_{2N},\cO/\varpi^N)^\vee_{\m_{Q_N}}\bigr),\] since $\pr$ is an endomorphism of $S_{\xi,\tau}(U_i(Q_N)_{2N},\cO/\varpi^N)_{\m_{Q_N}}$ and Pontrjagin duality is an exact contravariant functor.

Let $\Delta_{Q_N}$ be as above; it is of $p$-power order by the
definitions of the $U_i(Q_N)_{2N}$.
It follows from point~\eqref{projectivity} in the previous section that $M_{1,Q_N}$
is a finite projective
$(\cO/\varpi^N)[\Delta_{Q_N}][\Gamma_{2N}]$-module.
Since
$Z$ centralises $U_1(Q_N)_{2N}$, there is also a natural action of $Z$
on $M_{1,Q_N}$.

\begin{remark}
\label{rem:duality explanation}
The reason for including a Pontrjagin dual in the definition
of $M_{i,Q_N}$ is that
$S_{\xi,\tau}(U_i(Q_N)_{2N},\cO/\varpi^N)$
is a space of automorphic forms,
and so is most naturally thought of as being contravariant
in the level, while patching is a process that involves passing to
a projective limit over the level (rather than a direct limit).
Now since
$S_{\xi,\tau}(U_i(Q_N)_{2N},\cO/\varpi^N)$
is a space of automorphic forms on the definite unitary group $\tG$,
it is a space of functions on a finite set, and so has a natural self-duality.
Thus, by exploiting this self-duality to convert its contravariant functoriality
into a covariant functoriality,
we could omit the Pontrjagin dual in the preceding definition,
and indeed it {\em is} traditionally omitted (see e.g.\
\cite{cht}, \cite{blgg}, \cite{jack}, and~\cite{emertongeerefinedBM}).
However, we have found it conceptually clearer to include this duality
in our definitions and constructions.
\end{remark}

We now define a $KZ$-equivariant map \[\alpha_N:
M_{1,Q_N}\to\Ind_{KZ}^{G_N}\left( (M_{1,Q_N})_{K_N}\right)\](where $(M_{1,Q_N})_{K_N}$
denotes the $K_N$-coinvariants in $M_{1,Q_N}$) in the following
way. Note firstly that there is a natural
identification \[(M_{1,Q_N})_{K_N}=\pr^\vee
\bigl(S_{\xi,\tau}(U_1(Q_N)_{N},\cO/\varpi^N)^\vee_{\m_{Q_N}}\bigr),\]so
it suffices to define a $KZ$-equivariant
map \[\alpha_N:S_{\xi,\tau}(U_1(Q_N)_{2N},\cO/\varpi^N)^\vee\to
\Ind_{KZ}^{G_N}S_{\xi,\tau}(U_1(Q_N)_{N},\cO/\varpi^N)^\vee.\] Now,
given $g\in G_N$, we have $g^{-1}K_{2N}g\subseteq K_N$, so that there
is a natural map \[g^* : S_{\xi,\tau}(U_1(Q_N)_{N},\cO/\varpi^N)\to
S_{\xi,\tau}(U_1(Q_N)_{2N},\cO/\varpi^N)\]given by
$(g^*\cdot f)(x):=f(xg)$, and a map \[ g_* := \bigl((g^{-1})^*\bigr)^{\vee} :
S_{\xi,\tau}(U_1(Q_N)_{2N},\cO/\varpi^N)^\vee\to
S_{\xi,\tau}(U_1(Q_N)_{N},\cO/\varpi^N)^\vee.\]
(The latter is well defined since $G_N$ is stable under taking
inverses. We note that $g^*$ (resp.\  $g_*$) may be interpreted as the
natural pullback (resp.\ pushforward) map on cohomology (resp.\ homology) under the natural
right (resp.\ left) action of $\tG(\A^{\infty})$ on the tower of
arithmetic quotients of $\tG$.)
We have $(gh)^* = g^* \circ h^*$ and hence
$(gh)_* = g_* \circ h_*$, whenever all are defined. We also remark that we have the relation \[\alpha_N(\pr^\vee(f))(g)=\pr^\vee(\alpha_N(f)(g)),\] which follows from the fact that $\pr$ and $\pr^\vee$ are defined using Hecke operators away from $p$ and hence commute with $g_*$ and, respectively, $g^*$.
Then we define
$\alpha_N$ by \[\bigl(\alpha_N(x)\bigr)(g):=g_*(x).\]In order to check that
this is $KZ$-equivariant, we must check that for all $k\in KZ$ we have
$\bigl(\alpha_N(kx)\bigr)(g)=(k\alpha_N(x))(g)$; this is equivalent to checking
that $g_*(kx) = (gk)_*(x)$, which is immediate from the definition.

Set
$M_{1,Q_N}^\square:=M_{1,Q_N}\otimes_{R_{\cS_{Q_N}}^{\univ}}R_{\cS_{Q_N}}^{\square_T}$.
We have an induced $KZ$-equivariant map $\alpha_N:
M_{1,Q_N}^\square\to\Ind_{KZ}^{G_N} (M_{1,Q_N}^\square)_{K_N}$. Define
\[\Rinfty:=R^{\loc}[[x_1,\dots,x_{q-[\tilde F^+:\Q]n(n-1)/2}]],\]
$$S_\infty := \cO[[z_1,\dots,z_{n^2\#T},y_1,\dots,y_q]],$$ for formal
variables $x_1,\dots,x_{q-[\tilde F^+:\Q]n(n-1)/2}$, $y_1,\dots,y_q$ and
$z_1,\dots,z_{n^2\#T}$. For each $N$, we fix a surjection $\Rinfty
\onto R_{\cS_{Q_N}}^{\square_T}$ of $R^{\loc}$-algebras (which, as
recalled in Section~\ref{subsec:aux primes}, is
possible by the choice of the sets $Q_N$). These choices allow us to
regard each $M^{\square}_{1,Q_N}$ as an $\Rinfty$-module. Also, we fix choices
of lifts representing the universal deformations over
$R^{\univ}_{\cS}$ and each $R^{\univ}_{\cS_{Q_N}}$ such
that our chosen lift over each $R^{\univ}_{\cS_{Q_N}}$ reduces to our
chosen lift over $R^{\univ}_{\cS}$. These choices give rise to
isomorphisms $R_{\cS_{Q_N}}^{\square_T}\isoto R_{\cS_{Q_N}}^\univ\cotimes_\cO\cO[[z_1,\dots,z_{n^2\#T}]]$
compatible with a fixed isomorphism $R_{\cS}^{\square_T}\isoto
R_{\cS}^{\univ}\cotimes_\cO\cO[[z_1,\dots,z_{n^2\#T}]]$; they also
allow us to regard each $M_{1,Q_N}^{\square}$ as an
$\cO[[z_1,\dots,z_{n^2\#T}]]$-module. Finally, for each $N$, choose a
surjection $\cO[[y_1,\dots,y_q]] \onto \cO[\Delta_{Q_N}]$ with kernel contained in the ideal generated by $(1+y_i)^{p^N}-1$ for $i=1,\dots,q$. This gives
each $M^{\square}_{1,Q_N}$ the structure of an $S_\infty$-module and
hence the structure of an $S_\infty[[K]]$-module
(where the action of $K$ factors through $\Gamma_{2N}$). We note that the action of $S_\infty$ on $M_{1,Q_N}^{\square}$ factors through that of $R_{\cS_{Q_N}}^{\square_T}$.

We now apply the Taylor--Wiles method in the usual way to pass to a
subsequence, and patch the modules $M_{1,Q_N}^\square$ together with the maps
$\alpha_N: M_{1,Q_N}^\square\to\Ind_{KZ}^{G_N}
(M_{1,Q_N}^\square)_{K_N}$. More precisely, for each $N\geq 1$, let
$\gb_N$ denote the ideal of $S_\infty$ generated by $\varpi^N$,
$z_i^N$ and $(1+y_i)^{p^N}-1$.
Let $\ga$ denote the ideal of $S_\infty$ generated by the $z_i$ and the $y_i$. Fix a sequence
$(\gd_N)_{N\geq 1}$ of open ideals of $R^{\univ}_{\cS}$ such that
\begin{itemize}
\item $ \varpi^N R^{\univ}_{\cS} \subset \gd_N \subset
\Ann_{R^{\univ}_{\cS}}( S_{\xi,\tau}(U_{2N},\cO/\varpi^N)^\vee_{\m})$;
\item $\gd_N \supset \gd_{N+1}$ for all $N$;
\item $\cap_{N\geq 1} \gd_N = (0)$.
\end{itemize}
(For example, we may take $\gd_N = \m_{R^{\univ}_{\cS}}^N \cap \Ann_{R^{\univ}_{\cS}}( S_{\xi,\tau}(U_{2N},\cO/\varpi^N)^\vee_{\m})$.)

At level $N$, we consider
tuples consisting of
\begin{itemize}
\item a surjective homomorphism of $R^{\loc}$-algebras $\phi :
  R_\infty \onto R^{\univ}_{\cS}/\gd_N$;
\item a finite projective $(S_\infty/\gb_N)[\Gamma_{2N}]$-module $M^{\square}$
  which carries a commuting action of $R_\infty\otimes_{\cO}\cO[Z]$
  such that the action of $S_\infty$ can be factored through that of $R_\infty$;
  \item the action of $Z\cap K$ should factor through that of $Z_{2N}:=(Z\cap K)/Z\cap K_{2N}$ and be compatible with the action of $\Gamma_{2N}$;
\item an isomorphism $\psi : (M^{\square}/\ga M^{\square}) \isoto
  S_{\xi,\tau}(U_{2N},\cO/\varpi^N)^\vee_{\m}$ that is both $\Gamma_{2N}$-
and $Z$-equivariant, and is compatible with $\phi$;
\item a $KZ$-equivariant and $R_{\infty}\otimes_{\cO}S_\infty$-linear map
$$\alpha : M^{\square} \to \Ind_{KZ}^{G_N}[(M^{\square})_{K_N}],$$ whose reduction modulo $\ga$ is compatible with the map \[ S_{\xi,\tau}(U_{2N},\cO/\varpi^N)^\vee_{\m}\to \Ind_{KZ}^{G_N}[S_{\xi,\tau}(U_{N},\cO/\varpi^N)^\vee_{\m}]\] induced by the action of $G_N$ on the spaces of algebraic modular forms.
\end{itemize}
We consider two such tuples $(\phi,M^{\square},\psi,\alpha)$ and
$(\phi',M^{\square'},\psi',\alpha')$ to be equivalent if $\phi=\phi'$ and if there
is an isomorphism of
$(S_\infty/\gb_N)[\Gamma_{2N}]\otimes_{\cO}R_\infty\otimes_{\cO}\cO[Z]$-modules
$M^{\square}\isoto M^{\square '}$ which identifies $\psi$ with $\psi'$ and $\alpha$ with
$\alpha'$. Note that there are at most finitely many equivalence
classes of such tuples. (Even though the algebra $\cO[Z]$ is not of finite type over $\cO$, the compatibility between the $Z$-action and the $\Gamma_{2N}$-action gives only finitely many possibilities for the $\cO[Z]$-action. Also, for a given $M^{\square}$ there are only finitely many
$KZ$-equivariant homomorphisms $\alpha : M^{\square} \to
\Ind_{KZ}^{G_N}[(M^{\square})_{K_N}]$, because $M^\square$ and $K\backslash G_N /KZ$ are
finite.) Note also that a tuple $(\phi,M^{\square},\psi,\alpha)$  of level
$N$ gives rise to a tuple $(\phi',M^{\square '},\psi',\alpha')$ of
level $(N-1)$ by setting $\phi':=\phi \mod \gd_{N-1}$, $M^{\square '} :=
(M^{\square}/\gb_{N-1})_{K_{2(N-1)}}$ and $\psi' := \psi \mod \varpi^{N-1}$. The
map $\alpha'$ is defined by the formula $\alpha'(\overline m)(g) =
\overline{\alpha(m)(g)}$. Here $\overline{m}$ denotes the image of $m
\in M^{\square}$ in $M^{\square '}$ and $\overline{\alpha(m)(g)}$ denotes the image of
$\alpha(m)(g)$ in $(M^{\square '})_{K_{N-1}}$. Note that for any $m \in M^{\square}$, $\gamma \in K_{2(N-1)}$ and
$g\in G_{N-1}$, we have
\begin{multline*}
 \alpha\bigl((\gamma-1)m\bigr)(g) = \alpha(m)\bigl(g(\gamma-1)\bigr) \\
 = \alpha(m)\bigl((g\gamma g^{-1}-1)g\bigr) 
 = (g\gamma g^{-1}-1) \alpha(m)(g), 
\end{multline*}
by the $K$-equivariance of $\alpha$. Using the fact that $g\gamma
g^{-1} \in K_{N-1}$ it is straightforward to see that $\alpha'$ is
well-defined.

For each pair of integers $N'\geq N \geq 1$, we define a tuple
$(\phi,M^{\square},\psi,\alpha)$ of level $N$ as follows: we set $\phi$ equal to $R_\infty \onto
R^{\square_T}_{\cS_{Q_{N'}}} \onto R^{\univ}_{\cS}/\gd_N$ and we set
$M^{\square}=(M_{1,Q_{N'}}^\square\otimes_{S_\infty} S_\infty/\gb_N)_{K_{2N}}$.
The map $\psi$ comes from points~\eqref{gamma-to-gamma0}
and~\eqref{projectivity} in the previous section. The map
 $\alpha$ comes from $\alpha_{N'}$ defined above (in the same
 way that $\alpha'$ is defined in terms of $\alpha$ in the previous paragraph) and the compatibility it is required to satisfy comes from the commutative diagram in point~\eqref{gamma-to-gamma0} of the previous section.
Since there are only finitely many isomorphism classes of tuples at
 each level $N$, but $N'$ is allowed to be arbitrarily large, we can
 apply a diagonal argument to find a subsequence of pairs
 $(N'(N),N)_{N\geq1}$ indexed by $N$ such that for each $N\geq 2$, the
 tuple indexed by $N$ reduced to level $(N-1)$ is isomorphic to the
 tuple indexed by $(N-1)$. For each $N\geq 2$, we fix a choice of such
 an isomorphism.

We now define
\[ M_\infty := \varprojlim_N
(M^{\square}_{1,Q_{N'(N)}}\otimes_{S_{\infty}}S_{\infty}/\mf b_N)_{K_{2N}},
\]
where the transition maps are induced by the isomorphisms fixed in the
previous paragraph. (We drop the square from the notation here in
order to avoid notational overload in later sections.)
Each of the terms in the projective limit is a (literally) finite $\mathcal O$-module,
endowed with commuting actions of $S_{\infty}[[K]]$ and
$R_\infty\otimes_{\cO}\cO[Z]$, and by construction the transition maps
in the projective limit respect these actions.  Thus
$M_{\infty}$ is naturally a profinite topological $S_{\infty}[[K]]$-module which carries a
commuting action of $R_\infty\otimes_{\cO}\cO[Z]$,
the topology on $M_{\infty}$ being the projective limit topology (where each of
the terms in the projective limit is endowed with the discrete topology).
Moreover, the action of $S_\infty$ on $M_\infty$ can be factored through some map
$S_\infty \to R_\infty$. This follows from the fact that the image of $S_\infty$ in $\mathrm{End}_{S_\infty}(M_\infty)$ is closed and the analogous statement holds
at each finite level $N$. (Recall that the action of $S_\infty$ on $M_{1,Q_N}^{\square}$ factors through that of $R_{\cS_{Q_N}}^{\square_T}$.) We remark that $M_\infty$ and its extra structures depend on the many choices we have made, in particular on the subsequence of pairs $(N'(N),N)$ and on the choice of isomorphisms between different levels for $N\geq 2$.

The module $M_{\infty}$ is the key construction of the paper; the remainder
of this section is devoted to recording some additional properties that
it enjoys.  Firstly, since the transition maps in the projective limit
are given simply by reducing from level $N$ to level $N-1$, it is easily
verified that the natural map induces an isomorphism
$$(M_{\infty}/\gb_N)_{K_{2N}}\iso (M_{1,Q_{N'(N)}}^{\square}/\gb_N)_{K_{2N}}.$$
Next, from this, it follows from the topological form of Nakayama's lemma that
$M_\infty$ is in fact a finite $S_{\infty}[[K]]$-module.
 It follows that the topology on $M_{\infty}$ coincides with the quotient
topology obtained by writing it as a quotient of $S_{\infty}[[K]]^r$,
where $S_{\infty}[[K]]$ is endowed with its natural profinite
topology.
Crucially, there is a $KZ$-equivariant and
$R_{\infty}$-linear map
\[ \alpha_{\infty} : M_\infty \to \Ind_{KZ}^G
M_\infty\]
by taking the projective limit of the maps
\[ (M^{\square}_{1,Q_{N'(N)}}\otimes_{S_\infty} S_\infty/\mf b_N)_{K_{2N}}
\to \Ind_{KZ}^{G_N}\left(
  (M^{\square}_{1,Q_{N'(N)}}\otimes_{S_\infty}S_\infty/\mf b_N)_{K_N}
\right). \]
induced by $\alpha_{N'(N)}: M_{1,Q_{N'(N)}}\to
\Ind_{KZ}^{G_{N'(N)}}\left( (M_{1,Q_{N'(N)}})_{K_{N'(N)}}\right)$. We
denote this induced map by $\alphabar_{N'(N)}$.

The following proposition establishes the additional key properties of
the patched module~$M_{\infty}$ that we will need.

\begin{prop}
  \label{prop:Minfty is projective and has a G-action}
$M_\infty$ is finitely generated and
  projective over $S_\infty[[K]]$, and consequently is finitely generated
over $R_{\infty}[[K]]$.  Furthermore, $\alpha_\infty$ is
  injective, and its image is $G$-stable,
so that $\alpha_\infty$ induces an action of $G$ on $M_\infty$.
\end{prop}
\begin{proof}As we already noted above, $M_\infty$ is finitely generated
over $S_\infty[[K]]$; in particular we may
  choose a surjection $S_\infty[[K]]^r\onto M_\infty$ for some $r\ge
  1$. In order to check that $M_\infty$ is a projective
  $S_\infty[[K]]$-module, it is enough to check that this surjection
  splits.

  Since each $M_{1,Q_N}$ is a projective
  $(\cO/\varpi^N)[\Delta_{Q_N}][\Gamma_{2N}]$-module, we see that
  \[ (M_\infty/\gb_N)_{K_{2N}}\cong
  (M^{\square}_{1,Q_{N'(N)}}/\gb_N)_{K_{2N}}\]
  is a  projective
  $S_\infty/\gb_N[\Gamma_{2N}]$-module,
  so we have a cofinal system of ideals (namely, the ideals generated by
  $\gb_N+\ker(\cO[[K]]\to \cO[\Gamma_{2N}])$) defining the topology of $S_\infty[[K]]$ modulo which the
  surjection splits. The sets of possible splittings at these finite
  levels then give us a projective system of non-empty finite sets,
  and an element of the projective limit of this projective system
  gives the required splitting.

Since, as observed above, the $S_{\infty}$-action on $M_{\infty}$ factors
through the $R_{\infty}$-action, we see that $M_{\infty}$ is also finitely
generated over $R_{\infty}[[K]]$.

We now check that $\alpha_\infty$ is injective. Note that by
definition, for each
$\alpha_N: M_{1,Q_N}\to \Ind^{G_N}_{KZ}\bigl((M_{1,Q_N})_{K_N}\bigr)$ we
have $\bigl(\alpha_N(m)\bigr)(1)=1_*(m)=\overline{m}$, where $\overline{m}$
denotes the image of $m$ in $(M_{1,Q_N})_{K_N}$. From this (with $N$
replaced by $N'(N)$)
we deduce
that $\bigl(\alphabar_{N'(N)})(m)\bigr)(1) = \overline{m}$ where $\overline{m}$ is the image
of $m\in (M^{\square}_{1,Q_{N'(N)}}/\gb_N)_{K_{2N}}$ in
$(M^{\square}_{1,Q_{N'(N)}}/\gb_N)_{K_{N}}$.  This then implies that $\bigl(\alpha_\infty(m)\bigr)(1)=m$ for each $m\in M_{\infty}$, and $\alpha_\infty$ is certainly
injective.

In order to show that the image of $\alpha_\infty$ is
$G$-stable, we will show
that for all $g\in G$, $m\in M_\infty$ we have\[g\bigl(\alpha_\infty(m)\bigr)=\alpha_\infty\bigl(\bigl(\alpha_\infty(m)\bigr)(g)\bigr).\]
In other words, we will show that for all $g,h\in G$ and
$m\in M_\infty$, we have  \[\bigl(\alpha_\infty(m)\bigr)(hg)=\bigl(\alpha_\infty\bigl((\alpha_\infty(m))(g)\bigr)\bigr)(h).\]

Let $m$ be an element of $M_\infty$ and let $N$ be any
integer large enough so that $g,h,gh\in G_N$. This certainly means
that $g,h,gh\in G_{2N}$ as well. Since $N$ can be arbitrarily large, it is enough to show that both
sides of the equation above become equal in
$(M_{\infty}/\gb_N)_{K_N}$ and we do this by explicit computation.

We let $\pi_N : M_\infty \to (M_\infty/\gb_N)_{K_{2N}}$ and $\sigma_N
: M_\infty \to (M_\infty/\gb_N)_{K_N}$ denote the projection
maps. Then by definition, we have
\[ \sigma_N\bigl(\alpha_{\infty}(m)(hg)\bigr) = \alphabar_{N'(N)}\bigl(\pi_N(m)\bigr)(hg) \]
and
\begin{align*}
 \sigma_N\bigl(\alpha_\infty(\alpha_\infty(m)(g))(h)\bigr) & = \alphabar_{N'(N)}\bigl(
 \pi_N(\alpha_\infty(m)(g))\bigr)(h) \\
& = \alphabar_{N'(N)}\bigl( \sigma_{2N}(\alpha_\infty(m)(g)) \bmod \gb_N\bigr)(h)  \\
& = \alphabar_{N'(N)}\bigl( \alphabar_{N'(2N)}(\pi_{2N}(m)) (g) \bmod \gb_N\bigr)(h).
\end{align*}
Now, for integers $N,\wt N,N''\geq 1$ with $\wt N\leq N'(N)$, we let
$U_1(Q_{N'(N)},\wt N)_{N''}$ be the
open compact subgroup lying between $U_1(Q_{N'(N)})_{N''}$ and
$U_0(Q_{N'(N)})_{N''}$ for which $U_0(Q_{N'(N)})_{N''}/U_1(Q_{N'(N)},\wt N)_{N''}
\cong (\Z/p^{\wt N}\Z)^q$. Then we have a commutative diagram

\[
\begin{CD}
  S_{\xi,\tau}\bigl(U_1(Q_{N'(2N)},2N)_{4N},\cO/\varpi^{2N}\bigr)^{\vee}
  @>g_* >>
  S_{\xi,\tau}\bigl(U_1(Q_{N'(2N)},2N)_{2N},\cO/\varpi^{2N}\bigr)^{\vee}
  \\
@|  @VV\text{nat}V \\
  S_{\xi,\tau}\bigl(U_1(Q_{N'(2N)},2N)_{4N},\cO/\varpi^{2N}\bigr)^{\vee}
  @.
  S_{\xi,\tau}\bigl(U_1(Q_{N'(2N)},N)_{2N},\cO/\varpi^{N}\bigr)^{\vee}
  \\
@VV\text{nat}V  @VVh_*V \\
  S_{\xi,\tau}\bigl(U_1(Q_{N'(2N)},N)_{2N},\cO/\varpi^{N}\bigr)^{\vee}
  @>(hg)_* >>
  S_{\xi,\tau}\bigl(U_1(Q_{N'(2N)},N)_{N},\cO/\varpi^{N}\bigr)^{\vee}.
\end{CD}
\]

First localising at $\m_{Q_{N'(2N)}}$, then applying the projector $\pr^\vee$,
then tensoring over
$R^{\univ}_{\cS_{Q_{N'(N)}}}$ with $R^{\square_T}_{\cS_{Q_{N'(N)}}}$,
and finally
reducing modulo $\gb_{2N}$ or $\gb_N$ as appropriate, we obtain a
commutative diagram:
\[
\begin{CD}
  (M_{\infty}/\gb_{2N})_{K_{4N}}
  @>g_*>>
  (M_{\infty}/\gb_{2N})_{K_{2N}}
  \\
@|  @VV\text{nat}V \\
  (M_{\infty}/\gb_{2N})_{K_{4N}}
  @.
   (M_{\infty}/\gb_{N})_{K_{2N}}
  \\
@VV\text{nat}V  @VVh_*V \\
  (M_{\infty}/\gb_{N})_{K_{2N}}
  @>(hg)_*>>
    (M_{\infty}/\gb_{N})_{K_{N}}.
\end{CD}
\]

Now, on the one hand, the element $\pi_{2N}(m)$ lies in the space in the upper left corner
of this diagram. By definition, we have
\[  \alphabar_{N'(N)}\bigl( \alphabar_{N'(2N)}(\pi_{2N}(m)) (g) \bmod \gb_N\bigr)(h)
= h_* \circ \text{nat} \circ g_* \bigl(\pi_{2N}(m)\bigr) \]
On the other hand, $\pi_{N}(m)$ lies in the
space in the lower left corner of the diagram, and we have
\[  \alphabar_{N'(N)}\bigl(\pi_N(m)\bigr)(hg)  = (hg)_*\bigl(\pi_N(m)\bigr) = (hg)_*\circ
\text{nat} \bigl(\pi_{2N}(m)\bigr). \]
The desired equality now follows from the commutativity of the above diagrams.
\end{proof}

Note that for each positive integer $m$, the compact open subgroups $U_m$ have
the same level away from $\gp$. Let $U^\p \subset \tG(\mathbb{A}_{\tF^+}^{\infty,
  \gp})$ denote that common level. Define the $\varpi$-adically completed cohomology space \[\tilde S_{\xi, \tau}(U^\p,
\cO)_{\m}:=\varprojlim_s\bigl(\varinjlim_m
S_{\xi,\tau}(U_m,\cO/\varpi^s)_{\m}\bigr).\] The space $\tilde
S_{\xi,\tau}(U^\p,\cO)_{\m}$ is equipped with a natural $G$-action, induced from
the action of $G$ on algebraic automorphic forms. Moreover,
$\tilde S_{\xi,\tau}(U^\p,\cO)_{\m}$ is equipped with a natural action of the Hecke
algebra \[\mathbb{T}^{S_p}_{\xi,\tau}(U^\p,\cO)_{\m}:=\varprojlim_{m}\mathbb{T}^{S_p}_{\xi,\tau}(U_m,\cO)_{\m}.\]
By taking the inverse limit of the
$\mathbb{T}^{S_p}_{\xi,\tau}(U_m,\cO)_{\m}$-valued deformations of $\bar\rho$ of
type $\cS$ we obtain a map $R^{\univ}_{\cS}\to
\mathbb{T}^{S_p}_{\xi,\tau}(U^\p,\cO)_{\m}$. Therefore, $\tilde
S_{\xi,\tau}(U^\p,\cO)_{\m}$ is also equipped with an action of the local
deformation ring $R_{\tp}^\square$ via the composition $R_{\tp}^\square \to
R^{\univ}_{\cS}\to \mathbb{T}^{S_p}_{\xi,\tau}(U^\p,\cO)_{\m}$.

\begin{cor}\label{cor: comparison with completed cohomology} There is a $G$-equivariant isomorphism \[(M_\infty/\ga M_\infty) \isoto \tilde S_{\xi,\tau}(U^\p,\cO)_{\m}^{d},\] which is compatible with the $R_{\tp}^\square$-action on both sides.
\end{cor}
\begin{proof}  We have a $KZ$-equivariant isomorphism \[\psi_\infty: (M_\infty/\ga M_\infty) \isoto \tilde S_{\xi,\tau}(U^\p,\cO)_{\m}^{d}\] obtained by taking the projective limit of the $\Gamma_{2N}$- and $Z$-equivariant isomorphisms $\psi$ at level $N$.

We first show that $\psi_\infty$ is compatible with the $R_{\tp}^\square$-action on both sides. The $R_{\tp}^\square$-action on the left hand side is via the map $R_{\tp}^\square\to R_\infty$ and each $\psi$ at level $N$ is compatible with $\phi :
  R_\infty \onto R^{\univ}_{\cS}/\gd_N$. On the other hand, the action of $R_{\tp}^\square$ on the right hand side is via the map $R_{\tp}^\square \to R^{\univ}_{\cS}$ and $\cap_{N\geq1}\gd_N = (0)$. The desired compatibility follows.

  Moreover, $\psi_\infty$ fits into a commutative diagram
\[
\begin{CD}
 \bigl(M_\infty/\ga M_\infty \bigr)
  @>\psi_\infty>>
\tilde S_{\xi,\tau}(U^\p,\cO)_{\m}^{d}
  \\
@VV g V  @VV g V \\
\Ind_{KZ}^G\bigl(M_\infty/\ga M_\infty\bigr)
  @>\psi_\infty>>\Ind_{KZ}^G\bigl(\tilde S_{\xi,\tau}(U^\p,\cO)_{\m}^{d}\bigr)\end{CD}
\]
where both vertical maps are induced by the action of $g\in G$. (The fact that the right vertical map has $G$-stable image follows from the analogue of Proposition~\ref{prop:Minfty is projective and has a G-action} for the completed cohomology space $\tilde S_{\xi,\tau}(U^\p,\cO)_{\m}^{d}$. Since $M_\infty$ can be thought of as a patched version of completed cohomology, the arguments for $\tilde S_{\xi,\tau}(U^\p,\cO)_{\m}^{d}$ are analogous, but easier.)
 The $G$-equivariance of $\psi_\infty$ now follows from Proposition~\ref{prop:Minfty is projective and has a G-action}.
\end{proof}

\subsection{Admissible unitary Banach
  representations}\label{subsec:admissible unitary Banach}
An \emph{$E$-Banach space representation} $V$ of $G$ is an $E$-Banach space $V$ together with
a $G$-action by continuous linear automorphisms such that the corresponding map $G\times V\to V$
is continuous. A Banach space representation $V$ is called \emph{unitary} if there exists a $G$-invariant
norm defining the topology on $V$. The existence of such a norm is equivalent to the existence of
an open, bounded $G$-invariant $\cO$-lattice $\Theta\subset V$. A unitary $L$-Banach space representation
is \emph{admissible}  if $\Theta \otimes_\cO \F$ is an admissible (smooth) representation of $G$.
(This means that the space of invariants $(\Theta\otimes_\cO \F)^H$ is finite-dimensional for every open
subgroup $H\subset G$.) This definition of admissibility  is equivalent to that of~\cite{MR2290601} by Proposition 6.5.7 of~\cite{emerton.locally.analytic.vectors}. (The latter definition requires $\Theta^d$ to be
a finitely generated module over $\cO[[H]]$.)

Fix a
lifting $r:G_F\to\GL_n(\cO)$ of $\rbar$. We now explain how
$M_\infty$ allows us to associate to $r$ an admissible unitary Banach
representation $V(r)$ of $\GL_n(F)$.

As above, we identify $F$ with $\tF_{\tp}$. By definition, $r$ comes
from a homomorphism of $\cO$-algebras $x:R_{\tp}^\square\to\cO$. We
extend this to a homomorphism of $\cO$-algebras \[x':R^\square_{\tp}\widehat{\otimes}\left(\widehat{\otimes}_{v \in  S_p\setminus\{\p\}}R_\tv^{\square,\xi,\tau}\right)\to\cO\] by
using the homomorphisms $R_\tv^{\square,\xi,\tau}\to\cO$
corresponding to our given potentially diagonalisable representation
$r_{\textrm{pot.diag}}$, and then extend $x'$  arbitrarily to a homomorphism of $\cO$-algebras
$y:R_\infty\to\cO$. We set
$V(r):=(M_\infty\otimes_{R_\infty,y}\cO)^d[1/p]$.\begin{prop}
  \label{prop: V(r) is admissible Banach}The representation $V(r)$ is
  an admissible unitary Banach representation of $\GL_n(F)$.
\end{prop}
\begin{rem}Note that we do not know if $V(r)$ is independent of
  either the global setting or the choice of $y$. We also do not
  know that $V(r)$ is necessarily nonzero, although we will prove that
  $V(r)\ne 0$ for many regular de Rham representations $r$ in Section~\ref{sec:breuilschneider} below, as a consequence of the stronger result that (for the
particular choice of $r$ under consideration) the
subspace of locally algebraic vectors in $V(r)$ is nonzero.

\end{rem}
\begin{proof}[Proof of Proposition~\ref{prop: V(r) is admissible
    Banach}] The image of $(M_\infty\otimes_{R_\infty,y}\cO)^d$ in
  $V(r)$ is a unit ball stable under $\GL_n(F)$,
 and so $V(r)$ is a unitary representation.

 In order to see that $V(r)$ is admissible, we must
  show that for each $N\ge 0$, the $\F$-vector space
  $\bigl((M_\infty\otimes_{R_\infty,y}\cO)^d\otimes\F\bigr)^{K_N}$ is
  finite-dimensional. Writing $\ybar$ for the composite
  $R_\infty\stackrel{y}{\to}\cO\to\F$, we must check that
  $(M_\infty\otimes_{R_\infty,\ybar}\F)_{K_N}$ is
  finite-dimensional. Since $M_\infty$ is a finite
  $S_\infty[[K]]$-module, and $\ker\ybar$ induces an open ideal of
  $S_\infty$, this is immediate.
\end{proof}

\begin{rem}
  \label{rem: getting a Banach space from a general r without having
    to worry about a pot diag lift of rbar}While we have assumed
  throughout this section that $\rbar$ has a potentially
  diagonalisable lift with regular Hodge--Tate weights, this
  hypothesis is not needed for our main results, which concern
representations $r:G_F\to \GL_n(E)$. Indeed, possibly after making a
  ramified extension of $E$, it is easy to see that any such
  representation can be conjugated to a representation $r'$ which factors
through $\GL_n(\cO)$ and whose reduction
  $\rbar'$ is semisimple, so that $\rbar'$ has a
  lift of the required kind (possibly after further extending $E$) by Lemma~\ref{lem: existence of pot
    diagonal lift}.
\end{rem}

\section{Hecke algebras and types}\label{sec:types}In the following
two sections we will use the local Langlands correspondence and the
theory of types to establish local-global compatibility results for
our patched modules. In particular, we will make use of the results of
\cite{MR1711578} and \cite{MR1728541} in order to study spaces of
automorphic forms which correspond to fixed inertial types. As
explained in the introduction, in some of our results we will for
simplicity restrict ourselves to the case of Weil--Deligne
representations with $N=0$; this means that we will limit ourselves
to considering potentially crystalline Galois representations. However, some of
our results are more naturally expressed in the more general context
of representations with arbitrary monodromy, so we will make it clear
when we impose this restriction. We begin by collecting and
explaining various results from the literature that we will need.

Let $F/\Qp$ be a finite extension. Recall that we write
$K=\GL_n(\cO_F)$, $G=\GL_n(F)$, and that $W_F$ is the Weil group of
$F$. Although several of our references work over $\C$, we work consistently over $\Qbar_p$ (except in
Subsection~\ref{god_save_the_queen}, where we fix a single Bernstein component,
and work over a finite extension of $\Q_p$).
The various results over $\C$ are transferred to
our context over $\Qbar_p$ via our fixed choice of $\imath:\Qpbar\stackrel{\sim}{\to}\C$. 

As recalled in Section~\ref{subsec:notation}, the
local Langlands correspondence $\rec_p$ gives a bijection between the
isomorphism classes of irreducible smooth representations of $G$ over~$\Qpbar$,
and the $n$-dimensional Frobenius semisimple Weil--Deligne
representations of~$W_F$, independently of the choice of $\imath$;
thus none of our results below depend on the choice of~$\imath$.

\subsection{Bernstein--Zelevinsky theory}\label{subsec BZ theory}We
now recall some details of the local Langlands correspondence and its
relationship to Bernstein--Zelevinsky theory, following~\cite{rodier}.

Given a partition $n=n_1+\dots+n_r$, let $P=MN$ be the corresponding
standard parabolic subgroup of $G$ with Levi subgroup $M$ and
unipotent radical $N$ (standard with respect to the Borel subgroup of upper triangular matrices), so that
$M\cong\GL_{n_1}(F)\times\cdots\times\GL_{n_r}(F)$. For any smooth
representations $\pi_i$ of $\GL_{n_i}(F)$, the tensor product
$\pi_1\otimes\cdots\otimes\pi_r$ is naturally a representation of $M$
and thus of $P$, and we denote by $\pi_1\times\cdots\times\pi_r$ the
normalised induction
$\nind_P^G (\pi_1\otimes\cdots\otimes\pi_r)$.
If $\pi_1,\dots,\pi_r$ are irreducible
supercuspidal representations then $\pi_1\times\dots\times\pi_r$ has
finite length by~\cite[Proposition 4]{rodier}.

Any smooth irreducible representation $\pi$ of $G$ is necessarily a
subrepresentation of some $\pi_1\times\cdots\times\pi_r$ for some $P$
and some irreducible supercuspidal representations
$\pi_1,\dots,\pi_r$. By~\cite[Proposition 5]{rodier}, $\pi$ determines
the multiset $\{\pi_1,\dots,\pi_r\}$ (and thus the corresponding
partition of $n$, up to reordering) uniquely, and we refer to it as
the \emph{cuspidal support} of $\pi$. For any representation $\pi_i$
of $\GL_{n_i}(F)$, and any integer $s$, we write
$\pi_i(s):=\pi_i\otimes|\det|^s$. Then by~\cite[Th\'eor\`eme
1]{rodier} (and the remark which follows it),
$\pi_1\times\cdots\times\pi_r$ is reducible if and only if there are
some $i,j$ with $n_i=n_j$ and $\pi_j\cong\pi_i(1)$.

We define a \emph{segment} to be a set  of isomorphism classes
of irreducible cuspidal representations $\GL_{n_i}(F)$ of the form
$\Delta=\{\pi_i,\pi_i(1),\dots,\pi_i(r-1)\}$ for some $r\ge 1$, and we
write $\Delta=[\pi_i,\pi_i(r-1)]$. We say that two segments
$\Delta_1,\Delta_2$ are \emph{linked} if neither contains the other,
and $\Delta_1\cup\Delta_2$ is also a segment. If $\Delta_1=[\pi_i,\pi_i']$ and
$\Delta_2=[\pi_i'',\pi_i''']$ are linked, we say that $\Delta_1$
\emph{precedes} $\Delta_2$ if $\pi_i''=\pi_i(r)$ for some $r\ge 0$.

If $\Delta=[\pi_i,\pi_i(r-1)]$ is a segment, then
$\pi_i\times\cdots\times\pi_i(r-1)$ has a unique irreducible
subrepresentation $Z(\Delta)$ and a unique irreducible quotient
$L(\Delta)$. By~\cite[Th\'eo\-r\`eme 2]{rodier}, if
$\Delta_1,\dots,\Delta_r$ are segments such that if $i<j$, then
$\Delta_i$ does not precede $\Delta_j$, then the representation
$Z(\Delta_1)\times\cdots\times Z(\Delta_r)$ has a unique irreducible
subrepresentation, which we denote
$Z(\Delta_1\times\cdots\times\Delta_r)$. Every irreducible smooth
representation of $G$ is of the form
$Z(\Delta_1\times\cdots\times\Delta_r)$ for some segments
$\Delta_1,\dots,\Delta_r$, which are uniquely determined up to
reordering.  By~\cite[Th\'eor\`eme 3]{rodier}, the analogous statement
also holds for the $L(\Delta_i)$. By~\cite[Th\'eor\`eme 5]{rodier},
$Z(\Delta_1,\dots,\Delta_r)$ occurs with multiplicity one as a
subquotient of $Z(\Delta_1)\times\cdots\times Z(\Delta_r)$.

The link with the local Langlands correspondence is as
follows~\cite[\S 4.4]{rodier}. The representation $\rec_p(\pi)$ is
irreducible if and only if $\pi$ is supercuspidal. For each integer
$r\ge 1$ there is an explicitly defined $r$-dimensional Weil--Deligne
representation $\Sp(r)$ whose restriction to the Weil group is
unramified and whose monodromy operator has rank $r-1$ (see page 213 of~\cite{rodier}). Then if
$\Delta=[\pi_i,\pi_i(r-1)]$ is a segment, we have
$\rec_p\bigl(L(\Delta)\bigr)=\rec_p(\pi_i)\otimes\Sp(r)$, and more generally we
have
$\rec_p\bigl(L(\Delta_1,\dots,\Delta_r)\bigr)=\oplus_{i=1}^r\rec_p\bigl(L(\Delta_i)\bigr)$.

\subsection{The Bernstein Centre}\label{subsec:bernstein centre}We now
briefly recall some of the results of~\cite{MR771671} (in the special case that
the reductive group under consideration there is $\GL_n$) in a fashion adapted
to our needs. The \emph{Bernstein spectrum} is the set of $G$-orbits of pairs
$(M,\omega)$, where $M$ is a Levi subgroup of $G$, $\omega$ is an irreducible
supercuspidal representation of $M$, and the action of $G$ is via conjugation;
note that up to conjugacy, $(M,\omega)$ is of the form
$(\GL_{n_1}(F)\times\cdots\times\GL_{n_r}(F),\pi_1\otimes\cdots\otimes\pi_r)$,
as in the preceding section. As we will explain below, the Bernstein spectrum
naturally has the structure of an algebraic variety over $\Qpbar$ (with
infinitely many connected components), the \emph{Bernstein variety}. Given an
irreducible smooth $G$-representation~$\pi$, we obtain a point of the Bernstein
spectrum by passing to the cuspidal support of $\pi$.
This map is surjective; any point $(M,\omega)$ of the Bernstein spectrum
is equal to the cuspidal support of $\pi$ for any Jordan--H\"older factor
$\pi$ of $\nind_P^G \omega$ (and indeed these are precisely the $\pi$ for
which $(M,\omega)$ arises as the cuspidal support; here $P$ is any
parabolic subgroup of $G$ admitting $M$ as a Levi quotient ---
the collection of Jordan--H\"older factors of $\nind_P^G \omega$
is independent of the choice of $P$).

The connected components of the Bernstein variety are as follows. Fix
a pair $(M,\omega)$ as above; then the component of the Bernstein
variety containing $(M,\omega)$ is the union of the $G$-orbits of the
pairs $(M,\alpha\omega)$, where $\alpha$ is an unramified
quasicharacter of $M$. We say that two pairs $(M,\omega)$ and
$(M',\omega')$ are \emph{inertially equivalent} if they are in the
same Bernstein component, and write $[M,\omega]$ for the equivalence class. Fixing one such pair $(M,\omega)$, it is
easy to see that there is a natural algebraic structure on the
inertial equivalence class, because the set of unramified
quasicharacters of $M$ has a natural algebraic structure, and thus so
does any quotient of it by a finite group; this gives the structure of
the Bernstein variety. Given an irreducible smooth $G$-representation $\pi$, we will
refer to the inertial equivalence class of its cuspidal support as the
\emph{inertial support} of $\pi$.

For any connected component $\Omega$ of the Bernstein variety, we have
a corresponding full subcategory of the category of smooth
$G$-representations, whose objects are the smooth representations all
of whose irreducible subquotients have cuspidal support in $\Omega$.
Such a subcategory is called a \emph{Bernstein component} of the
category of smooth $G$-representations, and in fact the category of
smooth $G$-representations is a direct product of the Bernstein
components. Given a Bernstein component $\Omega$, the \emph{centre}
$\mathfrak{Z}_\Omega$ of $\Omega$ is the centre of the corresponding
Bernstein component (that is, the endomorphism ring of the identity
functor), so that $\mathfrak{Z}_\Omega$ acts naturally on each
irreducible smooth representation $\pi\in\Omega$. Since $\pi$ is irreducible,
each element of $\mathfrak{Z}_\Omega$ will act on $\pi$ through a scalar.
In fact this scalar depends only on the cuspidal support of $\pi$,
and in this way $\mathfrak{Z}_{\Omega}$ is identified
with the ring of regular functions on
the connected component $\Omega$ of the Bernstein variety.

The above notions extend in an obvious way to products of groups
$\GL_{n_i}(F)$, and we will make use of this extension below without
further comment (in order to compare representations of $G$ with
representations of a Levi subgroup).

Finally, we note that from the link between the local Langlands correspondence and
Bernstein--Zelevinsky theory explained in Section~\ref{subsec BZ theory}, it
is immediate that two irreducible smooth representations $\pi, \pi'$
of $G$ lie in the same Bernstein component if and only if
$\rec_p(\pi)|_{I_F}\cong\rec_p(\pi')|_{I_F}$ (where we ignore the
monodromy operators).

\subsection{Bushnell--Kutzko theory}\label{subsec:BK theory}Fix a
Bernstein component $\Omega$. In~\cite{MR1711578}, there is the
definition of a \emph{semisimple Bushnell--Kutzko type} $(J,\lambda)$
for $\Omega$, where $J\subseteq K$ is a compact open subgroup, and
$\lambda$ is a smooth irreducible $\Qpbar$-representation of $J$. The
pair $(J,\lambda)$ has the property that if $\pi$ is an irreducible
smooth representation of $G$, then $\Hom_J(\lambda,\pi)\ne 0$ if and
only if $\pi\in\Omega$, and in fact the functor $\Hom_J(\lambda,
\ast)$ induces an equivalence of categories between $\Omega$ and the
category of left modules of the Hecke algebra $\cH(G,\lambda):=\mathcal H(G, J;
\lambda):=\End_G(\cInd_{J}^{G} \lambda)$, with the inverse functor
being given by $\cInd_J^G \lambda \otimes_{\mathcal H(G,
  \lambda)}(*)$.

Let $\pi$ be an irreducible representation in $\Omega$, and let $(M,\omega)$
be a representative for the inertial support of $\pi$.  We can and do suppose that $M$ is a
standard Levi subgroup $\prod_{i=1}^r\prod_{j=1}^{d_i}\GL_{e_i}(F)$,
and that $\omega=\otimes_{i=1}^r\pi_{i}^{\otimes d_i}$, where
$\pi_i$ and $\pi_{i'}$ are not inertially equivalent (that is, do not
differ by a twist by an unramified quasicharacter) if $i\ne i'$. Then
by the construction of $(J,\lambda)$ in~\cite{MR1711578}, there is a
pair $(J\cap M,\lambda_M)$ which is a semisimple Bushnell--Kutzko type
for the Bernstein component $\Omega_M$ of $M$ determined by $\omega$, in the
sense that if $\pi_M$ is an irreducible smooth representation of $M$,
then $\Hom_{J\cap M}(\lambda_M,\pi_M)\ne 0$ if and only if
$\pi_M\in\Omega_M$, and the functor $\Hom_{J\cap M}(\lambda_M,
\ast)$ induces an equivalence of categories between $\Omega_M$ and the
category of left modules of $\cH(M,\lambda_M)$ in the same way
as above.

Let $P$ be a parabolic subgroup of $G$ with Levi factor $M$. Then the
normalised induction $\nind_P^G$ restricts to a functor from $\Omega_M$ to
$\Omega$, and there is a unique injective algebra homomorphism
$t_P:\cH(M,\lambda_M)\to\cH(G,\lambda)$ (which is denoted
$j_{\cQ}$ on page 55 of~\cite{MR1711578}) such that under the
equivalences of categories explained above, $\nind_P^G$ corresponds to
the pushforward along $t_P$ (given by $\Hom_{\cH(M,\lambda_M)}(\cH(G,\lambda),*)$).

It will be useful to us to have a somewhat more explicit description
of the pair $(J\cap M,\lambda_M)$. Recall that we may write
$M=\prod_{i=1}^r\prod_{j=1}^{d_i}\GL_{e_i}(F)$, and that
$\omega=\otimes_{i=1}^r\pi_{i}^{\otimes d_i}$. Then as is explained
in~\cite[\S 1.4-5]{MR1711578} (see also the paragraph after Lemma
7.6.3 of~\cite{MR1204652}), we may write $J\cap
M=\prod_{i=1}^r\prod_{j=1}^{d_i}J_i$,
$\lambda_M=\bigotimes_{i=1}^r\lambda_i^{\otimes d_i}$, where
$(J_i,\lambda_i)$ is a maximal simple type occurring in $\pi_i$ in the
sense of~\cite{MR1643417}. There is a corresponding isomorphism
$\cH(M,\lambda_M)\cong\otimes_{i=1}^r\cH\bigl(\GL_{e_i}(F),\lambda_i\bigr)^{\otimes
  d_i}$, and each $\cH\bigl(\GL_{e_i}(F),\lambda_i\bigr)$ is commutative, so
that $\cH(M,\lambda_M)$ is commutative. The following somewhat
technical lemma will be useful to us in Section~\ref{sec:local-global
 compatibility}. We remind the reader that the algebra
$\cH(M,\lambda_M)$ is naturally isomorphic to the convolution algebra of
compactly supported
functions $f: M \to \End_{\Qpbar}(\lambda_M)$ such that
$f(jmj')=j\circ f(m) \circ j'$ for all $m\in M,j,j'\in J\cap
M$ (see~\cite[\S 2.2]{barthel-livne}).

\begin{lem}
  \label{lem:H(M,lambda_M) has a basis which is finite over Z(M)}There
  is an integer $e\ge 1$ and a $\Qpbar$-basis for $\cH(M,\lambda_M)$
  with the property that if $\nu$ is an element of this basis, then the
  $e$-fold convolution of $\nu$ with itself is supported on $t_\nu(J\cap M)$
  for some $t_\nu\in Z(M)$.
\end{lem}
\begin{proof}
  By the above remarks, we need only prove the corresponding result
  for the Hecke algebra of a maximal simple type
  $\cH\bigl(\GL_{e_i}(F),\lambda_i\bigr)$. In this case, the proof of Theorem
  7.6.1 of~\cite{MR1204652} shows that we can take a basis given by
  Hecke operators supported on cosets of the form $tJ_i$ where
  $t\in\bm{D}(\mathfrak{B})$ in the notation of~\cite{MR1204652}. By
  Lemma 7.6.3 of~\cite{MR1204652}, it suffices to show that there is a
  positive integer $e$ such that if $t\in\bm{D}(\mathfrak{B})$ then the $e$-fold composition of a Hecke operator $\psi_t$ supported on $tJ_i$ with itself is supported on $sJ_i$, where $s$ is a scalar matrix.
  But $\bm{D}(\mathfrak{B})$ is a cyclic
  group, generated by a uniformiser $\varpi_E$ of some finite extension
  of fields $E/F$ inside $\GL_{e_i}(F)$, so we can just take $e$ to be
  the ramification degree $e(E:F)$. In that case, the $e$-fold composition of a Hecke operator supported on $\varpi_EJ_i$ with itself is supported on $\varpi_F J_i$, as remarked on the bottom of page 201 of~\cite{MR1204652}.
\end{proof}

\begin{rem}\label{rem:ramification index} If $\lambda_i$ is a maximal
  simple type for $\GL_{n_i}(F)$, then the ramification degree
  $e(E:F)$ is equal to $n_i/f_i$, where $f_i$ is the number of
  unramified characters of $F^\times$ which preserve a supercuspidal
  $\GL_{n_i}(F)$-representation containing $\lambda_i$. This follows
  from Lemma 6.2.5 of~\cite{MR1204652}. Therefore, the element
  $\det(\varpi_{E})$ of $F^\times$ has valuation $f_i$.
\end{rem}

\subsection{Results of Schneider--Zink and Dat}\label{subsec:SZ} We will need a
slight refinement of the Bernstein centre and of the theory of
Bushnell--Kutzko, which is constructed in the
paper~\cite{MR1728541}. Note that there is not quite a bijection
between irreducible smooth representations of $G$ and characters of
the Bernstein centre; as explained in Section~\ref{subsec:bernstein
  centre}, any two Jordan--H\"older constituents of a parabolic
induction from an irreducible cuspidal representation correspond to
the same character of the centre, so that for example the trivial
representation and the Steinberg representation correspond to the same
character. Furthermore, as recalled in Section~\ref{subsec:bernstein
  centre}, two irreducible representations lie in the same Bernstein
component if and only if the corresponding Weil--Deligne
representations agree on inertia, but have possibly differing
monodromy operators, and it can be useful to have a finer
decomposition. In particular, we wish to be able to consider only the
representations with $N=0$.  

Let $\Omega$ be a Bernstein component, with $(J,\lambda)$ being the corresponding semisimple Bushnell--Kutzko type. Let $\Irr(\Omega)_0$ denote the
irreducible elements $\pi$ of $\Omega$ with the property that $N=0$ on
$\rec_p(\pi)$. In~\cite[\S 2]{MR1728541}, the material recalled in
Section~\ref{subsec BZ theory} above is recast in terms of certain
partition valued functions, which depend on the inertial support $\Omega$. In~\cite[\S 3]{MR1728541}, a partial ordering is defined on
these functions, and there is a unique maximal element for this
ordering, which we will denote by $\cP$ from now on. In~\cite[\S
6]{MR1728541} an irreducible direct summand $\sigma_{\cP}(\lambda)$ of
$\Ind_J^K\lambda$ is constructed, with the property that if $\pi$ is an
irreducible smooth $G$-representation, then
$\Hom_K(\sigma_\cP(\lambda),\pi)\ne 0$ if and only if
$\pi\in\Irr(\Omega)_0$, in which case $\sigma_\cP(\lambda)$ occurs in
$\pi|_K$ with multiplicity one. (This follows immediately from
Proposition 6.2 of~\cite{MR1728541}, using the relationship between
the partial ordering on partition valued functions in~\cite{MR1728541} and the monodromy operators,
which is explained in the proof of Proposition 6.5.3
of~\cite{MR2656025}.)

If $\tau$ is the inertial type corresponding to $\Omega$, then we
write $\sigma(\tau)$ for $\sigma_\cP(\lambda)$. As remarked on page
201 of~\cite{MR1728541}, since $\mathcal P$ is maximal, $\sigma(\tau)$
coincides with the representation
$\rho_{\mathfrak{s}}=e_K\Ind_J^K\lambda$ constructed in Theorem 4.1 of
~\cite{dat}. Here $e_K$ is a certain idempotent in $\cH(G,\lambda)$
which is constructed in~\cite[\S 4.2]{dat}.

Let $\mathfrak Z_{\Omega}$ be the centre of $\Omega$.  By Theorem 4.1
of~\cite{dat}, the action of $\mathfrak Z_{\Omega}$ on
$\cInd_{K}^{G}\sigma(\tau)$ induces an isomorphism $$\mathfrak
Z_{\Omega}\cong\End_G\bigl(\cInd_{K}^{G}\sigma(\tau)\bigr) =: \cH\bigl(G,\sigma(\tau)\bigr),$$
so in particular
$\cH\bigl(G,\sigma(\tau)\bigr)$ is commutative.

Let $W_{[M,\omega]}:=N_G([M,\omega])/M$ be the relative normaliser of the
inertial equivalence class of $(M,\omega)$, with $(M,\omega)$ as in
Section~\ref{subsec:bernstein centre}. Then by Proposition 2.1 of~\cite{dat},
the natural algebra homomorphism $t_P:\cH(M,\lambda_M)\to\cH(G,\lambda)$ induces
an isomorphism $t_P:\cH(M,\lambda_M)^{W_{[M,\omega]}}\isoto
Z\bigl(\cH(G,\lambda)\bigr)$. It follows from \cite[Theorem 4.1]{dat} that
there is an algebra homomorphism $s_P:\cH(G,\lambda)\to \cH\bigl(G,\sigma(\tau)\bigr)$,
which makes $\cH\bigl(G,\sigma(\tau)\bigr)$ a direct summand of $\cH(G,\lambda)$, and the
composite $s_P\circ t_P$ induces an isomorphism
$$\cH(M,\lambda_M)^{W_{[M,\omega]}}\isoto \cH\bigl(G,\sigma(\tau)\bigr).$$
(The map $s_P$ is
just given by $h\mapsto e_K * h * e_K$, where $e_K$ is the idempotent mentioned
above.)

We summarise much of the preceding discussion as the following theorem.
\begin{thm}\label{thm: inertial local Langlands, N=0}If $\tau$ is an
  inertial type, then there is a finite-dimensional smooth irreducible
  $\Qpbar$-representation $\sigma(\tau)$ of $K$ such that
  if $\pi$ is any irreducible smooth
  $\Qpbar$-representation of $G$, then the restriction of $\pi$ to ${K}$ contains {\em (}an isomorphic
  copy of{\em )} $\sigma(\tau)$ as a subrepresentation if and only if
  $\rec_p(\pi)|_{I_F}\sim\tau$ and $N=0$ on
  $\rec_p(\pi)$. Furthermore, in this case the restriction of $\pi$ to
  ${K}$ contains a unique copy of $\sigma(\tau)$. The ring
  $\cH\bigl(G,\sigma(\tau)\bigr):=\End_G\bigl(\cInd_{K}^{G}\sigma(\tau)\bigr)$ is commutative.\end{thm}
 \begin{remark} We may think of this result as an \emph{inertial local Langlands correspondence} for $\GL_n$, in the potentially crystalline case. For $\GL_2/F$, the correspondence between $\tau$ and $\sigma(\tau)$ matches the \emph{inertial local Langlands correspondence} of Henniart (see the appendix to~\cite{breuil-mezard}), restricted to the potentially crystalline case.
 \end{remark}

\begin{prop}\label{BZ} Let $\pi_i$ be an irreducible supercuspidal
  representation of $\GL_{n_i}(F)$ for $1\le i\le r$, and
  $n=n_1+\cdots+ n_r$. If  $\pi_{j}\not\cong \pi_i(1)$  {\em (}the condition
  being empty if $n_i\neq n_j${\em )} for every $i<j$ then  the $G$-socle
  and the $G$-cosocle of $\pi_1\times \cdots\times \pi_r$ are
  irreducible.
  Moreover, the socle occurs as a subquotient with multiplicity one
  and is the only generic subquotient of $\pi_1\times \cdots\times
  \pi_r$.
\end{prop}
\begin{proof} This is an exercise in Bernstein--Zelevinsky theory
  \cite{BZ77ENS,Zel80}. We will make use of the material recalled
  from~\cite{rodier} at the beginning of this section. If we let
  $\Delta_1=\{\pi_1\}, \ldots, \Delta_r=\{\pi_r\}$, then by assumption
  the $\Delta_i$ are segments such that $\Delta_i$ does not precede
  $\Delta_j$ for $i< j$. Since the segments are of length one, we have
  $L(\Delta_i)\cong Z(\Delta_i)\cong \pi_i$ by definition, so that as
  recalled above $\pi_1\times\dots\times\pi_r$ has a unique
  irreducible subrepresentation $Z(\Delta_1, \ldots, \Delta_r)$ and a
  unique irreducible quotient $L(\Delta_1, \ldots, \Delta_r)$. In
  addition, $Z(\Delta_1, \ldots, \Delta_r)$ occurs as a subquotient
  with multiplicity one, so we only need to show that it is the unique
  generic subquotient. 

Let $U_n \subset \GL_n(F)$ be the subgroup of 
unipotent upper-triangular matrices,
and let $\theta_n: U_n\rightarrow \Qpbar^{\times}$ be the
character $(u_{ij})\mapsto \psi(\sum_{i=1}^{n-1} u_{i, i+1})$, where
$\psi: F\rightarrow \Qpbartimes$ is a fixed smooth non-trivial
character. If $\pi$ is a representation of $G$, we let
$\pi_{\theta_n}$ be the largest quotient of $\pi$ on which $U_n$ acts
by $\theta_n$.  If $\pi$ is an irreducible representation of $G$ then
the dimension of $\pi_{\theta_n}$ is at most one, and is equal to one
if and only if $\pi$ is generic. Since $U_n$ is equal to the union of
its compact open subgroups, the functor $\pi\mapsto \pi_{\theta_n}$ is
exact. Thus it is enough to show that $(\pi_1\times \cdots\times
\pi_r)_{\theta_n}$ is one dimensional and $Z(\Delta_1, \ldots,
\Delta_r)_{\theta_n}$ is non-zero.

In \cite{BZ77ENS,Zel80} the authors define a family of exact functors
$\pi\mapsto \pi^{(i)}$ for $0\le i\le n$ from the category of smooth
representations of $\GL_n(F)$ to the category of smooth
representations of $\GL_{n-i}(F)$. The representation $\pi^{(i)}$ is
called the $i$-th derivative of $\pi$. We have $\pi^{(0)}=\pi$ and
$\pi^{(n)}=\pi_{\theta_n}$. If $\pi$ is irreducible and supercuspidal
then $\pi^{(i)}=0$ for $0<i<n$ and $\pi^{(n)}$ is one dimensional,
by~\cite[Theorem 4.4]{BZ77ENS}. By~\cite[Corollary 4.6]{BZ77ENS} we have
that $(\pi_1\times\cdots\times\pi_r)^{(n)}\cong (\pi_1)^{(n_1)}\otimes
\cdots \otimes (\pi_{r})^{(n_r)}$, which is one dimensional, since the
$\pi_i$ are supercuspidal representations of $\GL_{n_i}(F)$. It
follows from~\cite[Theorem 6.2]{Zel80} that $Z(\Delta_1, \ldots,
\Delta_r)^{(n)}\neq 0$, as required.
\end{proof}

If $\pi$ in $\Omega$ is  irreducible, then the action of $\mathfrak
Z_{\Omega}$ on $\pi$ defines a $\Qpbar$-algebra morphism $\chi_{\pi}: \mathfrak
Z_{\Omega}\rightarrow \End_G(\pi)\cong \Qpbar$.  The following is a
strengthening of Theorem 4.1 of \cite{dat}.

\begin{prop}\label{stronger_dat} Let $\pi$ be an irreducible smooth
  $\Qpbar$-representation of $G$, such that $\rec_p(\pi)|_{I_F}\sim
  \tau$ and $N=0$ on $\rec_p(\pi)$. Then
$$\cInd_{K}^{G}\sigma(\tau)\otimes_{\mathfrak Z_{\Omega}, \chi_{\pi}} \Qpbar\cong \pi_1\times \cdots\times \pi_r,$$
where $\pi_i$ is a supercuspidal representation of $\GL_{n_i}(F)$,
such that if $i<j$ then $\pi_{j}\not\cong \pi_i(1)$ {\em (}the condition
being empty if $n_i\neq n_j${\em )}, and $\rec_p(\pi)\cong \oplus_{i=1}^r
\rec_p(\pi_i)$. Moreover, the representation $\pi$ is the unique irreducible quotient of
$\cInd_{K}^{G}\sigma(\tau)\otimes_{\mathfrak Z_{\Omega}, \chi_{\pi}} \Qpbar$.
\end{prop}
\begin{proof} Since $N=0$ on $\rec_p(\pi)$, we may write it as a
  direct sum of irreducible representations of the Weil group $W_F$,
  and there is a partition $n=n_1+\dots+n_r$ and supercuspidal
  representations $\pi_i$ of $\GL_{n_i}(F)$ such that
  $\rec_p(\pi)\cong \oplus_{i=1}^r \rec_p(\pi_i)$. After reordering we
  may assume that if $i<j$ then $\pi_{j}\not\cong \pi_i(1)$. It
  follows from Proposition \ref{BZ} that $\pi_1\times \cdots\times
  \pi_r$ has a unique irreducible quotient $\pi'$.

  Then $\rec_p(\pi')\cong \oplus_{i=1}^r \rec_p(\pi_i)$ as
  representations of $W_F$, and $N=0$ on $\rec_p(\pi')$ since all the
  segments have length one, as in the proof of Proposition
  \ref{BZ}. Thus $\rec_p(\pi')\cong \rec_p(\pi)$, which implies that
  $\pi\cong \pi'$. Since the socle of $\pi_1\times\cdots\times\pi_r$
  is irreducible and occurs as a subquotient with multiplicity one,
  the action of $\mathfrak Z_{\Omega}$ on
  $\pi_1\times\cdots\times\pi_r$ factors through an algebra morphism
  $\chi: \mathfrak Z_{\Omega}\rightarrow \Qpbar$. Since the cosocle is
  isomorphic to $\pi$, we deduce that $\chi=\chi_{\pi}$.

  Theorem \ref{thm: inertial local Langlands, N=0} implies that $\pi$
  is a quotient of $\cInd_K^G \sigma(\tau)$. Since $\cInd_K^G
  \sigma(\tau)$ is projective there exists a $G$-equivariant map
  $\varphi: \cInd_K^G \sigma(\tau)\rightarrow
  \pi_1\times\cdots\times\pi_r$ such that the composition with
  $\pi_1\times\cdots\times\pi_r\twoheadrightarrow \pi$ is
  surjective. The $G$-cosocle of the cokernel of $\varphi$ is zero.
  Since $\pi_1\times\cdots\times \pi_r$ is of finite length, so is the
  cokernel of $\varphi$, and we deduce that $\varphi$ is
  surjective. Since $\mathfrak Z_{\Omega}$ acts naturally on
  everything, $\varphi$ induces a surjection $\bar{\varphi}:
  \cInd_{K}^{G}\sigma(\tau)\otimes_{\mathfrak Z_{\Omega}, \chi_{\pi}}
  \Qpbar\twoheadrightarrow \pi_1\times \cdots\times \pi_r$.

  Furthermore, \cite[Theorem 4.1]{dat} implies that the semisimplification of
  $\cInd_{K}^{G}\sigma(\tau)\otimes_{\mathfrak Z_{\Omega}} \chi_{\pi}$
  is isomorphic to the semisimplification of $\pi_1'\times\cdots\times
  \pi_s'$, where $\pi'_i$ are supercuspidal representations of
  $\GL_{n_i'}(F)$ for some integers $n_i'$, such that
  $n=n_1'+\cdots+n_s'$. Since $\pi$ is an irreducible subquotient of
  both $\pi_1'\times\ldots \times \pi_s'$ and
  $\pi_1\times\cdots\times\pi_r$, \cite[Theorem 2.9]{BZ77ENS} implies
  that $\pi_1'\times\cdots \times \pi_s'$ and
  $\pi_1\times\cdots\times\pi_r$ have the same
  semisimplification. Thus
  $\cInd_{K}^{G}\sigma(\tau)\otimes_{\mathfrak Z_{\Omega}, \chi_{\pi}} \Qpbar$
  and $\pi_1\times \cdots\times \pi_r$ have the same
  semisimplification, which implies that $\bar{\varphi}$ is an
  isomorphism, as required.
 \end{proof}

\begin{cor}\label{generic_sub} Let $\Omega$ be a Bernstein component
  corresponding to an inertial type $\tau$ and let $\mathfrak
  Z_{\Omega}$ be the centre of $\Omega$. Let $\chi: \mathfrak
  Z_{\Omega}\rightarrow \Qpbar$ be a $\Qpbar$-algebra morphism. Then the $G$-socle of
  $\cInd_{K}^{G}\sigma(\tau)\otimes_{\mathfrak Z_{\Omega}, \chi} \Qpbar$ is
  irreducible and generic, and all the other irreducible subquotients
  are not
  generic.

  Conversely, if an irreducible representation $\pi$ in $\Omega$ is
  generic then $\pi$ is isomorphic to the $G$-socle of
  $\cInd_{K}^{G}\sigma(\tau)\otimes_{\mathfrak Z_{\Omega}, \chi_{\pi}}
  \Qpbar$.
 \end{cor}
 \begin{proof} The first part follows from Propositions \ref{BZ} and
   \ref{stronger_dat}. The converse may be seen as follows. There
   exist supercuspidal representations $\pi_i$ of $\GL_{n_i}(F)$ such
   that $n=n_1+\cdots+ n_r$ and $\pi$ is a subquotient of $\pi_1\times
   \cdots\times \pi_r$. If $w$ is a permutation of $\{1, \ldots, r\}$
   then $\pi_1\times \cdots\times \pi_r$ and $\pi_{w(1)}\times
   \cdots\times \pi_{w(r)}$ have the same semisimplification by
    \cite[Theorem 2.9]{BZ77ENS}, so we may assume that the $\pi_i$
   satisfy the conditions of Proposition \ref{BZ}. Since the socle of
   $\pi_1\times\cdots\times\pi_r$ is irreducible and occurs as a
   subquotient with multiplicity one, the action of $\mathfrak
   Z_{\Omega}$ on $\pi_1\times\cdots\times\pi_r$ factors through a
   maximal ideal, which is equal to $\chi_{\pi}$, as $\pi$ occurs as a
   subquotient. If we let $\pi'$ be the $G$-cosocle of
   $\pi_1\times\cdots\times\pi_r$ then $\pi'$ satisfies the conditions
   of Proposition \ref{stronger_dat}, and we have
   $\chi_{\pi'}=\chi_{\pi}$. The assertion follows from Propositions
   \ref{BZ} and \ref{stronger_dat}.
 \end{proof}

 \begin{cor}\label{cor: generic implies hecke iso} Let $\pi$ be an
   irreducible smooth generic $\Qpbar$-representation of $G$, such
   that $\rec_p(\pi)|_{I_F}\sim \tau$ and $N=0$ on $\rec_p(\pi)$. Then
   we have a natural isomorphism
   $\cInd_{K}^{G}\sigma(\tau)\otimes_{\mathfrak Z_{\Omega}, \chi_{\pi}}
   \Qpbar \cong \pi$.
\end{cor}
\begin{proof} By
  Corollary~\ref{generic_sub}, we see that the $G$-socle of
  $\cInd_{K}^{G}\sigma(\tau)\otimes_{\mathfrak Z_{\Omega}, \chi_{\pi}}\Qpbar$
  is isomorphic to $\pi$ and occurs with multiplicity one as a
  subquotient. Theorem \ref{thm: inertial local Langlands, N=0}
  implies that $\pi$ is a quotient of
  $\cInd_{K}^{G}\sigma(\tau)\otimes_{\mathfrak Z_{\Omega},
  \chi_{\pi}}\Qpbar$. This implies the assertion.
\end{proof}

\subsection{Rationality}\label{god_save_the_queen} As a preparation for the next section we explain in this subsection how the
results above remain true with a finite extension $E$ of $\Q_p$ as coefficient field,
as long as $E$ is sufficiently large (depending on the  Bernstein component
$\Omega$). We do this by following various parts of the construction of the
Bernstein centre in~\cite{renard}, working over $E$ rather than over $\Qpbar$,
and we then deduce the results from those over $\Qpbar$ by faithfully flat
descent. Let $(M, \omega)$ be the supercuspidal support of some irreducible representation in $\Omega$, let $\mathcal X(M)$ be the group of unramified
characters $\chi: M\rightarrow \Qpbartimes$, and let
$$ \mathcal X(M)(\omega):=\{ \chi\in \mathcal X(M): \omega \cong \omega \otimes \chi\}.$$
Let $M^0$ be the intersection of the kernels of the characters
$\chi\in\cX(M)$, and let $T$ be the intersection of the kernels of the
$\chi\in \mathcal X(M)(\omega)$.
The restriction of $\omega$ to
$M^0$ is a finite direct sum of irreducible representations, see
\cite[VI.3.2]{renard}. We fix one irreducible summand $\rho$. It
follows from Lemme VI.4.4 of \cite{renard} and its proof that $\rho$ extends to  a representation $\rho_T$ of $T$, such that  $\Ind_T^M \rho_T$
is isomorphic to a finite direct sum of copies of $\omega$. Thus $\chi\in \mathcal X(M)$ lies in $\mathcal X(M)(\omega)$ if and only if the restriction of $\chi$ to $T$ is trivial.
Thus the restriction to $T$ induces a bijection
\numequation\label{restrict_to_T}
 \mathcal X(M)/\mathcal X(M)(\omega)\overset{\cong}{\rightarrow} \mathcal X(T),
 \end{equation}
where $\mathcal X(T)$ is the group of characters from $T$ to $\Qpbartimes$, which are trivial on $M^0$.

Let $D$ be the Bernstein component (for $M$) containing $\omega$,  let $\mathfrak Z_D$ be the centre of $D$ and let $\Pi(D):= \cInd_{M^0}^M \rho$. It is shown in \cite[VI.4.1]{renard}
that $\Pi(D)$ is a projective generator for $D$.   Thus we may identify $\mathfrak Z_D$ with the centre of the ring
$\End_M\bigl(\Pi(D)\bigr)$. Since $\rho$ is irreducible, $\Pi(D)$  is a finitely generated $\Qpbar[M]$-module.

Let $\mathfrak Z_{\Omega}$ be the Bernstein centre of $\Omega$ and let $\Pi(\Omega):= \nind_P^G \Pi(D)$, where $P$ is any parabolic subgroup with Levi subgroup $M$.
It is shown in \cite[Thm.VI.10.1]{renard} that $\Pi(\Omega)$ is a projective generator of $\Omega$, which is a finitely generated $\Qpbar[G]$-module.  Thus we may identify $\mathfrak Z_{\Omega}$ with the centre of the ring $\End_G\bigl(\Pi(\Omega)\bigr)$.

It follows from Bushnell--Kutzko theory that $\omega\cong \cInd^M_{\mathbf{J}}{\Lambda}$, and  $\rho\cong \cInd_J^{M^0} \lambda$, where ${\mathbf{J}}$ is an open compact-mod-centre
subgroup of $M$, $J$ is an open compact subgroup of $M$, and
$\Lambda$, $\lambda$ are (necessarily) finite-dimensional irreducible representations.  We may realise both $\Lambda$ and $\lambda$ over a finite extension $E$ of $\Qp$. By compactly inducing these realisations,
we deduce that both $\omega$ and $\rho$ can be realised over $E$. We denote these representations by $\omega_E$ and $\rho_E$,  respectively.
It is shown in~\cite[Lemme V.2.7]{renard} that $\mathcal X(M)(\omega)$ is a
finite group. Let $W(D)$ be the subgroup of $N_G(M)/M$
stabilising $D$. For each $w\in W(D)$ there are precisely $|\mathcal X(M)(\omega)|$ unramified characters $\xi$ such that $\omega^{w}\cong  \omega \otimes \xi$. Since the group
$M/M^0$ is finitely generated, by replacing $E$ with a finite extension, we may assume that all the characters $\xi$ are $E$-valued. By further
enlarging $E$ we may assume that  $\sqrt{q}$, where $q$ is the number of elements in the residue field of $F$, is contained in $E$. Then the modulus character of $P$ is defined over $E$.

Let $\Pi(D)_E:=\cInd_{M^0}^M \rho_E,$
and let $\mathfrak Z_{D, E}$ denote the centre of
its endomorphism ring $\End_M\bigl(\Pi(D)_E\bigr)$.
Since $\rho_E\otimes_E\Qpbar \cong \rho$, we have
$\Pi(D)_E\otimes_E \Qpbar\cong \Pi(D)$. We may express the generators of $\Pi(D)$ (as a $\Qpbar[M]$-module) as a finite $\Qpbar$-linear combination of elements of $\Pi(D)_E$.
The $E[M]$-submodule of $\Pi(D)_E$ generated by these elements has to equal $\Pi(D)_E$, as the quotient is zero once we extend the scalars to $\Qpbar$. In particular, $\Pi(D)_E$
is a finitely generated $E[M]$-module.

Let $\Pi(\Omega)_E:= \nind_P^G \Pi(D)_E$ and let $\mathfrak Z_{\Omega, E}$ be the centre of $\End_G(\Pi(\Omega)_E)$. The smooth parabolic induction commutes with
$\otimes_E \Qpbar$, as the set $P\backslash G/ H$ is finite for every
open  subgroup $H$ of $G$ and tensor products commute with inductive
limits, so $\Pi(\Omega)_E\otimes_E \Qpbar\cong\Pi(\Omega)$. Since $\Pi(\Omega)$ is a finitely generated $\Qpbar[G]$-module, arguing as in the previous paragraph we deduce that $\Pi(\Omega)_E$ is  a finitely generated $E[G]$-module.

The following observation (see for example  Lemma~5.1
of~\cite{paskunasimage}) is very useful. If $\pi$ and $\pi'$ are representations of some group $\mathrm G$ on $E$-vector spaces, such that $\pi$ is a finitely generated $E[\mathrm G]$-module, then
\numequation \label{base_change_homs}
\Hom_{E[\mathrm G]}(\pi, \pi')\otimes_E \Qpbar  \cong \Hom_{\Qpbar[\mathrm G]}(\pi\otimes_E \Qpbar, \pi'\otimes_E \Qpbar).
\end{equation}
It follows from \eqref{base_change_homs} that
\numequation\label{base_change_D}
\End_M(\Pi(D)_E)\otimes_E \Qpbar\cong \End_M\bigl(\Pi(D)\bigr),
\end{equation}
\numequation\label{base_change_Omega}
\End_G(\Pi(\Omega)_E)\otimes_E \Qpbar \cong \End_G\bigl(\Pi(\Omega)\bigr).
\end{equation}

Let $D_E$ be the full subcategory of smooth representation $\omega'$ of $M$ on $E$-vector spaces, such that $\omega'\otimes_E \Qpbar$ is in $D$. It follows from
\eqref{base_change_homs} that $\Hom_M(\Pi(D)_E, \omega')\otimes_E \Qpbar\cong \Hom_M(\Pi(D), \omega'\otimes_E\Qpbar)$. Since $\Pi(D)$ is a projective generator of $D$, we
deduce that $\Pi(D)_E$ is a projective generator of $D_E$. (This follows from
the fact that $\Qpbar$ is
  faithfully flat over $E$; we will repeatedly use this fact below without
  further comment.)
   In particular, $\mathfrak Z_{D, E}$ is naturally isomorphic to the centre of the category~$D_E$.

Similarly we let $\Omega_E$  be the full subcategory of smooth representations $\pi'$ of $G$ on $E$-vector spaces, such that $\pi'\otimes_E \Qpbar$ is in $\Omega$.
The same argument as above gives that $\Pi(\Omega)_E$ is a projective generator of $\Omega_E$ and $\mathfrak Z_{\Omega, E}$ is naturally isomorphic to the centre
of $\Omega_E$.

\begin{lem}\label{central} Let $\mathcal A$ be an $E$-algebra and let $\mathcal Z$ be an $E$-subalgebra of $\mathcal A$. If $\mathcal Z\otimes_E \Qpbar$ is the
centre of $\mathcal A\otimes_E\Qpbar$ then $\mathcal Z$ is the centre of $\mathcal A$.
\end{lem}
\begin{proof}  For each $z\in \mathcal Z$, we define an $E$-linear map $\phi_z:\mathcal A\rightarrow \mathcal A$,
$a\mapsto a z-za$. Since $z\otimes 1$ is central in $\mathcal A\otimes_E {\Qpbar}$ by assumption, we deduce that
$(\Im \phi_z)\otimes_E \Qpbar=0$, which implies that $\Im \phi_z=0$. We deduce that $\mathcal Z$ is contained in $\mathcal Z(\mathcal A)$, the centre
of $\mathcal A$. If $z\in \mathcal Z(\mathcal A)$ then $z\otimes 1$ is contained in the centre of $\mathcal A\otimes_E {\Qpbar}$ and thus in $\mathcal Z\otimes_E \Qpbar$ by assumption.
Hence $(\mathcal Z(\mathcal A)/\mathcal Z) \otimes_E \Qpbar = 0$, which implies that $\mathcal Z=\mathcal Z(\mathcal A)$.
\end{proof}

\begin{lem}\label{extend_scalars_1}The isomorphism \eqref{base_change_D} induces an isomorphism $\mathfrak Z_{D,E}\otimes_E \Qpbar \cong \mathfrak Z_D$.
\end{lem}
\begin{proof} Since $\Pi(D)\cong \Ind_T^M (\cInd^T_{M^0} \rho)$ and induction is a functor, we have an inclusion
$\End_T(\cInd^T_{M^0} \rho)\subset \End_M\bigl(\Pi(D)\bigr)$. Now~\cite[Thm.VI.4.4]{renard} implies that this inclusion identifies
$\End_T(\cInd^T_{M^0} \rho)$ with the centre of $\End_M\bigl(\Pi(D)\bigr)$. The assertion follows from Lemma \ref{central} applied to
$\mathcal Z=\End_T(\cInd^T_{M^0} \rho_E)$ and $\mathcal A=\End_M\bigl(\Pi(D)_E\bigr)$.
\end{proof}

Let $\Irr(D)$ be the set of irreducible representations in $D$. Every such irreducible representation is of the form $\omega\otimes \chi$ for some
$\chi\in \mathcal X(M)$. We thus have a bijection $\mathcal X(M)/\mathcal X(M)(\omega)\overset{\cong}{\rightarrow} \Irr(D)$, $\chi \mapsto \omega \otimes \chi$. Composing this bijection with
\eqref{restrict_to_T} we obtain a bijection
\numequation\label{biject_yeah}
\Irr(D)\overset{\cong}{\rightarrow} \mathcal X(T).
\end{equation}
Now $\mathcal X(T)$ is naturally isomorphic to the set of homomorphisms of $\Qpbar$-algebras from $\Qpbar[T/M^0]$ to $\Qpbar$.  It is explained in  \cite[VI.4.4]{renard} that we have identifications
$$\mathfrak Z_D\cong \End_T(\Ind_{M^0}^T \rho)\cong \Qpbar[T/M^0];$$
see Th\'eor\`eme VI.4.4 for the first isomorphism and Proposition VI.4.4 for the second, so that \eqref{biject_yeah} induces  a natural bijection between $\Irr (D)$ and $\MaxSpec \mathfrak Z_D$.

The group $W(D)$ acts on $\Irr(D)$ by conjugation. For each $w\in
W(D)$ let $\xi\in \mathcal X(M)$ be any character such that $\omega^{w}\cong \omega\otimes\xi$, and let $\xi_w$ be the
restriction of $\xi$ to $T$. If $\xi_1$ and $\xi_2$ are two such characters then $\omega \otimes \xi_1\cong \omega^w\cong \omega\otimes\xi_2$, and
hence $\xi_1\xi_2^{-1}$ lies in $\mathcal X(M)(\omega)$. It follows from the definition of $T$ that the restriction of $\xi_1 \xi_2^{-1}$ to $T$ is trivial.
Thus $\xi_w$ depends only on  $w$ and not
on the choice of $\xi$. If $\chi\in \mathcal X(M)$ then $(\omega\otimes \chi)^w\cong \omega \otimes \chi^{w}\xi$. Thus the action of $W(D)$ on $\mathcal X(T)$ via \eqref{biject_yeah} is given by
$w. \chi= \chi^w \xi_w$. If we identify $\mathcal X(T)$ with the maximal spectrum of  $\Qpbar[T/M^0]$ then
this action is induced by the action of $W(D)$ on $\Qpbar[T/M^0]$ by $\Qpbar$-linear automorphisms given on the basis elements by $w. (t M^0)= \xi_w^{-1}(t) t^w M^0$:
if $\chi: \Qpbar[T/M^0]\rightarrow \Qpbar$ is a morphism of $\Qpbar$-algebras then $(w.\chi)(tM^0)=\chi( w^{-1} .(t M^0))= \chi( \xi_{w^{-1}}^{-1}(t) w^{-1}t w M^0)=
\xi_{w}(t) \chi^w(tM^0)$.

\begin{lem}\label{preserve_act} The action of $W(D)$ on $\mathfrak Z_D$ preserves $\mathfrak Z_{D,E}$.
\end{lem}
\begin{proof} Since  $\omega$ and $\rho$  can both be defined over $E$,
  so can the representation $\rho_H$ defined at the beginning of
  \cite[VI.4.4]{renard}, and in particular so can $\rho_T$,
its restriction to $T$. Hence,
if we identify $\mathfrak Z_D$ with $\Qpbar[T/M^0]$ as in  \cite[Prop.VI.4.4]{renard} then $\mathfrak Z_{D,E}$ is identified with $E[T/M^0]$. Since the characters $\xi_w$ are $E$-valued by
the choice of $E$, we get the assertion.
\end{proof}

\begin{lem}\label{invariants_W(D)} $\mathfrak Z_{\Omega, E}= \mathfrak Z_{D, E}^{W(D)}$.
\end{lem}
\begin{proof} Since $\Pi(\Omega)=\nind_P^G \Pi(D)$ and parabolic induction is a functor,  we have an inclusion $\mathfrak Z_D \subset \End_G\bigl(\Pi(\Omega)\bigr)$.
It follows from the discussion immediately preceding the proof of
Theorem VI.10.4 of~\cite{renard} that
this inclusion identifies $\mathfrak Z_{D}^{W(D)}$ with the centre of $\End_G\bigl(\Pi(\Omega)\bigr)$. The assertion follows from Lemma \ref{central} applied to
$\mathcal Z=\mathfrak Z_{D, E}^{W(D)}$ and $\mathcal A= \End_G\bigl(\Pi(\Omega)_E\bigr)$.
\end{proof}

\begin{prop}\label{extend_scalars_2} The isomorphism
  \eqref{base_change_Omega} induces an isomorphism
  $$\mathfrak Z_{\Omega,E}\otimes_E \Qpbar \cong \mathfrak Z_\Omega.$$
  \end{prop}
\begin{proof} Using Lemmas~\ref{extend_scalars_1} and~\ref{invariants_W(D)} we obtain $$\mathfrak Z_{\Omega,E}\otimes_E \Qpbar \cong
\mathfrak Z_{D, E}^{W(D)}\otimes_E \Qpbar\cong (\mathfrak Z_{D, E}\otimes_E \Qpbar)^{W(D)}\cong \mathfrak Z_D^{W(D)}\cong \mathfrak Z_{\Omega},$$
where the last isomorphism follows from \cite[VI.10.4]{renard}, as in the proof of Lemma~\ref{invariants_W(D)}.
\end{proof}

\begin{lem}\label{wow_yeah_chenevier_yeah} $\mathfrak Z_{\Omega, E}$ coincides with the ring $E[\mathcal B]$ constructed in {\em \cite[Prop.~3.11]{chenevier}}.
\end{lem}
\begin{proof} Let $\Delta$ be  the subgroup of $\mathcal X(M)\rtimes W(D)$ consisting of pairs $(\xi, w)$, such that $\omega^{w}\cong \omega \otimes \xi$. This subgroup
acts naturally on $E[M/M^0]$.  The map $\xi\mapsto (\xi, 1)$ identifies $\mathcal X(M)(\omega)$ with a normal subgroup of
$\Delta$ and the quotient is isomorphic to $W(D)$. We have
$$\Qpbar[M/M^0]^{\Delta}\cong (\Qpbar[M/M^0]^{\mathcal X(M)(\omega)})^{W(D)}\cong \Qpbar[T/M^0]^{W(D)}\cong \mathfrak Z_D^{W(D)},$$
see \cite[Rem.VI.4.4]{renard} for the second isomorphism. Chenevier defines $E[\mathcal B]$ to be $E[M/M^0]^{\Delta}$. This subring gets identified with
$E[T/M^0]^{W(D)}$ inside $\Qpbar[T/M^0]^{W(D)}$, and with $\mathfrak Z_{D,E}^{W(D)}$ inside $\mathfrak Z_D^{W(D)}$, see the proof of Lemma \ref{preserve_act}.
The assertion follows from Lemma \ref{invariants_W(D)}.
 \end{proof}

 Let $\sigma(\tau)$ be the representation of $K$ given by Theorem \ref{thm: inertial local Langlands, N=0}. After replacing $E$ by a finite extension we may assume that there
 exists a representation $\sigma(\tau)_E$ of $K$ on an $E$-vector space, such that $\sigma(\tau)_E\otimes_E \Qpbar\cong \sigma(\tau)$. Then
 $\cInd_K^G \sigma(\tau)_E$ is an object in $\Omega_E$. Since $\mathfrak Z_{\Omega, E}$ is the centre of $\Omega_E$ it acts on  $\cInd_K^G \sigma(\tau)_E$, thus inducing a
 homomorphism $\mathfrak Z_{\Omega, E}\rightarrow \End_G\bigl(\cInd_K^G \sigma(\tau)_E\bigr)$.
 \begin{lem}\label{lem: rational Bernstein Hecke iso} The map $\mathfrak Z_{\Omega, E}\rightarrow \End_G\bigl(\cInd_K^G \sigma(\tau)_E\bigr)$ is an isomorphism.
 \end{lem}
 \begin{proof} It follows from  Theorem 4.1 of \cite{dat} and
   Proposition \ref{extend_scalars_2} above that the map is an isomorphism once we extend scalars to $\Qpbar$. This implies the assertion.
 \end{proof}

 Let $\mathcal R:=\End_G\bigl(\Pi(\Omega)\bigr)$. Since $\Pi(\Omega)$ is a projective generator the functors $M\mapsto M\otimes_{\mathcal R} \Pi(\Omega)$ and
$\pi\mapsto \Hom_G(\Pi(\Omega), \pi)$ induce an equivalence of categories between the category of right $\mathcal R$-modules and $\Omega$.
If $\pi$ is irreducible, then the action of $\mathfrak Z_{\Omega}$ on $\pi$ factors through $\chi_{\pi}: \mathfrak Z_{\Omega}\rightarrow \Qpbar$.
It follows from \cite[Lem.VI.10.4]{renard} that $\mathcal R$ is a finitely generated $\mathfrak Z_{\Omega}$-module , which implies that the module corresponding to $\pi$ is a finite dimensional
$\Qpbar$-vector space.  Since $\mathfrak Z_{\Omega, E}$ is a finitely generated algebra over $E$, $E(\chi_\pi):=\chi_{\pi}(\mathfrak Z_{\Omega, E})$ is a finite extension of $E$.

In the above $E$ was only required to be sufficiently large. Thus if $E'$ is a subfield of $\Qpbar$ containing $E$, then we let $\Omega_{E'}$, $\Pi(\Omega)_{E'}$
 be the corresponding objects defined over $E'$ instead of $E$. Then $\Pi(\Omega)_{E'}$ is a projective generator of $\Omega_{E'}$
 and the functors $M\mapsto M\otimes_{\mathcal R_{E'}} \Pi(\Omega)_{E'}$ and
$\pi\mapsto \Hom_G(\Pi(\Omega)_{E'}, \pi)$ induce an equivalence of categories between the category of right $\mathcal R_{E'}$-modules and $\Omega_{E'}$, where
$\mathcal R_{E'}:=\End_G(\Pi(\Omega)_{E'})$.

 \begin{lem} Every irreducible generic  $\pi\in \Omega$  can be realised over $E(\chi_{\pi})$.
 \end{lem}
 \begin{proof} In order to ease the notation,
	 we write $E':=E(\chi_{\pi})$ and 
 $\pi':= \cInd_K^G \sigma(\tau)_E \otimes_{\mathfrak Z_{\Omega, E}} E'$. Then
 $$\pi'\otimes_{E'} \Qpbar\cong \cInd_K^G \sigma(\tau) \otimes_{\mathfrak Z_{\Omega}, \chi_{\pi}}\Qpbar \cong \pi_1\times\ldots\times  \pi_r,$$
 where the last isomorphism is given by Proposition \ref{stronger_dat}. Hence, $\pi'\otimes_{E'} \Qpbar$ is of finite length, which implies that
 $\Hom_G(\Pi(\Omega), \pi'\otimes_{E'}\Qpbar)$ is a finite dimensional $\Qpbar$-vector space, which implies that
 $M':=\Hom_G(\Pi(\Omega)_E, \pi')$ is a finite dimensional $E'$-vector space.

If $\pi''$ is an irreducible $E'$-subrepresentation
 of $\pi'$, and if we define $M'':= \Hom_G(\Pi(\Omega)_{E'}, \pi'')$,
then $M''$  is an irreducible $\mathcal R_{E'}$-module which is finite dimensional over~$E'$.
It follows from \cite[Cor.12.7.1a)]{bourbaki8} that $M''\otimes_{E'}\Qpbar$ is a semi-simple $\mathcal R$-module. Hence,
$\pi''\otimes_{E'} \Qpbar$ is a semi-simple $G$-representation. Proposition \ref{BZ}  implies that the $G$-socle of $\pi''\otimes_{E'} \Qpbar$
 is irreducible and is isomorphic to $\pi$. Thus $\pi''\otimes_{E'} \Qpbar\cong \pi$.
  \end{proof}

Henceforth for each Bernstein component $\Omega$ we will fix a sufficiently
large $E$ as above and work with it. Agreeing on this, we will \emph{omit $E$
  from the notation} when there is no danger of confusion. For instance we will
write $\mathfrak{Z}_{\Omega}$, $\sigma(\tau)$, and so on, in place of
$\mathfrak{Z}_{\Omega,E}$, $\sigma(\tau)_E$ and so on. Note that
we fixed a choice of $E$ in Section~\ref{sec:patching}; however, it is
harmless to replace our patched module $M_\infty$ with its base
extension to the ring of integers in any larger choice of $E$, and we will do so without
further comment.

\section{Local-global compatibility}\label{sec:local-global compatibility}

The goal of this section is to prove that the patched module
$M_\infty$ satisfies local-global compatibility, in the following
sense: the $G$-action on $M_\infty$ (obtained by patching global
objects) will induce a tautological Hecke action on certain patched
modules for particular $K$-types. On the other hand, we will define a
second Hecke action via an interpolation of the classical local Langlands correspondence. We will
then prove that these two Hecke actions coincide. The details are made
explicit below.

Note that it is plausible that $M_\infty$
should satisfy local-global compatibility, since it is patched
together from spaces of algebraic modular forms; the difficulty in
proving this is that the modules at finite level are all $p$-power
torsion, while local-global compatibility is usually defined after
inverting $p$, so that we need to establish some integral control over
the compatibility. Some of our arguments were inspired by the treatment
of the two-dimensional crystalline case in~\cite[\S 3.6]{MR2392362},
and somewhat related considerations in the arguments of~\cite{kisinfmc}.

Let $\sigma$ be a locally algebraic type for $G=\GL_n(F)$ defined over
$E$. Then by definition $\sigma$ is an absolutely irreducible representation of
$K=\GL_n(\cO_F)$ over $E$ of the form $\sigma_\sm \otimes \sigma_\alg$,
where $\sigma_\sm$ is a smooth type for $K$
(i.e.\ $\sigma_\sm=\sigma(\tau)$ for some inertial type $\tau$) and
$\sigma_\alg$ is the restriction to $K$ of an irreducible algebraic
representation of $\Res_{F/\Qp}\GL_n$; we will sometimes also write
$\sigma_\alg$ for the corresponding $G$-representation. (So, all of our locally algebraic types are ``potentially
crystalline'', in the sense that they detect representations for which
$N=0$.) Set $\cH(\sigma):=\End_G(\cInd_K^G\sigma)$.

We say that a continuous representation $r:G_F\to \GL_n(E)$
has Hodge--Tate weights prescribed by $\sigma_\alg$ if $r$ is regular of weight $\xi$ and $\sigma_\alg$ is \emph{dual} to (the restriction to $K$ of) the representation with highest weight vector $\xi$. (Given such an $r$, the representation $\sigma_\alg$ is the restriction to $K$ of the representation $\pi_\alg$ defined in Section~\ref{subsec:notation}). We will say that $r$ is \emph{potentially crystalline of type $\sigma$} if it is potentially
crystalline with inertial type $\tau$ and Hodge--Tate weights prescribed by $\sigma_\alg$. We also say that a global representation
has type $\sigma$ if it restricts to such an~$r$. Let
$R_{\tp}^\square(\sigma)$ be the local universal lifting
ring of type $\sigma$ at $\tp$ (i.e.\ the unique reduced and
$p$-torsion free quotient of $R_{\tp}^\square$ corresponding to
potentially crystalline lifts of type $\sigma$).

Let $\CX = \Spf R_{\tp}^\square(\sigma)$, with ideal of definition
taken to be the maximal ideal, and let $\CX^\rig$ denote its rigid
generic fibre (as constructed in ~\cite[\S 7]{deJ}). Note that
$\CX^\rig =\cup_j U_j$ is an increasing union of affinoids, and in fact is a
quasi-Stein rigid space, since it is a closed subspace of an open polydisc, which is an increasing union of
closed polydiscs. By a standard
abuse of notation, we will write $\cO_{\cX^\rig}$ for  the ring of rigid-analytic functions on $\CX^\rig$. Then $\cO_{\cX^\rig} =
\varprojlim_j \Gamma(U_j,\cO_{U_j})$ and we equip it with the inverse
limit topology.
We note
that by~\cite[Lemma 7.1.9]{deJ}, there is a bijection between the
points of $\CX^\rig$ and the closed points of $\Spec
R_{\tp}^\square(\sigma)[1/p]$. The universal lift over
$R_{\tp}^\square(\sigma)$ gives rise to a  continuous family of
representations $\rho^{\rig}: G_F \to \GL_n(\CO_{\CX^\rig})$. (The
continuity of $\rho^{\rig}$ is equivalent to that of each of the representations
$G_F \to \GL_n\bigl(\Gamma(U_j,\cO_{U_j})\bigr)$ obtained by
restricting elements of $\GL_n(\CO_{\CX^{\rig}})$ to~$U_j$.) If $x$
is a point of $\CX^\rig$ with residue field $E_x$, we denote by
$\rho_x:G_F\to \GL_n(E_x)$ the specialisation of $\rho^{\rig}$ at
$x$. We define the locally algebraic $G$-representation $\pi_{\lalg,x}:=
\pi_{\sm}(\rho_x)\otimes_E
\sigma_{\alg}=\pi_{\sm}(\rho_x)\otimes_{E_x} \pi_{\alg}(\rho_x)$.
 (Recall the notation $\pi_{\sm}(\rho_x)$ and $\pi_{\alg}(\rho_x)$ from
\S\ref{subsec:notation}; in particular, $\pi_\sm(\rho_x)=r_p^{-1}(\WD(\rho_x)^{F-\semis})$.) Note that $\cH(\sigma)$ acts via a character
on the space $\Hom_K(\sigma,\pi_{\lalg,x})$, the latter being
one-dimensional (by Theorem~\ref{thm: inertial local Langlands, N=0}
together with the argument of~\cite[Lemma
1.4]{MR2290601}). 

The following theorem, which may be of independent interest, gives our
interpolation of the local Langlands correspondence. Its proof will occupy much
of this section.
\begin{thm}
  \label{thm: algebraic local Langlands}There is an $E$-algebra homomorphism
 $$ \eta:\cH(\sigma)\to
  R_{\tp}^\square(\sigma)[1/p]$$ which interpolates the local Langlands
  correspondence $r_p$. More precisely, for any closed point $x$ of $\Spec
  R_{\tp}^\square(\sigma)[1/p]$, the $\cH(\sigma)$-action on
  $\Hom_K(\sigma,\pi_{\lalg,x})$ factors as $\eta$ composed with the evaluation
  map $R_{\tp}^\square(\sigma)[1/p]\to E_x$.
\end{thm}

We begin by proving the following weaker result, showing the existence of a
rigid analytic local Langlands map.

\begin{prop}\label{prop: rigid analytic local Langlands} There is an $E$-algebra
  homomorphism $\eta:\cH(\sigma)\to \CO_{\CX^{\rig}}$ which interpolates the
  local Langlands correspondence $r_p$. More precisely, for any point $x\in
  \CX^{\rig}$, the action of
$\cH(\sigma)$ on  $\Hom_K(\sigma,\pi_{\lalg,x})$ factors as
$\eta$ composed with the evaluation map $\CO_{\CX^\rig}\to E_x$. \end{prop}

Recall that $\CX^\rig =\cup_j U_j$ can be written as an increasing union of rigid spaces associated to reduced affinoid algebras.
Th\'eor\`eme C of~\cite{MR2493221} associates a family of Weil-Deligne representations to a family of Galois representations over the rigid space associated to a reduced affinoid algebra. Applying it to each $U_j$ and to $\rho^{\mathrm{rig}}|_{U_j}$, we obtain a
compatible family of Weil--Deligne representations
$\rho_{\WD}^\rig:W_F \to \GL_n(\Gamma(U_j,\cO_{U_j}))$ and thus a
Weil--Deligne representation $\rho_{\WD}^\rig : W_F \to \GL_n(\cO_{\CX^\rig})$.
 Note that $\rho_{\WD}^{\mathrm{rig}}$ has $N=0$. 

 For a point $x$ of
$\mathcal{X}^{\mathrm{rig}}$, we denote by $\rho_{\WD,x}$ the
specialisation of $\rho_{\WD}^{\mathrm{rig}}$ at $x$. Then $\rho_{\WD,x}|_{I_F}\simeq \tau$ for all points $x$ of $\cX^\mathrm{rig}$. Recall that
$\mathfrak{Z}_{\Omega}$ is the Bernstein centre for the Bernstein
component $\Omega$ corresponding to $\sigma(\tau)$.

\begin{prop}\label{prop:Bernstein-to-Xrig} There exists a unique $E$-algebra map
  $I:\mathfrak{Z}_{\Omega}\to\mathcal{O}_{\mathcal{X}^{\mathrm{rig}}}$
  such that for any point $x$ of $\mathcal{X}^{\mathrm{rig}}$ with
  residue field $E_{x}$, the smooth $G$-representation $\pi_{x}$
  corresponding to $\rho_{\WD,x}$ via the
  local Langlands correspondence $\rec_p$ determines via specialisation the map $x\circ I:\mathfrak Z_\Omega\to E_x$.
\end{prop}

\begin{proof} Consider the following map, obtained by specialisation:
  $$\gamma_{G}:\mathfrak{Z}_{\Omega}\to\prod_{x\in\mathcal{X}^{\mathrm{rig}}}E'_{x},$$
  where $\gamma_{G}$ is defined on the factor corresponding to $x$ by evaluating
  $\mathfrak{Z}_{\Omega}$ at the closed point in the Bernstein component
  $\Omega$ determined via local Langlands by $x$, and $E'_x/E_x$ is a
  sufficiently large finite extension. 

  Consider as well the following map, also obtained by
  specialisation: \[\gamma_{\WD}:\mathcal{O}_{\mathcal{X}^{\mathrm{rig}}}\to\prod_{x\in\mathcal{X}^{\mathrm{rig}}}E'_{x}.\]
  This is an injection since $\cX^\rig$ is reduced and since each
  $\Gamma(U_j,\cO_{U_j})$ is Jacobson. (The Jacobson property is true
  of any affinoid algebra. To see that $\cX^\rig$ is reduced, it is
  enough to check it on completed local rings at closed points, but
  these are the same as the completed local rings of
  $R_{\tilde{\mathfrak{p}}}^{\square}(\sigma)\left[1/p\right]$
  by~\cite[Lemma 7.1.9]{deJ}.
  The latter is reduced (being a localisation of $R_{\tilde{\mathfrak{p}}}^{\square}(\sigma)$, which is reduced by definition)
and excellent, since
  $R_{\tilde{\mathfrak{p}}}^{\square}(\sigma)$ is a complete, local,
  noetherian ring (and thus excellent by \cite[Scholie 7.8.3(iii)]{ega-4-ii}). The reducedness of the completed local rings now
  follows from~\cite[Scholie 7.8.3(v)]{ega-4-ii}.)

   In order to define our map
  $I$, it suffices to show that the image of $\mathfrak{Z}_{\Omega}$ under
  $\gamma_{G}$ is contained in the image of $\gamma_{\WD}$. Let
  $T:W_{F}\to\mathfrak{Z}_{\Omega}$ be the pseudo-representation constructed in
  Proposition~3.11 of~\cite{chenevier}. (Note that Chenevier's
  $E[\mathcal{B}]$ is our $\fkZ_{\Omega}$ by Lemma~\ref{wow_yeah_chenevier_yeah}.) By the construction of $T$, we have
  $\gamma_{G}\circ
  T=\gamma_{\WD}\circ\mathrm{tr}(\rho_{\WD}^{\mathrm{rig}})$. Therefore the
  proof of the proposition is reduced to Lemma \ref{lem: image of Chenevier's
    map} below.
\end{proof}

Write $v:W_F\twoheadrightarrow\Z$ for the valuation map assigning $1$ to any lift of the geometric Frobenius. Let $\phi \in W_F$ be an element of valuation $1$. For $w\in W_F$ and any $I_F$-representation $r_0$, let $r_0^{w}$ be the twist $r_0^{w}(\gamma):=r_0(w^{-1}\gamma w)$.

\begin{lem}\label{lem: structure of irred $W_F$-rep} Let
  $r$ be an irreducible continuous representation of $W_{F}$ over $\overline{\mathbb{Q}}_p$.
\begin{enumerate}

  \item The restriction $r|_{I_F}$  decomposes as a direct sum of non-isomorphic
  irreducible $I_F$-representations $\oplus_{i=1}^f r_{1}^{\phi^i}$ for some
  integer $f\geq1$. If $t\in \Z$ then $r(\phi^t)$
    respects the decomposition {\em (}i.e.\ $r(\phi^t)$ sends $r_1^{\phi^i}$
    into itself for $1\le i\le f${\em )} exactly when $f\mid t$. 

  \item We have $\mathrm{tr}\bigl(r(w)\bigr)\not=0$ for some $w\in W_F$ of valuation $t$ if and only if $f\mid t$.

  \item The unramified characters $\chi$ of $W_{F}$ satisfying
  $r\otimes\chi\simeq r$ are exactly the characters of order
  dividing $f$.\end{enumerate}
 \end{lem}

\begin{proof}
  (1) The representation $r|_{I_F}$ factors through a finite quotient
  $I_F/H$, so it decomposes as a direct sum of irreducible
  $I_F$-representations $\oplus_{i=1}^{f}r_{i}$, for some integer
  $f\geq1$. The fact that $r$ is irreducible as a $W_F$-representation
  implies that $r(\phi)$ acts transitively on (the representation
  spaces of) the $r_{i}$. Up to reordering the $r_{i}$, we may assume
  that it sends $r_{i}$ to $r_{i+1}$, where
  $r_{f+1}:=r_{1}$. Moreover, we also deduce that $r_{i+1}\simeq
  r_{i}^\phi$ and that $r_{1}\simeq r_{1}^{\phi^f}$. Finally, all the
  representations $r_{i}$ are non-isomorphic, since if there was an
  isomorphism between them, we could define a proper
  $W_F$-subrepresentation of $r$ and thus contradict the
  irreducibility of $r$. (More precisely, if we had an isomorphism
  $r_1 \simeq r_{1+s}$ for some $1\leq s<f$, then we could assume that
  $f=sf'$ for some integer $f'$ and get
  $I_F$-isomorphisms \[\alpha_{sk}:r_1\oplus\dots\oplus
  r_s\xrightarrow{\sim} r_{1+sk}\oplus \dots \oplus r_{s(1+k)}\] for
  each $1\leq k<f'$. In that case, we could take the
  $I_F$-subrepresentation of $r$ generated by $v+\alpha_s(v)+\dots +
  \alpha_{s(f'-1)}(v)$ with $v\in r_1\oplus\dots\oplus r_{s}$; it is
  easy to check that this space is also stable under $\phi$ if we
  choose the $\alpha_{sk}$ appropriately.) The fact that $r(\phi)$ induces a cyclic permutation of the $f$ irreducible constituents implies the statement about $r(\phi^t)$.

  (2) Since $r(w)$ is not supported on the diagonal unless $f\mid t$ we get
  the only if part. For the if part, assume that $f\mid t$. By part 1, the matrix $r(\phi^t)$ has the same block decomposition as $r|_{I_F}$. Note that the group algebra of $I_F/H$ surjects onto $\oplus_{i=1}^f \mathrm{End}_{\Qpbar}(r_{i})$, since the $r_{i}$ are non-isomorphic irreducible representations of the finite group $I_F/H$. Therefore, there is some linear combination of matrices $\sum_{h\in I_F}c_h\cdot r(h)$ which has non-zero trace against the non-zero matrix $r(\phi^t)$. This implies that $\mathrm{tr}\bigl(r(h\cdot\phi^t)\bigr)\not=0$ for some $h\in I_F$.

  (3) Observe that $r\otimes\chi\simeq r$ if and only if $\chi(w)\tr r(w)=\tr r(w),~\forall w\in W$.
  The latter condition is equivalent via part 2 to the condition that $\chi(w)=1$ for all $w\in W_F$ such that $f|v(w)$, or equivalently that $\chi^f=1$. Hence part 3 is verified.\end{proof}

\begin{lem}\label{lem: image of Chenevier's map} The image of $T$ generates $\mathfrak{Z}_{\Omega}$ as an
  $E$-algebra.
\end{lem}
\begin{proof}It suffices (by the faithful flatness of the extension $\overline{\mathbb{Q}}_p/E$) to prove the result after
  replacing $E$ with $\overline{\mathbb{Q}}_p$.  Since the inertial type $\tau$ factors through a finite
  quotient $I_F/H$, it decomposes as a direct sum
  $\oplus_{i=1}^r (\tau_i)^{d_i}$, where the $\tau_i$ are non-isomorphic inertial
  types such that $\sigma(\tau_i)$ is cuspidal. As in the proof of
  Proposition~3.11 of~\cite{chenevier}, the Bernstein component $\Omega$
  decomposes as $\Omega_1\times\dots\times\Omega_r$ and $\mathfrak Z_\Omega =
  \otimes_{i=1}^r\mathfrak Z_{\Omega_i}$, where each $\Omega_i$ corresponds to
  the simple type $\sigma((\tau_i)^{d_i})$. If we let $T_i:W_F\to
  \mathfrak{Z}_{\Omega_i}$ be the pseudo-representation associated to $\Omega_i$
  by Proposition~3.11 of~\cite{chenevier}, then by definition
  $T(g):=\sum_{i=1}^rT_i(g)$. It suffices to show that the image of $T$ in $\mathfrak{Z}_{\Omega}$
  generates each $\mathfrak{Z}_{\Omega}$-subalgebra $\mathfrak{Z}_{\Omega_i}$ for $i=1,\dots,r$.

  Let $r_i$ be an irreducible $W_F$-representation such that
  $r_i|_{I_F}\simeq\tau_i$, and let $f_i$ be the integer associated to $r_i$ by
  Lemma~\ref{lem: structure of irred $W_F$-rep}. By choosing
  $\rec_p^{-1}(r_i)^{\otimes d_i}$ as a base point, each closed point of $\Spec
  \mathfrak{Z}_{\Omega_i}$ may be represented by an unramified character
  $\chi_i=(\chi_{i,1},\dots,\chi_{i,d_i})$ (or more precisely by
  $\otimes_{j=1}^{d_i}( \rec_p^{-1}(r_i)\otimes \chi_{i,j})$ up to a permutation
  of factors), where the $\chi_{i,j}$ are unramified characters of $F^\times$. Then
  each $T_i(g)$ is defined
  by \[T_i(g)(\chi_i):=\mathrm{tr}(r_i)(g)\sum_{j=1}^{d_i}\chi_{i,j}\bigl(\mathrm{Art}_F(g)\bigr).\]
  Consider elements $g\in W_F$ of the form $h\cdot\phi^{t_i}$, with $h\in I_F$ and $t_i \in f_i\mathbb{Z}$. By
  Lemma~\ref{lem: structure of irred $W_F$-rep}(1), the matrix $r_i(\phi^{t_i})$ is non-zero and consists of
  $f_i$ blocks which match the block decomposition of $\tau_i$.  Because the constituents
  of $\tau_i$ are non-isomorphic for different $i$'s, we may choose
  the $c_h$ such that \[\sum_{h\in I_F/H}c_h\cdot\mathrm{tr}\bigl(r_i(h)\cdot
  r_i(\phi^{t_i})\bigr)\not=0\] and $\sum_{h\in I_F/H}c_h\cdot r_{i'}(h)=0$ for
  $i'\not=i$. In particular, this means that $\sum_{h\in I_F/H} c_h
  \cdot T(h\phi^{t_i})\in \mathfrak{Z}_{\Omega_i}$.

  We will now compute $\sum_{h\in I_F/H} c_h\cdot T(h\phi^{t_i})$, as an element of $\mathfrak{Z}_{\Omega_i}$. Since the terms for $i'\not=i$ vanish, we can identify this with $\sum_{h\in I_F/H} c_h\cdot T_i(h\phi^{t_i})$ and work inside the Bernstein centre $\mathfrak{Z}_{\Omega_i}$ for the simple type $\sigma((\tau_i)^{d_i})$. If we factor out the non-zero scalar $\sum_{h\in I_F/H}c_h\cdot\mathrm{tr}\bigl(r_i(h)\cdot r_i(\phi^{t_i})\bigr)$, we are left with $\sum_{j=1}^{d_i}\chi_{i,j}\bigl(\mathrm{Art}_F(\phi^{t_i})\bigr)$. We wish to identify this as a regular function on the Bernstein component $\mathfrak{Z}_{\Omega_i}$. Notice that $\mathrm{Art}_F(\phi^{t_i})\in F^\times$ has valuation $t_i$,
which by Lemma~\ref{lem: structure of irred $W_F$-rep}(3) and Remark~\ref{rem:ramification index} coincides with the valuation of $\det(\pi_{E_i})^{t_i/f_i}$, where $E_i/F$ is the extension in Lemma \ref{lem:H(M,lambda_M) has a basis which is finite over Z(M)} for the cuspidal type $\sigma(\tau_i)$.

By the proof of Lemma~\ref{lem:H(M,lambda_M) has a basis which is finite over Z(M)}, the Hecke algebra $\cH(\sigma(\tau_i))$ is generated by the Hecke operator (well-defined up to a non-zero scalar) supported on $\pi_{E_i}$. By the isomorphism between $\cH(\sigma(\tau_i))$ and the Bernstein centre $\mathfrak{Z}_{\Omega(\sigma(\tau_i))}$ for the type $\sigma(\tau_i)$, the latter is generated by the regular function on unramified characters \[\chi\mapsto \chi(\det(\pi_{E_i})).\] Now, the Bernstein centre $\mathfrak{Z}_{\Omega_i}$ can be identified with the elements in the product $\prod_{j=1}^{d_i}\mathfrak{Z}_{\Omega(\sigma(\tau_i))}$ which are invariant under the action of the symmetric group $S_{d_i}$ (see the proof of Proposition 3.11 of~\cite{chenevier} or use the Satake isomorphism on the level of Hecke algebras). For $j=1,\dots,d_i$, let $X_{ij}\in \prod_{j=1}^{d_i}\mathfrak{Z}_{\Omega(\sigma(\tau_i))}$ corespond to the regular function defined above in the $j$th component. From the observation on the valuation of $\mathrm{Art}_F(\phi^{t_i})$, we see that the function \[(\chi_{i,1},\dots,\chi_{i,d_i})\mapsto \sum_{j=1}^{d_i}\chi_{i,j}\bigl(\mathrm{Art}_F(\phi^{t_i})\bigr)\] matches $\sum_{j=1}^{d_i}X^{t_i/f_i}_{ij}\in \mathfrak{Z}_{\Omega_i}$ up to a non-zero scalar.

  Note that we can ensure that $t_i/f_i$ is any integer. Therefore, we can
  generate all elements in $\mathfrak{Z}_{\Omega_i}$ of the form $\sum_{j=1}^{d_i}X_{ij}^k$ for any $k\in \mathbb{Z}$. Since $\mathfrak{Z}_{\Omega_i}$
  is obtained by taking invariants under $S_{d_i}$ in $\mathbb{\bar Q}_p[X^{\pm 1}_{i1},\dots, X^{\pm 1}_{i{d_i}}]$, it is generated as a  $\overline{\mathbb{Q}}_p$-algebra
  by the elementary symmetric polynomials in $X_{ij}$ together with $\prod_{j=1}^{d_i}X^{-1}_{ij}$.
  Over  $\overline{\mathbb{Q}}_p$, which is a field of
  characteristic $0$, we may take the sums of powers of $d_i$ variables as
  generators for the elementary symmetric polynomials in those variables. We may
  also generate the product of the inverses of the variables from sums of powers
  with negative exponents.
\end{proof}

\begin{remark}
While the proof of Lemma~\ref{lem: image of Chenevier's map} is slightly technical,
the lemma itself is rather natural; it expresses the idea that local Langlands
should make sense in families, and hence that the family of $G$-representations
parameterised by $\mathfrak{Z}_{\Omega}$ --- and thus the parameter
ring $\mathfrak{Z}_{\Omega}$ itself --- should be completely determined by
the corresponding family of Weil group representations, which are encoded
by the $\mathfrak{Z}_{\Omega}$-valued pseudo-representation $T$.

If we let $\mathfrak{A}_{\Omega}$ denote the $E$-subalgebra of $\mathfrak{Z}_{\Omega}$
generated by the image of $T$, then this is a finite type $E$-algebra,
and we have a morphism $\Spec \mathfrak{Z}_{\Omega} \to \Spec \mathfrak{A}_{\Omega}$.
It is not hard to see (e.g.\ by applying local Langlands over the fraction field
of $\mathfrak{A}_{\Omega}$) that this is a birational map, which is in
fact a bijection on points
(as one sees by applying local Langlands at the closed points).
Unfortunately, we were unable to find a completely conceptual proof
in general that this morphism is an isomorphism of varieties over $E$.

In the case when $\mathfrak{Z}_{\Omega}$ parameterises supercuspidal
representations, one can see this as follows: it suffices to check
that one obtains an isomorphism after passing to the
formal completion at each closed point $x \in \Spec \mathfrak{Z}_{\Omega}$.
Let $\pi_x$ be the supercuspidal $G$-representation corresponding
to $x$, and let $T_x: W_F \to E_x$  the specialisation of $T$ to the image of $x$
in $\Spec \mathfrak{A}_{\Omega}$. Let $R_x$ be the universal formal
deformation ring of~$T_x$, so that we have morphisms
$$\Spf \widehat{\mathfrak Z_{\Omega}}_x \to \Spf \widehat{\mathfrak A_{\Omega}}
\to \Spf R_x,$$ the second being induced by $T$.  Let $r_x: W_F \to \GL_n(E'_x)$
denote the (absolutely) irreducible representation attached to $\pi_x$ via
local Langlands, where $E'_x/E_x$ is a finite extension, and let
$T'_x$ denote the composite $T_x:W_F\to E_x\to E'_x$. Then $T'_x$ is the pseudo-representation attached
to $r_x$. Since $r_x$ is irreducible, the universal formal deformation
rings of $r_x$ and $T'_x$ coincide (\cite[Th\'eor\`eme 3]{MR1411348},
\cite[Corollaire 6.2]{MR1378546}), and are thus both given by
$R_x\otimes_{E_x}E'_x$.
A direct analysis,
using that the source and target are both obtained simply by
forming unramified twists, and that local Langlands gives a bijection
on isomorphism classes that is compatible with twisting, shows
that the composite of the base change to $E'_x$ of the above morphisms is an isomorphism. Since the first
of them is dominant, it is also an isomorphism.  Thus the morphism
$\Spec \mathfrak Z_{\Omega} \to \Spec \mathfrak A_{\Omega}$ is a
bijection on closed points and induces isomorphisms after completing
at each closed point. From the latter, we see that it is \'etale and
radiciel, hence an open immersion by~\cite[Th\'eor\`eme
17.9.1]{ega-4-iv} and,
since it is also surjective, we see that it is in fact an isomorphism.

One could use a variant
of the argument in first paragraph of
the proof of Lemma~\ref{lem: image of Chenevier's map}
to reduce the general case of the lemma to the case when $\mathfrak Z_{\Omega}$
parameterises a family of supercuspidal representations, where
the preceding argument then applies.   In this way, one could give
a slightly more conceptual proof of the Lemma.
\end{remark}

\begin{proof}
  [Proof of Proposition~\ref{prop: rigid analytic local Langlands}]

  We adopt the notation of \S\ref{god_save_the_queen}. In particular $M$ is the
  Levi subgroup in the supercuspidal support of some (thus any) irreducible
  representation in $\Omega$, and $\mathcal X(M)$ is the group of unramified
  characters of $M(F)$. The group automorphism $\mathcal X(M)\iso \mathcal X(M)$
  given by $\chi_M\mapsto \chi_M |\det|^{\frac{1-n}{2}}$ gives rise to an
  $E$-isomorphism $\Spec\mathfrak{Z}_D\iso \Spec\mathfrak{Z}_D$. The latter map
   is invariant under the $W(D)$-action (the point is that $|\det|$ is invariant
    under $G$-conjugation) so it descends to an $E$-isomorphism
    $\Spec\mathfrak{Z}_\Omega\iso \Spec\mathfrak{Z}_\Omega$ in view of
    Lemma~\ref{invariants_W(D)}. Let $\mathrm{tw}:\mathfrak{Z}_\Omega\ra \mathfrak{Z}_\Omega$
    denote the induced isomorphism.

  We have a natural isomorphism
$\iota_\sigma:\mathcal{H}(\sigma_{\mathrm{sm}})\stackrel{\sim}{\to}\mathcal{H}(\sigma)$; viewing Hecke
algebras as endomorphism-valued functions on $G$, this is
given by $\psi\mapsto \psi\cdot \sigma_\alg$ . (This is \emph{a priori} only an
injection, but in fact is an isomorphism by the proof of Lemma 1.4 of
\cite{MR2290601}.)
  Now we construct $\eta$ as the following composite map
  $$ \cH(\sigma)\stackrel{\iota^{-1}_\sigma}{\simeq} \cH(\sigma_\sm) \simeq \mathfrak{Z}_\Omega\stackrel{\mathrm{tw}}{\ra} \mathfrak{Z}_\Omega \stackrel{I}{\ra}\cO_{\cX^\rig},$$
  where the second map comes from Lemma~\ref{lem: rational Bernstein
    Hecke iso}. Note that $\eta$ is already defined over $E$.  To
  verify the desired interpolation property of $\eta$, we let $x:\cO_{\cX^\rig}\ra E_x$ be an $E$-algebra map. Then $x\circ I\circ \mathrm{tw}:\mathfrak{Z}_\Omega\ra E_x$ gives the supercuspidal support of $\pi_\sm(\rho_x)=r_p^{-1}(\WD(\rho_x)^{F-\semis})$ by Proposition~\ref{prop:Bernstein-to-Xrig}; indeed,
since $I$ interpolates $\rec_p$ by that proposition,
$I\circ \mathrm{tw}$ interpolates
$r_p$. In order to complete the proof, we can and do base change to
$\Qpbar$. Then Proposition~\ref{stronger_dat}
shows us that $\mathfrak{Z}_\Omega$ acts on
 $\Hom_K\bigl(\sigma_\sm,\pi_\sm(\rho_x)\bigr)$ through $x\circ I\circ
 \mathrm{tw}$.

To conclude, it is enough to observe that the action of $\cH(\sigma_\sm)$ on the space
$\Hom_K\bigl(\sigma_\sm,\pi_\sm(\rho_x)\bigr)$ is compatible with the $\mathfrak{Z}_\Omega$-action on the same space via the isomorphism $\cH(\sigma_\sm) \simeq \mathfrak{Z}_\Omega$, and also with the $\cH(\sigma)$-action on $\Hom_K(\sigma,\pi_{\mathrm{l.alg},x})$ via the canonical isomorphisms between the algebras (via $\iota_\sigma$) and the modules. These are readily checked.
\end{proof}

In order to deduce Theorem~\ref{thm: algebraic local
  Langlands} from Proposition~\ref{prop: rigid analytic local Langlands}, we
will now use the results of~\cite{MR2560407} to show that the image of $\eta$ is bounded, in the sense that for any
$h\in\cH(\sigma)$, the valuation of $\eta(h)$ at each point of
$\mathcal{X}^\rig$ is uniformly bounded.

Recall that $\sigma=\sigma_{\mathrm{alg}}\otimes\sigma_{\mathrm{sm}}$,
where $\sigma_{\mathrm{sm}}=\sigma(\tau)$ for some inertial type
$\tau:I_{F}\to \GL_{n}(E)$. Let $x: R_{\tilde{\mathfrak{p}}}^\square(\sigma)[1/p]\to E_x$ be a
closed point, so that $E_x$ is a finite extension field of $E$. Then $x$ defines a local
Galois representation $\rho_{x}:G_{F}\to \GL_{n}(E_x)$ which is potentially crystalline, and has
Hodge--Tate weights determined by the highest weight $\xi$ of
$\sigma_{\mathrm{alg}}$.
Set $\pi_x:=\pi_{\sm}(\rho_x)$. Recall that $\pi_{\mathrm{l. alg},x}$ is the locally algebraic representation defined over $E_x$ corresponding to $\rho_x$ (see \S\ref{subsec:notation}), whose smooth part is $\pi_x$ and which determines the character $\chi_{\pi_x}\circ\iota_\sigma^{-1}:\mathcal{H}(\sigma)\to E_x$ (via the action of $\cH(\sigma)$ on $\Hom_K(\sigma,\pi_{\mathrm{l. alg},x})$.

Let $P=MN$ be a parabolic subgroup of~$G$, with
Levi $M$ and unipotent radical~$N$. Let $Z(M)$ be the centre of $M$, let
$N_{0}\subset N$ be a compact open subgroup and define $Z(M)^{+}:=\{t\in
Z(M)|tN_{0}t^{-1}\subset N_{0}\}$. When $P$ is a standard (upper) parabolic, the
subgroup $Z(M)^{+}$ of $Z(M)$ consists of elements with non-decreasing $p$-adic
valuations on the diagonal. Then~\cite{MR2292633} defines a Jacquet module
functor $J_{P}$ on locally analytic representations of $G$.

We will consider the following condition on a locally analytic representation $V$ of $G$.

\begin{condition}\label{cond:bounds}
  For every parabolic subgroup $P=MN$ as above, with modulus character $\delta_P$,
  every $\chi:Z(M)\to E^{\times}$ such that
  $\mathrm{Hom}_{Z(M)}\bigl(\chi,J_{P}(V)\bigr)\not=0$, and every $t\in
  Z(M)^{+}$, we have
  $|\chi(t)\delta_{P}(t)^{-1}|_{p}\leq1$. \end{condition}

As above, we write $\tau=\oplus_{i=1}^{r}(\tau_{i})^{d_{i}}$, where the $\tau_{i}$
are pairwise non-isomorphic $I_{F}$-representations corresponding via Theorem~\ref{thm: inertial local Langlands, N=0}
(the inertial local Langlands correspondence) to cuspidal types
$\sigma(\tau_{i})$ of $\GL_{e_i}(\cO_F)$. Let
$M=\prod_{i=1}^r\GL_{e_i}(F)^{d_i}$ be a standard Levi of $G$, with
corresponding standard parabolic $P=MN$. 

From now on until the end of this section, we will replace $\pi_x$ (as well as $\chi_{\pi_x}$, $\sigma_\sm$, and so on) by its base extension to $\Qpbar$. Then as recalled in Section~\ref{subsec BZ theory}, $\pi_{x}$ is the unique irreducible quotient of a normalised parabolic
induction
$\tilde{\pi}_{x}:=\nind_{P}^{G} \pi_{x,M}$,
where $\pi_{x,M}=\otimes_{i=1}^{r}(\otimes_{j=1}^{d_{i}}\pi_{x,i,j})$ such that each $\pi_{x,i,j}$ is a supercuspidal representation of
$\GL_{e_{i}}(F)$ containing the type $\sigma(\tau_{i})$ and where for
each $i$, we have
$\pi_{x,i,j}\not\simeq\pi_{x,i,j'}(1)$ for $1\leq j'<j\leq
d_{i}$. 

Proposition~\ref{stronger_dat} gives a $G$-equivariant
homomorphism
$\varphi:\cInd_{K}^{G}\sigma_{\mathrm{sm}}\to\tilde{\pi}_{x}$, which
identifies $\tilde{\pi}_{x}$ with
$\cInd_{K}^{G}\sigma_{\mathrm{sm}}\otimes_{\mathcal{H}(\sigma_{\mathrm{sm}}),\chi_{\pi_{x}}}\Qpbar$. We
identify $W_{[M,\pi_{x,M}]}$ with $\prod_{i=1}^{r}S_{d_{i}}$ in the obvious way,
where $S_{d_{i}}$ is the symmetric group on $\{1,\dots,d_{i}\}$. Note that $W_{[M,\pi_{x,M}]}$ and the identification are independent of $x$.

\begin{lem}\label{lem: inequalities from econ} Let $\chi(\pi_{x,i,j})$
  denote the central character of $\pi_{x,i,j}$. For an element
  $w=\{w_{i}\}_{i=1}^{r}\in W_{[M,\pi_{x,M}]}$, define characters
  $\chi_{x,w}:Z(M)\to \Qpbartimes$ by
  $\chi_{x,w}=\bigotimes_{i=1}^{r}\bigotimes_{j=1}^{d_{i}}\chi(\pi_{x,i,w_{i}(j)})$.

  For every $t\in Z(M)^{+}$, there exists a constant $C_t$ such that $|\chi_{x,w}(t)|_{p}\leq
  C_t$ for all points $x$ of $\cX^\rig$.
\end{lem}

\begin{proof} We know that $\sigma_{\mathrm{alg}}\otimes\tilde{\pi}_{x}$, after
  twisting by a unitary character (this twist is discussed
  at the beginning of \S\ref{sec:breuilschneider} below),
  corresponds to the potentially crystalline Galois representation
  $\rho_{x}$ with Hodge--Tate weights determined by
  $\sigma_{\mathrm{alg}}$, in the sense that
  $\tilde{\pi}_{x}|\det|^{1-n}\leftrightarrow \WD(\rho_{x})^{F-ss}$ via the modified
  local Langlands correspondence as in Section 4 of
  \cite{MR2359853}. Moreover, note that by Lemma 4.2 of Section 4 of \cite{MR2359853},
  $\sigma_\mathrm{alg}\otimes{\tilde{\pi}}_{x}$ actually has a model
  over a sufficiently large finite extension of $\mathbb{Q}_p$, so the
  characters $\chi_{x,w}$ then take values in some sufficiently large
  finite extension $E'_x/\Qp$. 

  The equivalence between parts (ii)
  and (iv) of~\cite[Thm.~1.2]{MR2560407} (where our coefficient field is taken to be $E'_x$) shows that
  $\sigma_{\mathrm{alg}}(\det)^{1-n}\otimes\tilde{\pi}_{x}|\det|^{1-n}$ has a unitary central
  character and satisfies Condition~\ref{cond:bounds}~\cite{MR2292633}. Therefore so does $\sigma_{\mathrm{alg}}\otimes\tilde{\pi}_{x}$.

 Note that by Proposition 4.3.6
  of~\cite{MR2292633} we have
  $J_{P}(\sigma_{\mathrm{alg}}\otimes\tilde{\pi}_{x})\stackrel{\sim}{\to}\sigma_{\mathrm{alg}}^{N}\otimes
  r_{P}^{G}(\tilde{\pi}_{x})\delta_{P}^{1/2}$, where $r_{P}^{G}$ is
  the normalised Jacquet functor for smooth
  representations. Putting this formula together with Condition~\ref{cond:bounds}, we see that
  then we have  $|\chi(t)|_{p}\leq |\sigma_{\mathrm{alg}}^{N}(t)\cdot \delta_P(t)^{-1/2}|^{-1}$
  for every $\chi$ occurring in $r_{P}^{G}(\tilde{\pi}_{x})$.

  Now, Proposition 3.2(2) of~\cite{MR2560407} computes
  $r_{P}^{G}(\tilde{\pi}_{x})$ (observe that in the notation of
  \emph{loc.\ cit.}, all $b_i$ are 1 in our case) and shows that the characters
  $\chi_{x,w}=\bigotimes_{i=1}^{r}\bigotimes_{j=1}^{d_{i}}\chi(\pi_{x,i,w_{i}(j)})$
  of $Z(M)$ for all sets $w$ of permutations $w_{i}$ of $\{1,\dots,d_{i}\}$
  occur in $r_{P}^{G}(\tilde{\pi}_{x})$. The result follows.\end{proof}

Let $Z(M)^{++}\subset Z(M)$ be the subgroup generated by elements with
the property that the $p$-adic valuations are non-decreasing on the
diagonal of each block $\GL_{e_i}^{d_i}$.
Clearly, $Z(M)^+ \subset Z(M)^{++}$.

\begin{cor}\label{cor: inequalities from econ} The conclusion of
Lemma~{\em \ref{lem: inequalities from econ}} holds for all $t\in Z(M)^{++}$.
\end{cor}
\begin{proof} There is a permutation of $\{1,\dots,r\}$ which induces a permutation on the factors $\GL_{e_i}^{d_i}$ of
  $M$ such that the image of $t$ under that permutation has non-decreasing $p$-adic
  valuations. Let $M'$ be the Levi subgroup of $G$ with the permuted
  blocks as factors. Abstractly, $M'\simeq M$ and by Proposition 6.4
  of~\cite{Zel80}, we know that the induction $\nind_{P}^{G} \pi_{M}$
  is independent of the ordering of the $\tau_{i}$. We conclude by
  applying Lemma \ref{lem: inequalities from econ} to the Levi $M'$
  instead of $M$.
\end{proof}

As discussed in Sections~\ref{subsec:BK theory} and~\ref{subsec:SZ}, we have a
semisimple Bushnell--Kutzko type $(J,\lambda)$ such that $\sigma_\mathrm{sm}$ is
a direct summand of $\mathrm{Ind}_J^K(\lambda)$, and the natural map $s_P:\cH(G,\lambda) \to
\cH(\sigma_\mathrm{sm})$ induces an isomorphism
$$Z\bigl(\cH(G,\lambda)\bigr) \iso \cH(\sigma_\mathrm{sm}).$$
In particular, this means that
$\tilde{\pi}_{x}|_J$ contains $\lambda$. Then in the notation of
Section \ref{subsec:BK theory}, $\pi_{x,M}$ contains the type $(J\cap
M,\lambda_{M})$. Let $\chi_{\pi_{x,M}}$ be the character by which
$\mathcal{H}(M,\lambda_{M})$ acts on $\mathrm{Hom}_{J\cap
  M}(\lambda_{M},\pi_{M})$.

\begin{cor}\label{cor: inequalities for centre} Let $t\in Z(M)$ and let
  $\nu_t\in\mathcal{H}(M,\lambda_M)$ be an intertwiner supported on
  $t(J\cap M)$. Then there exists a constant $C_t$ such
  that for all points $x$ of $\mathcal{X}^{\rig}$ we have $| \chi_{\pi_{x,M}}(\nu_t)|_p\leq C_t.$
  \end{cor}

\begin{proof}

  Assume $\nu_t\not=0$. Note that since $t\in Z(M)$, $\nu_t(t)$
  commutes with the action of $J\cap M$ on $\lambda_M$ and since
  $\lambda_M$ is irreducible we deduce that $\nu_t(t)$ is a nonzero scalar. Rescaling, we may assume that
  $\nu_{t}(t):=\mathrm{id}_{\lambda_{M}}$.  Let $s\in Z(M)^{++}$
  be such that $s={}^{w}t$ for some $w\in W_{[M,\pi_{x,M}]}$. It follows from the
  definitions that  $\chi_{\pi_{x,M}}(\nu_{t})=\chi_{x,w}(s)$, and the corollary then follows
  from Corollary \ref{cor: inequalities from econ}.\end{proof}

\begin{cor} \label{lem: inequalities for Levi} Let
  $\nu\in\mathcal{H}(M,\lambda_{M})$. Then there exists a constant
  $C_{\nu}$ such
  that $|\chi_{\pi_{x,M}}(\nu)|_{p}\le C_\nu$ for all points $x$ of $\cX^\rig$.
\end{cor}

\begin{proof} Since $\nu$ has compact support and we only need some
  bounded constant $C_{\nu}$, it suffices to prove the claim in the
  case that
  $\nu$ is an element of the basis of $\cH(M,\lambda_M)$ given by
  Lemma~\ref{lem:H(M,lambda_M) has a basis which is finite over
    Z(M)}. Let $\nu'$ be the $e$-fold convolution of $\nu$ with itself,
  where $e$ is as in the statement of Lemma~\ref{lem:H(M,lambda_M) has
    a basis which is finite over Z(M)}; then from Corollary
  \ref{cor: inequalities for centre} applied to $\nu'$, we see that
  there is some constant $C_\nu$  such
  that $|\chi_{\pi_{x,M}}(\nu')|_p\leq C_\nu^e$ for all $x$. This
  implies that $|\chi_{\pi_{x,M}}(\nu)|_p\leq C_\nu$, as required.
\end{proof}

\begin{prop}\label{prop: weaker statement} For any
  $\nu\in\mathcal{H}(\sigma)$, there is a constant $C_\nu$ such
  that \[|\chi_{\pi_{x}}\bigl(\iota_\sigma^{-1}(\nu)\bigr)|_p\leq C_\nu\] for
  all points $x$ of $\mathcal{X}^{\rig}$.
\end{prop}
\begin{proof}
  Recall from Section~\ref{subsec:SZ} that we have an isomorphism
  \[(s_P\circ{t_P}):\cH(M,\lambda_M)^{W_{[M,\pi_{x,M}]}}\to\cH(\sigma_{\mathrm
    sm}).\] Set $\nu_t:= (s_P\circ
  t_P)^{-1}\bigl(\iota_\sigma^{-1}(\nu)\bigr)$. Corollary~\ref{lem: inequalities for
    Levi} implies that there is a constant $C_\nu$ such
  that $|\chi_{\pi_{x,M}}(\nu_t)|_p \leq C_\nu$ for all
  $x$. 

  Now, in order to conclude, we just need to relate
  $\chi_{\pi_{x,M}}(\nu_M)$ to $\chi_{\pi_x}\bigl((s_P\circ t_P)(\nu_M)\bigr)$ for any
  $\nu_M\in\cH(M,\lambda_M)^{W_{[M,\pi_{x,M}]}}$. As recalled in Section~\ref{subsec:BK theory}, $\nind_P^G$ corresponds on the
  level of Hecke modules to pushforward along the map $t_P$. More
  precisely, if we let $\cM:=\mathrm{Hom}_{J\cap M}(\lambda_M,\pi_{x,M})$
  and $\cN:=\mathrm{Hom}_J(\lambda,\pi_x)$, then $\cN\simeq
  \mathrm{Hom}_{\cH(M,\lambda_M)}(\cH(G,\lambda),\cM)$. Here we view
  $\cH(G,\lambda)$ as a left $\cH(M, \lambda_M)$-module via $t_P$ and
  the action of $\cH(G,\lambda)$ on the space of homomorphisms is via
  right translation. For $z\in Z\bigl(\cH(G,\lambda)\bigr)$, we note that the
  right action is also a left action, so the eigenvalue of $z$ on
  $\cN$ is the same as the eigenvalue of $(t_P)^{-1}(z)$ on $\cM$. On
  the other hand, $\mathrm{Hom}_K(\sigma,\pi_x)\simeq e_K\cN$ for the
  idempotent $e_K$ in $\cH(G,\lambda)$ which defines
  $\sigma$. Therefore, any eigenvalue of $e_K*z$ on $e_K\cN$ is an
  eigenvalue of $z$ on $\cN$.
  We deduce that
  $\chi_{\pi_{x,M}}(\nu_M)$=$\chi_{\pi_x}\bigl((s_P\circ t_P)(\nu_M)\bigr)$.
\end{proof}

\begin{proof}
    [Proof of Theorem~\ref{thm: algebraic local Langlands}]
    By \cite[Thm 3.3.8]{kisindefrings}
we know that
$R_{\tp}^\square(\sigma)[1/p]$ is a regular ring. Proposition~7.3.6 of~\cite{deJ} (which is
applicable because $R_{\tp}^\square(\sigma)[1/p]$ is in particular normal) then implies that
the ring
of rigid analytic functions on $\cX^{\rig}$ whose absolute value is bounded by $1$ coincides precisely with the
normalisation of $R^\square_{\tp}(\sigma)$ in $R^\square_{\tp}(\sigma)[1/p]$, and so
the ring of bounded rigid analytic functions on $\cX^\rig$ is equal to
$R^\square_{\tp}(\sigma)[1/p]$.

The result then follows immediately from Proposition~\ref{prop: weaker
  statement} and the defining property of $\eta$.
 \end{proof}

\begin{remark}\label{inequalities in the crystalline case} In order to
  deduce Theorem~\ref{thm: algebraic local Langlands} from
  Proposition~\ref{prop: rigid analytic local Langlands} in the
  \emph{crystalline} case (that is, the case that $\sigma_{\mathrm{sm}}$ is the trivial representation), one could appeal directly to the inequalities in Proposition 3.2 of~\cite{MR2359853} (see also~\cite{MR2290601}).  In this case, one can obtain a precise bound in terms of the Hodge--Tate weights (so in terms of $\sigma = \sigma_{\mathrm{alg}}$) on the power of $p$ by which we need to scale the usual generators of the spherical Hecke algebra $\cH(\sigma)$. Therefore, in the crystalline case, one can prove that the integral Hecke algebra $\cH(\sigma^\circ)$ (with $\sigma^\circ$ the algebraic representation of $K$ over $\cO$ satisfying $\sigma^\circ\otimes_{\cO}E\simeq\sigma_{\mathrm{alg}}$) maps to the normalisation of the local deformation ring of type $\sigma$. We expect a statement like this should hold true in the general case as well, see Remark~\ref{rem: bounded Hecke} for more details.
\end{remark}

We now return to the global setting. Let the notation be as in Section~\ref{sec:patching}. Fix a $K$-stable $\cO$-lattice $\sigma^\circ$ in $\sigma$.
Set \[M_\infty(\sigma^\circ):=\left(\Hom^{\mathrm{cont}}_{\cO[[K]]}(M_\infty,(\sigma^\circ)^d)\right)^d,\]
where we are considering homomorphisms that are continuous for the profinite topology on $M_\infty$ and the $p$-adic topology on $(\sigma^\circ)^d$,
and where we equip $\Hom^{\mathrm{cont}}_{\cO[[K]]}(M_\infty,(\sigma^\circ)^d)$
with the $p$-adic topology. Note that $M_\infty(\sigma^\circ)$ is an $\cO$-torsion free, profinite, linear-topological $\cO$-module.

\begin{lem}\label{furniture}There is  a natural isomorphism of topological $\cO$-modules
$$M_\infty(\sigma^\circ)\cong \varprojlim_n \left(\Hom^{\mathrm{cont}}_{\cO[[K]]}(M_\infty,(\sigma^\circ/\varpi^n)^\vee) \right)^{\vee}.$$
\end{lem}
\begin{proof} Let $H:=\Hom^{\mathrm{cont}}_{\cO[[K]]}(M_\infty,(\sigma^\circ)^d)$, so that $M_\infty(\sigma^\circ)= H^d$. Then
$H^d\cong \varprojlim_n \Hom_{\cO}( H, \cO/\varpi^n)\cong \varprojlim_n \Hom_{\cO}( H/\varpi^n, \cO/\varpi^n)\cong
\varprojlim_n (H/\varpi^n)^{\vee}.$ Since $M_{\infty}$ is a projective $\cO[[K]]$-module, the short exact sequence
\[0\rightarrow (\sigma^\circ)^d\overset{ \varpi^n \! \cdot}{\rightarrow} (\sigma^\circ)^d \rightarrow (\sigma^\circ)^d/\varpi^n\rightarrow 0\] yields an isomorphism
$H/\varpi^n\cong \Hom^{\mathrm{cont}}_{\cO[[K]]}(M_\infty, (\sigma^\circ)^d/\varpi^n)$. Finally, \[(\sigma^\circ)^d/\varpi^n\cong \Hom_{\cO}(\sigma^{\circ}/\varpi^n, \cO/\varpi^n)\cong (\sigma^\circ/\varpi^n)^{\vee}.\]
\end{proof}

\begin{remark} One may modify the proof of Lemma \ref{furniture} to show that $M_{\infty}(\sigma^\circ)$ is naturally isomorphic to $\left(\Hom^{\mathrm{cont}}_{\cO[[K]]}(M_\infty,(\sigma^\circ)^{\vee})\right)^{\vee}.$
\end{remark}

\begin{rem}
  $M_\infty(\sigma^\circ)$ is essentially the patched module constructed
in Section 5.5 of~\cite{emertongeerefinedBM}, although as the
conventions and constructions of the current paper differ slightly
from those of~\cite{emertongeerefinedBM} (see e.g.\ the difference
in the choices of $v_1$, noted in the discussion of Subsection~\ref{subsec:unitary groups},
as well as Remark~\ref{rem:duality explanation})
we will not make this precise.
\end{rem}

Let
$\cH(\sigma^\circ):=\End_G(\cInd_K^G\sigma^\circ)$; this is an
$\cO_E$-subalgebra of $\cH(\sigma)$. Note that since $\sigma^\circ$ is a free $\cO$-module of finite rank, it follows from the proof of Theorem 1.2 of \cite{MR1900706} that Schikhof duality induces an isomorphism
$$\Hom^{\mathrm{cont}}_{\cO[[K]]}(M_\infty,(\sigma^\circ)^d)\cong \Hom_K(\sigma^\circ, (M_{\infty})^d).$$
Frobenius reciprocity gives  $\Hom_K(\sigma^\circ, (M_{\infty})^d)\cong \Hom_G(\cInd_K^G \sigma^{\circ}, (M_{\infty})^d)$. Thus $M_\infty(\sigma^\circ)$ is equipped
with a tautological Hecke action of $\cH(\sigma^\circ)$, which
commutes with the action of $R_\infty$.

Let $R_\infty(\sigma)$ be the quotient of $R_\infty$ which acts
faithfully on $M_\infty(\sigma^\circ)$. (It follows from Lemma 2.16 of~\cite{paskunasBM} that this is independent
of the choice of lattice $\sigma^\circ\subset \sigma$.)

Set
$R_\infty(\sigma)':=R_\infty\otimes_{R_{\tp}^\square}R_{\tp}^\square(\sigma)$.

\begin{lem}
\label{lem:classical lg compatibility}
\begin{enumerate}
\item $R_\infty(\sigma)$ is a reduced $\cO$-torsion free quotient of $R_\infty(\sigma)'$.
\item  If $h\in\cH(\sigma^\circ)$ is such that $\eta(h)\in
  R_{\tilde{\mathfrak{p}}}^\square(\sigma)$, then the action of $h$ on
  $M_\infty(\sigma^\circ)$ agrees with the action of $\eta(h)$ via the natural
  map $R_{\tilde{\mathfrak{p}}}^\square(\sigma)\to R_\infty(\sigma)'$.
\end{enumerate}
\end{lem}
\begin{proof}
(1) That $R_{\infty}(\sigma)$ is $\cO$-torsion free
follows immediately from the fact that by definition it acts faithfully
on the $\cO$-torsion free module $M_{\infty}(\sigma^{\circ}).$
The fact that it is actually a quotient of $R_{\infty}(\sigma)'$
is then essentially an immediate consequence of classical
local-global compatibility at $\tp$, but to see this will require a
little unraveling of the definitions. Note
that if $N$ is sufficiently large, then $K_N$ acts trivially on
$(\sigma^\circ)^d/\varpi^N$. Recall that $\Gamma_{N}$ is defined to be $\GL_n(\cO_F/\varpi_F^N\cO_F)$. Using Lemma \ref{furniture}, we see that
\begin{align*}
  M_\infty(\sigma^\circ)&=\varprojlim_{N}\Hom_{\cO[\Gamma_{2N}]}\left((M_\infty/\gb_N)_{K_{2N}},(\sigma^\circ)^d/\varpi^N\right)^{\vee}\\
  &=\varprojlim_{N}\Hom_{\cO[\Gamma_{2N}]}\left((M^{\square}_{1,Q_{N'(N)}}/\gb_N)_{K_{2N}},(\sigma^\circ)^d/\varpi^N\right)^{\vee},
\end{align*}
so it suffices to show that if $N\gg 0$ then the action of $R_\infty$ on
\[
\Hom_{\cO[[K]]}^\mathrm{cont}\bigl(M^{\square}_{1,Q_{N'(N)}}, (\sigma^\circ)^d/\varpi^N\bigr)
\]
factors through $R_\infty(\sigma)'$. Now, by definition we have
\begin{multline*}
M^\square_{1,Q_{N'(N)}}=  \\ \pr^\vee\bigl(
S_{\xi,\tau}(U_1(Q_{N'(N)})_{2N'(N)},\cO/\varpi^{N'(N)})^\vee_{\m_{Q_{N'(N)}}}\bigr)\otimes_{R_{\cS_{Q_{N'(N)}}}^{\univ}}R_{\cS_{Q_{N'(N)}}}^{\square_T},
\end{multline*}
so it suffices to prove the same result
for \begin{multline*}
\Hom_{\cO[[K]]}^\mathrm{cont}\Bigl(S_{\xi,\tau}\bigl(U_1(Q_{N'(N)})_{2N'(N)},
\cO/\varpi^{N'(N)}\bigr)^\vee_{\m_{Q_{N'(N)}}}\otimes_{R_{\cS_{Q_{N'(N)}}}^{\univ}}R_{\cS_{Q_{N'(N)}}}^{\square_T},
\\
(\sigma^\circ)^d/\varpi^N\Bigr),
\end{multline*}
which is equal to \begin{multline*}
\Hom_{\cO[[K]]}^{\cont}\Bigl(S_{\xi,\tau}\bigl(U_1(Q_{N'(N)})_{2N'(N)},\cO/\varpi^{N'(N)}\bigr)^\vee_{\m_{Q_{N'(N)}}},
\\ (\sigma^\circ)^d/\varpi^N\Bigr)\otimes_{R_{\cS_{Q_{N'(N)}}}^{\univ}}\bigl(R_{\cS_{Q_{N'(N)}}}^{\square_T}\bigr)^\vee,
\end{multline*}
which in turn equals \[S_{\xi,\tau}\bigl(U_1(Q_{N'(N)})_{0},(\sigma^\circ)^d/\varpi^N\bigr)_{\m_{Q_{N'(N)}}}\otimes_{R_{\cS_{Q_{N'(N)}}}^{\univ}}\bigl(R_{\cS_{Q_{N'(N)}}}^{\square_T}\bigr)^\vee.\]
Therefore it would suffice to prove the same result
for \[S_{\xi,\tau}\bigl(U_1(Q_{N'(N)})_{0},(\sigma^\circ)^d\bigr)_{\m_{Q_{N'(N)}}}\otimes_{R_{\cS_{Q_{N'(N)}}}^{\univ}}\bigl(R_{\cS_{Q_{N'(N)}}}^{\square_T}\bigr)^\vee.\]
If $\T$ denotes the image of $\T^{S_p\cup Q_{N'(N)},\univ}$ in
the endomorphism ring
$$\End_\cO\Bigl(S_{\xi,\tau}\bigl(U_1(Q_{N'(N)})_{0},(\sigma^\circ)^d\bigr)_{\m_{Q_{N'(N)}}}\Bigr),$$
then the action of $R_{\cS_{Q_{N'(N)}}}^{\univ}$ on
$S_{\xi,\tau}\bigl(U_1(Q_{N'(N)})_{0},(\sigma^\circ)^d\bigr)_{\m_{Q_{N'(N)}}}$ is given by an
$\cO$-algebra homomorphism $R_{\cS_{Q_{N'(N)}}}^{\univ}\to\T$.
Since the space of automorphic forms
$S_{\xi,\tau}\bigl(U_1(Q_{N'(N)})_{0},(\sigma^\circ)^d\bigr)_{\m_{Q_{N'(N)}}}$
is $\cO$-torsion free (by the choice of $U_{m,v_1}$ in Section~\ref{subsec:unitary groups})
and the algebra $\T$ is reduced (by the usual comparison between algebraic
modular forms and classical automorphic forms, and the semisimplicity of the
space of cuspidal automorphic forms; \emph{cf.}~\cite[Corollary 3.3.3 and
\S3.4]{cht}),
then by the definition of $R_{\tp}^\square(\sigma)$, we need to show that if
$\T\to\Qpbar$ is a closed point, then the restriction to $G_{F_{\tp}}$
of the corresponding Galois representation $G_{\tF^+,T\cup Q_{N'(N)}}\to\cG_n(\Qpbar)$ is
potentially crystalline of type $\sigma$; but this is immediate
from classical local-global compatibility (see e.g.\ Theorem 1.1
of~\cite{1202.4683}).

Finally, to see that $R_{\infty}(\sigma)$ is reduced, we note that
by part~(2) of Lemma~\ref{lem:basic facts about def rings acting on patched modules} below,
the ring $R_{\infty}(\sigma)[1/p]$
is a direct factor of the regular (by \cite[Thm 3.3.8]{kisindefrings})
ring $R_{\infty}(\sigma)'[1/p].$
Thus $R_{\infty}(\sigma)[1/p]$ is regular, and in particular reduced, and hence
so is its subring $R_{\infty}(\sigma).$ (The reader can easily check that
this reducedness is not used in the proof of Lemma~\ref{lem:basic facts about def rings
acting on patched modules}, and hence no circularity is involved in this argument.)

(2) Again, this is
  essentially an immediate consequence of classical local-global
  compatibility at $\tp$, but a little explanation is needed in order to
  make this plain.

  Note first that the natural action of $\cH(\sigma^\circ)$ on
  $M_\infty(\sigma^\circ)$ is induced via Frobenius reciprocity from
  the $G$-action on $M_\infty$, which is patched from the partial
  $G$-actions defined
  by \[\alphabar_{N'(N)}:(M_\infty/\mathfrak{b}_N)_{K_{2N}}\to\cInd_{KZ}^{G_N}\bigl((M_\infty/\mathfrak{b}_N)_{K_N}\bigr)\]
  and these in turn, after taking homomorphisms into
  $(\sigma^\circ)^d/\varpi^N$, induce
  partial $\cH(\sigma^\circ)$-actions on spaces of algebraic modular
  forms of weight $(\sigma^\circ)^d$. (More precisely, the $G$-action
  on $\varpi$-adically completed
  cohomology \[\tilde S_{\xi,\tau}\bigl(U_1^\p(Q_{N'(N)}),\cO\bigr):=
  \varprojlim_s\Bigl(\varinjlim_mS_{\xi,\tau}\bigl(U_1(Q_{N'(N)})_{m},\cO/\varpi^s\bigr)\Bigr)\]
  gives rise, via Frobenius reciprocity and the identification \[S_{\xi,\tau}\bigl(U_1^\p(Q_{N'(N)})_{0},(\sigma^\circ)^d\bigr)\simeq \Hom_K\Bigl(\sigma^\circ, \tilde S_{\xi,\tau}\bigl(U_1^\p(Q_{N'(N)}),\cO\bigr)\Bigr),\] to a natural action of $\cH(\sigma^\circ)$ on $S_{\xi,\tau}\bigl(U_1(Q_{N'(N)})_{0},(\sigma^\circ)^d\bigr)$.) We see therefore, as in part (1),
  that it is enough to consider the natural action of each
  $h\in\cH(\sigma^\circ)$ on the
  spaces \[S_{\xi,\tau}\bigl(U_1(Q_{N'(N)})_{0},(\sigma^\circ)^d\bigr)_{\m_{Q_{N'(N)}}}\otimes_{R_{\cS_{Q_{N'(N)}}}^{\univ}}\bigl(R_{\cS_{Q_{N'(N)}}}^{\square_T}\bigr)^\vee\]
  for $N\gg 0$. In addition to the natural action of
  $\cH(\sigma^\circ)$, this is equipped with an action of
  $R_{\tp}^\square(\sigma)$ via the composite $R_{\tp}^\square\to
  R^\loc\to R_{\cS_{Q_{N'(N)}}}^{\square_T}$, which factors through
  $R_{\tp}^\square(\sigma)$ by part
  (1). By classical local-global
  compatibility and the defining property of the morphism $\eta$ of Theorem~\ref{thm:
    algebraic local Langlands}, we see that, after inverting $p$, the action of $h$ on this
  space
  agrees with the action of $\eta(h)$. The desired result now follows
  from the fact that
  $S_{\xi,\tau}\bigl(U_1(Q_{N'(N)})_{0},(\sigma^\circ)^d\bigr)_{\m_{Q_{N'(N)}}}$
  is $\cO$-torsion free.
\end{proof}

We now use the usual commutative algebra arguments underlying the
Taylor--Wiles--Kisin method to study the support of
$M_\infty(\sigma^\circ)$.\begin{lem}
\label{lem:basic facts about def rings acting on patched modules}
\begin{enumerate}
\item The module $M_{\infty}(\sigma^{\circ})$ is finitely generated
over $R_{\infty}(\sigma)$ and Cohen--Macaulay,
and moreover $M_\infty(\sigma^\circ)[1/p]$ is
locally free of rank one over $R_\infty(\sigma)[1/p]$.
The topology on $M_{\infty}(\sigma^{\circ})$ coincides with
its $\mathfrak m$-adic topology, where $\mathfrak m$ denotes
the maximal ideal of $R_{\infty}(\sigma)$.
\item The support of $M_\infty(\sigma^\circ)$ in $\Spec R_{\infty}(\sigma)'$
is a union of irreducible components of $R_{\infty}(\sigma)'$. 

\item Let $\barR_\infty(\sigma)$ be the normalisation of
  $R_\infty(\sigma)$ inside $R_\infty(\sigma)[1/p]$. Then the action
  of $\cH(\sigma^\circ)$ on $M_\infty(\sigma^\circ)$ induces an $\cO$-algebra map
  $\alpha:\cH(\sigma^\circ) \to \barR_\infty(\sigma)$.
\end{enumerate}
\end{lem}
\begin{proof}
Since $M_\infty$ is a finite projective
$S_\infty[[K]]$-module, the module $M_\infty(\sigma^\circ)$ is finite and projective
(equivalently, free) over $S_\infty$.  Indeed,
we may write $M_{\infty}$ as a direct summand of $S_{\infty}[[K]]^r$ for some $r \geq 0,$
and so the space
$\Hom_{\cO[[K]]}^{\cont}\bigl(M_{\infty},(\sigma^{\circ})^{d}\bigr)^{d}$
is a direct summand of
\[\Hom_{\cO[[K]]}^{\cont}\bigl(S_{\infty}[[K]]^r,(\sigma^{\circ})^{d}\bigr)^{d}
\cong
\Bigl(\Hom_{\cO[[K]]}^{\cont}\bigl(S_{\infty}[[K]],(\sigma^{\circ})^{d}\bigr)^{d}\Bigr)^r.\]
Thus it suffices to note that since  $S_\infty[[K]]\cong S_\infty\widehat{\otimes}_{\cO}\cO[[K]]$ as an $\cO[[K]]$-module,  there is a natural isomorphism
\begin{displaymath}
\begin{split}
\Hom_{\cO[[K]]}^{\cont}\bigl(S_{\infty}[[K]],(\sigma^{\circ})^{d}\bigr)^{d}&\cong
\Hom_{\cO[[K]]}^{\cont}\bigl(\cO[[K]],\Hom_{\cO}^{\mathrm{cont}}(S_{\infty}, (\sigma^{\circ})^{d})\bigr)^d\\ &\cong
\Hom_{\cO}^{\cont}(S_{\infty}\otimes_{\cO} \sigma^{\circ}, \cO)^d
\cong S_{\infty}\otimes_{\cO}\sigma^{\circ}.
\end{split}
\end{displaymath}
Since $M_{\infty}(\sigma^{\circ})$ is free of finite rank over the formal
power series ring $S_{\infty}$, it is Cohen--Macaulay.
Since the $S_{\infty}$-action on $M_{\infty}(\sigma^{\circ})$ factors
through the action of $R_{\infty}$, which in turn factors through
$R_{\infty}(\sigma)$ by definition, we also conclude that $M_{\infty}(\sigma^{\circ})$
is finitely generated over $R_{\infty}(\sigma)$.

Since the identification of $M_{\infty}$ as a direct summand of $S_{\infty}[[K]]$
is compatible with the natural topologies on each of $M_{\infty}$ and $S_{\infty}[[K]],$
one easily verifies that the topology on $M_{\infty}(\sigma^{\circ})$
coincides with its $\mathfrak n$-adic topology, where $\mathfrak n$ denotes
the maximal ideal of $S_{\infty}$.
Furthermore,
since by definition $R_{\infty}(\sigma)$ embeds into
$\End_{S_{\infty}}\bigl(M_{\infty}(\sigma^{\circ})\bigr),$
we find that $R_{\infty}(\sigma)$ is finite as an $S_{\infty}$-algebra,
and so in particular the $\mathfrak n$-adic topology and $\mathfrak m$-adic topology
on $M_{\infty}(\sigma^{\circ})$ coincide (where, as in the statement
of the lemma,  $\mathfrak m$ denotes the maximal ideal of $R_{\infty}(\sigma)$).
Thus the topology on $M_{\infty}(\sigma^{\circ})$ coincides with
its $\mathfrak m$-adic topology.

By Lemma 3.3 of~\cite{blght}, Lemma 2.4.19 of~\cite{cht}, and Theorem
3.3.8 of~\cite{kisindefrings}, we see that
the ring $R_\infty(\sigma)'$ is equidimensional of the same Krull dimension as~$S_\infty$. Since $M_\infty(\sigma^\circ)$ is free of finite rank over
$S_\infty$, and the image of $R_\infty(\sigma)'$ in
$\End(M_\infty(\sigma^\circ))$ is an $S_\infty$-algebra, we see that
the depth of $M_\infty(\sigma^\circ)$ as an $R_\infty(\sigma)'$-module
is at least the Krull dimension of $S_\infty$. Since this is equal to
the Krull dimension of $R_\infty(\sigma)'$, it follows immediately
from Lemma~2.3 of~\cite{tay} that the support of
$M_\infty(\sigma^\circ)$ is a union of irreducible components of
$R_\infty(\sigma)'$. (Of course, conjecturally
  $R_\infty(\sigma)$ is actually equal to
  $R_\infty(\sigma)'$.)

  That $M_\infty(\sigma^\circ)[1/p]$ is locally free over
  $R_\infty(\sigma)[1/p]$ follows by an argument of Diamond
  (\emph{cf.}\ \cite{MR1440309}). More precisely,
  $R_{\tp}^\square(\sigma)[1/p]$ and each of the rings
  $R_{\tv}^{\sigma,\xi,\tau}[1/p]$ for places $v\mid p$, $v\ne\p$ are
  regular by Theorem~3.3.8 of~\cite{kisindefrings}, and
  $R_{\tv_1}^\square$ is formally smooth by Lemma~\ref{lem: at v1, we have smooth unramified
    lifts}, so
  $R_\infty(\sigma)'[1/p]$ is regular by Corollary~\ref{cor_risi}. Therefore
  $R_\infty(\sigma)[1/p]$ is also regular, so $M_\infty(\sigma^\circ)[1/p]$ is   locally free over $R_\infty(\sigma)[1/p]$ by Lemma~3.3.4
  of~\cite{kis04} (or rather by its proof, which goes over unchanged
  to our setting, where we do not assume that $R_\infty(\sigma)[1/p]$
  is a domain).

  That it is actually locally free of rank one can be checked at
  finite level, where it follows from the multiplicity one assertion
  in Theorem~\ref{thm: inertial local Langlands, N=0}, the choice
  of $v_1$ (and the fact that we have fixed the action mod $p$ of the
  Hecke operators at $\tv_1$), and the irreducibility of $\rhobar$, together
  with~\cite[Thms.\ 5.4 and 5.9]{labesse}.

This completes the proof of parts~(1) and (2), and so we turn to proving~(3).
To this end,
  let $\cA$ be the $R_\infty$-subalgebra of the endomorphism
  algebra of $M_\infty(\sigma^\circ)$ generated by
  $\cH(\sigma^\circ)$. Since $M_\infty(\sigma^\circ)$ is a finite type
  $R_\infty(\sigma)$-module, we see that $\cA$ is a finite
  $R_\infty(\sigma)$-algebra. Since $M_\infty(\sigma^\circ)[1/p]$ is
  in fact locally free of rank one over $R_\infty(\sigma)[1/p]$ we
  have the equality
  $\End_{R_\infty(\sigma)[1/p]}\bigl(M_\infty(\sigma^\circ)[1/p]\bigr)=R_\infty(\sigma)[1/p]$,
  so that $\cA[1/p]=R_\infty(\sigma)[1/p]$. So the natural map
  $\cH(\sigma^\circ) \to \cA$ lands inside $\barR_\infty(\sigma)$.
\end{proof}

The morphism $\alpha$ of Lemma~\ref{lem:basic facts about def rings acting on
  patched modules} induces an $E$-algebra morphism $\alpha: \cH(\sigma) \to R_\infty(\sigma)[1/p]$.

\begin{thm}
\label{thm:alpha equals beta}
 The map $\alpha$ coincides with the
 composition \[\cH(\sigma)\stackrel{\eta}{\to}R_{\tp}^\square(\sigma)[1/p]\to R_\infty(\sigma)[1/p].\]
\end{thm}
\begin{proof}Note firstly that if $h\in\cH(\sigma^\circ)$ is such that
  $\eta(h)\in R_{\tilde{\mathfrak{p}}}^\square(\sigma)$, then the two maps agree on $h$ by
  Lemma~\ref{lem:classical lg compatibility}~(2). Since $R_\infty(\sigma)$ is
  $p$-torsion free by Lemma~\ref{lem:classical lg compatibility}~(1),
it is therefore enough to show that $\cH(\sigma)$ is
  spanned over $E$ by such elements.

Now, $\cH(\sigma^\circ)$ certainly spans $\cH(\sigma)$ over $E$, so it is enough
to show that for any element $h'\in\cH(\sigma^\circ)$, we have $\eta(p^Ch')\in
R_{\tilde{\mathfrak{p}}}^\square(\sigma)$ for some $C\ge 0$; but this is obvious.
\end{proof}

  \begin{rem}\label{rem: definition of an automorphic component} It follows from
  Lemma~\ref{lem:basic facts about def rings acting on patched
    modules}~(2) that the locus of closed points of $\Spec
  R_{\tp}^\square(\sigma)[1/p]$ which come from closed points of $\Spec
  R_{\infty}(\sigma)[1/p]$ is a union of irreducible components, which
  we call the set of \emph{automorphic components} of $\Spec
  R_{\tp}^\square(\sigma)[1/p]$. (Note that we do
  not know \emph{a priori} that this notion is independent of the
  choice of global setting, although of course we expect that in fact
  every component of $\Spec R_{\tp}^\square(\sigma)[1/p]$ is an
  automorphic component.)
\end{rem}
\begin{rem}\label{rem: bounded Hecke}
  We expect that $\eta(\cH(\sigma^\circ))$ is contained in the
normalisation of $R_{\tp}^\square(\sigma)$ in
$R_{\tp}^\square(\sigma)[1/p]$; it may well be possible to prove this
via our methods, but as we do not need this result, we have not
pursued it. It is easy to see that the analogous result holds for
the quotient of $R_{\tp}^\square(\sigma)$ corresponding to the
automorphic components in the sense of Remark~\ref{rem: definition of an automorphic component}.
\end{rem}

\subsection{The action of the centre of $G$} We next prove a structural
result (Proposition~\ref{prop:freeness over the centre} below) which describes
the action of the centre~$Z$
of $G$ on $M_{\infty}$. 

As usual, we identify $Z$ with $F^{\times}$,
by associating to each element of $F^{\times}$ the corresponding
scalar matrix.   Local class field theory then gives an embedding
$Z \cong F^{\times} \buildrel \Art_F \over \longrightarrow G_F^{ab},$
which we again denote by $\Art_F$.

If $r^{\univ}: G_F \to \GL_n(R_{\tp}^{\square})$ denotes the universal
lift of $\rbar$, then its determinant is
a character $\det r^{\univ}: G_F^{\ab}\to (R_{\tp}^{\square})^{\times},$
which, when composed with $\Art_F$, induces a character
$\det r^{\univ}\circ \Art_F: Z \to (R_{\tp}^{\square})^{\times}.$
If we let $\Lambda_Z$ denote the completion
of the group algebra $\cO[Z]$ at the maximal ideal generated by~$\varpi$ and the elements
$z-(\epsilon^{n(n-1)/2}\det r^{\univ})\circ \Art_F (z)$, then this character induces a homomorphism
$\Lambda_Z \to R_{\tp}^{\square};$ the corresponding
morphism of schemes $\Spec R_{\tp}^{\square} \to \Spec \Lambda_Z$
simply associates to each deformation $r$ of $\rbar$ the
character $ (\epsilon^{n(n-1)/2}\det r^{\univ})\circ \Art_F$ of $Z$; in this optic, the
complete local ring $\Lambda_Z$
is identified with the universal deformation ring of the character
$\epsilonbar^{n(n-1)/2}\det \rbar$.

 By local-global compatibility, $Z$ acts on $S_{\xi,\tau}(U_i(Q_N)_{2N},\F)[\m_{Q_N}]$ via the
character $(\epsilonbar^{n(n-1)/2}\det\rbar)\circ\Art_F$.
This implies the elements  of the form $z-(\epsilon^{n(n-1)/2}\det r^{\univ})\circ \Art_F (z)$ act nilpotently 
on $M_{\infty}^{\vee}[\mm_{\infty}^n]$, for all $n\ge 1$ and for all $z\in Z$. Hence the action of
of $\cO[Z]$ on $M_{\infty}^{\vee}$ extends to a continuous action of $\Lambda_Z$. 
Pontryagin duality induces an isomorphism $ \End_{\cO}^{\cont}(M_{\infty}^{\vee})\cong (\End_{\cO}^{\cont}(M_{\infty}))^{\op}$,
which makes $M_{\infty}$ into a continuous $\Lambda_Z^{\op}$-module. Since $\Lambda_Z$ is commutative, 
$\Lambda_Z^{\op}=\Lambda_Z$.

Let $\varpi_F$ be a choice of uniformiser of~$F$ and let $z=\diag(\varpi_F,\dots,\varpi_F)\in Z$, so that $Z=(Z\cap K)z^\Z$. Write
$S=z-[\mu]\in \cO[Z]$, where we set $\mu=(\epsilonbar^{n(n-1)/2}\det\rbar)\circ\Art_F(\varpi_F)$. Let $\Lambda$ be the closure in $\Lambda_Z$
of the subring $\cO[S]$. Since $S$ lies in the maximal ideal of
$\Lambda_Z$, we see that $\Lambda$ is isomorphic to $\cO[[S]]$.
We now make use of
the category of pseudo-compact $\Lambda[[K]]$-modules;
see~\cite[\S IV.3]{MR0232821}, \cite{MR0202790} for the definition and properties of this category.

\begin{prop}
  \label{prop:freeness over the centre}$M_\infty$ is  projective in the category of pseudo-compact $\Lambda[[K]]$-modules. \end{prop}
\begin{proof} Let $P$ be a pro-$p$ Sylow subgroup of $K$. Since the index $(K:P)$ is finite and is not divisible by $p$,
it is enough to show that $M_{\infty}$ is projective in the category of pseudo-compact $\Lambda[[P]]$-modules.
We will in fact show that $M_{\infty}$ is a pro-free $\Lambda[[P]]$-module, i.e.\ it is isomorphic to a product of copies of
$\Lambda[[P]]$.

Since $P$ is a pro-$p$ group, $\Lambda[[P]]$ is a local ring with
residue field $\F$. It follows from the topological Nakayama's lemma
for pseudo-compact $\Lambda[[P]]$-modules, that  it is enough to show
that the first right derived functor of $-\wtimes_{\Lambda[[P]]}\F$
vanishes, see \cite[Proposition 3.1]{MR0202790}. We denote this derived functor by $\widehat{\Tor}^{\Lambda[[P]]}_1(\F, M_{\infty})$.

Note that the functor $-\wtimes_{\Lambda[[P]]} \F$ is the composite of the
two functors $-\wtimes_{\cO[[P]} \F$ and $-\wtimes_{\Lambda/\varpi}
\F$. Considering the corresponding spectral sequence, we see that it
is enough to show that both $\widehat{\Tor}^{\cO[[P]]}_1(\F, M_{\infty})$ and
  $\widehat{\Tor}^{\Lambda/\varpi}_1(\F, M_{\infty}\wtimes_{\cO[[P]]}
  \F)$ vanish.

 We know by Proposition~\ref{prop:Minfty
  is projective and has a G-action} that $M_{\infty}$ is a finitely
generated projective $S_{\infty}[[K]]$-module, thus a free, finitely
generated $S_{\infty}[[P]]$-module, and thus  projective in the category of
pseudo-compact $\cO[[P]]$-modules.
This implies that $\widehat{\Tor}^{\cO[[P]]}_1(\F,
M_{\infty})$ vanishes.

  Since $\Lambda/\varpi\cong \F[[S]]$ is a DVR, to show that $\widehat{\Tor}^{\Lambda/\varpi}_1(\F, M_{\infty}\wtimes_{\cO[[P]]} \F)$ vanishes it suffices to
show that $M_\infty \wtimes_{\OO[[P]]}\F$ is $S$-torsion free. Since $M_{\infty}$ is a free finitely generated $S_{\infty}[[P]]$-module, $M_{\infty}\wtimes_{\cO[[P]]} \F$ is a free finitely generated
  $S_{\infty}/\varpi$-module. Since $S_{\infty}/\varpi$ is a domain, it is enough to show that there is a polynomial $q\in \F[X]$ with zero constant term,
  such that the action of $q(S)$ on $M_{\infty}\wtimes_{\cO[[P]]} \F$ is given by a non-zero element of $S_{\infty}/\varpi$.

  We will now construct such a polynomial $q$, thus finishing the proof. Let~$\tZ$ denote
the centre of~$\tG$, so that in particular ~$\tZ(\A_{\tF^+}^{\infty})$
contains~$Z$; since the group
$\tZ(\A_{\tF^+}^\infty)/\bigl(\tZ(\A_{\tF^+}^{\infty})\cap U_0\bigr)
\tZ(\tilde F^+)$ is finite,
we can choose some $a>0$ for which we can write $z^a=u\gamma$, with $u\in
\tZ(\A_{\tF^+}^{\infty})\cap U_0$ and $\gamma\in  \tZ(\tilde F^+)$. We may write $u=u_p u^p$ with  $u_p\in Z\cap K$, $u^p\in
U_{0}^{\p}\cap \tZ(\A_{\tF^+}^{\infty})$. By replacing $a$ by a multiple we may assume that $u_p\in Z\cap P$.

We claim that $q(X)= (X+\mu)^a- \mu^a$ does the job, and note that $q(S)= z^a -\mu^a= \gamma u^p u_p-\mu^a$. Since $\gamma$ acts trivially on each $M_{1,Q_N}$, it acts trivially on $M_{\infty}$, and hence the action of $z$ on $M_{\infty}$
coincides with the action of $u_p u^p$. Since $u_p\in Z\cap P$, the action of $q(S)$ on $M_{\infty}\wtimes_{\cO[[P]]} \F$
coincides with the action of $u^p-\mu^a$. The action of $u^p$ on each $M_{1,Q_N}$ factors through
that of $\Delta_{Q_{N}}$. (Indeed, for each $v\in Q_N$, the component of $u^p$ at $v$
lies in $U_0(Q_N)_v$; since $U_1(Q_n)_v$ acts trivially on
$M_{1,Q_N}$, the claim is immediate from the definition
of~$\Delta_{Q_N}$.) Thus the action of
$q(S)$ on $M_{\infty}\wtimes_{\cO[[P]]} \F$ coincides with the action of an element of $S_{\infty}/\varpi$, and this
action is zero if and only if $u^p$ acts trivially on $M_{\infty}$ and $\mu^a=1$. In this case,
$z^a$ would act trivially on  $M_{\infty}/ \mf{I} M_{\infty}$, where $\mf{I}$ is the maximal ideal of $\cO[[Z\cap P]]$.

Suppose for
the sake of contradiction that this happens. We choose a locally algebraic type~$\sigma$ such
that~$M_\infty(\sigma^\circ)\ne 0$. It follows from \cite[Lemma 2.14]{paskunasBM} and the fact that $Z\cap P$ acts trivially on $\sigma^\circ/\varpi$ that
\begin{displaymath}
\begin{split}
M_\infty(\sigma^\circ)/\varpi M_\infty(\sigma^\circ)&\cong \Hom^{\cont}_{\cO[[K]]}(M_{\infty}/\varpi M_{\infty}, (\sigma^{\circ}/\varpi)^{\vee})^{\vee}\\ &\cong
\Hom^{\cont}_{\cO[[K]]}(M_{\infty}/\mf{I}M_{\infty}, (\sigma^{\circ}/\varpi)^{\vee})^{\vee}.
\end{split}
\end{displaymath}
Thus to see that $(z^a-\mu^a)$ does not act by zero on $M_\infty/\mf{I}
M_\infty$, it is enough to show that it does not act by zero on $M_\infty(\sigma^\circ)/\varpi M_\infty(\sigma^\circ)$.

Theorem~\ref{thm:alpha equals beta} implies that the action of $z$ on $M_\infty(\sigma^\circ)$ is the same as the action of $z$ under the map $\cH(\sigma^\circ)\to R^{\square}_{\tp}(\sigma)\to R_\infty(\sigma)$, which can be checked to be compatible with the map $\Lambda_Z\to R^{\square}_{\tp}\to R_\infty$. Explicitly, if $\Frob_p$ is the element of $G_F^{\ab}$ corresponding to $\varpi_F$ by local class field theory, then the action of $z$ on $M_\infty(\sigma^\circ)$ matches the determinant of $\Frob_p$. Then by
Theorem~\ref{thm:alpha equals beta},
$M_\infty(\sigma^\circ)/\varpi_E M_\infty(\sigma^\circ)$ would be
supported on a quotient of $R_{\tp}^{\square}(\sigma)$ corresponding
to representations where the determinant of~$\Frob^a_p$ is
fixed; that is to say (again by local-global compatibility),
for any representation
$r: G_F \to \GL_n(\Qbar_p)$ arising from a $\Qbar_p$-valued
point of $\Spec R(\sigma)[1/p],$
the value of $\det r$ on $\Frob^a_p$ would be determined modulo $\varpi$.

However, we know that $\Spec R(\sigma)$ is a union of irreducible components
of $\Spec R(\sigma)' := \Spec R_{\infty} \otimes_{R_{\tp}^{\square}}
R_{\tp}^{\square}(\sigma),$
and hence any twist of $r$ by an unramified character which is trivial
modulo $\mathfrak m_{\Zbar_p}$
also arises from a $\Qbar_p$-valued point of $\Spec R(\sigma)$.
Making an unramified twist by an appropriate character (e.g.\ by a character which takes $\Frob_p$ to
$1 +\unif_{E'}$,
for  some sufficiently ramified extension~$E'$ of~$E$),
shows that $\det r(\Frob^a_p)$ is not constant mod $\unif_E$, as required.
\end{proof}

\begin{rem}\label{rem: pseudo compact duality} Since Pontrjagin duality induces an anti-equivalence
  between pseudo-compact and discrete $\Lambda[[K]]$-modules
  \cite[Prop.\ 2.3]{MR0202790},
Proposition \ref{prop:freeness over the centre} may be reformulated as
the statement that~$M_{\infty}^{\vee}$ is injective in the category of discrete $\Lambda[[K]]$-modules.
\end{rem}

\begin{rem}  Following \cite[Definition 2.3.1]{emertonord1}, we say that a smooth representation $V$ of $ZK$ is locally $Z$-finite, if
for every $v\in V$ the $\cO[Z]$-submodule generated by $v$ is finitely generated as an $\cO$-module. Such representations
form a full subcategory $\Mod^{\Zfin}_{ZK}(\cO)$ of the
category $\Mod^{\sm}_{ZK}(\OO)$ of smooth representations of $ZK$ on $\OO$-torsion modules.
The category of discrete $\Lambda[[K]]$-modules coincides with the full subcategory of $\Mod^{\Zfin}_{ZK}(\cO)$ consisting of representations $V$,
such that every $v\in V$ is annihilated by a power of the maximal ideal
of $\Lambda$.

It follows from the Chinese remainder theorem that the category of discrete $\Lambda[[K]]$-modules is a direct summand of $\Mod^{\Zfin}_{ZK}(\cO)$,
so by Remark~\ref{rem: pseudo compact duality}, $M_{\infty}^{\vee}$ is injective in $\Mod^{\Zfin}_{ZK}(\cO)$. Since $K$ is compact, every smooth representation of $K$ is locally admissible. Combining this
observation with \cite[Lemma 2.3.4]{emertonord1} we deduce that $\Mod^{\Zfin}_{ZK}(\cO)$ coincides with the category of locally admissible representation
of $ZK$ on $\OO$-torsion modules, $\Mod^{\ladm}_{ZK}(\OO)$. Thus $M_{\infty}^{\vee}$ is injective in $\Mod^{\ladm}_{ZK}(\OO)$.
\end{rem}

In some situations it can be useful to consider quotients of
$M_{\infty}$ having a fixed central character.  (This corresponds, on the
Galois side, to considering deformations of $\rbar$ having a fixed
determinant.)  To this end, we state and prove Corollary~\ref{cor:fixed
central char} below.

Let $x: \Lambda_Z\rightarrow \OO$ be an $\OO$-algebra homomorphism. Let $\xi$ be the composition $Z\rightarrow \Lambda_Z^{\times}\overset{x}{\rightarrow} \OO^{\times}$.
Let $\Mod^{\sm, \xi}_{ZK}(\OO)$ be the full subcategory of $\Mod^{\sm}_{ZK}(\OO)$ consisting of those representations on which $Z$ acts by the character $\xi^{-1}$.
Let $\Mod^{\proaug}_{ZK}(\OO)$ be the category of profinite augmented
representations of $ZK$ over $\OO$, as defined in~ \cite[Definition 2.1.6]{emertonord1}. Pontrjagin duality induces an anti-equivalence of categories between $\Mod^{\sm}_{ZK}(\OO)$ and $\Mod^{\proaug}_{ZK}(\OO)$, \cite[(2.2.8)]{emertonord1}.
Let $\mathfrak C(\OO)$ be the full subcategory of $\Mod^{\proaug}_{ZK}(\OO)$ consisting of those representations on which $Z$ acts by the character $\xi$, so that
$\mathfrak C(\OO)$ is anti-equivalent to $\Mod^{\sm, \xi}_{ZK}(\OO)$ via Pontrjagin duality.

\begin{cor}
\label{cor:fixed central char}
$M_{\infty}\wtimes_{\Lambda_Z, x} \OO$ is a non-zero, projective object in $\mathfrak C(\OO)$.
\end{cor}
\begin{proof} Note that $\mathfrak C(\OO)$ is naturally a full
  subcategory of the category of pseudo-compact $\Lambda[[K]]$-modules.
The projectivity of $M_{\infty}\wtimes_{\Lambda_Z, x} \OO$ follows from the fact that the functor
$\Hom_{\mathfrak C(\OO)}(M_{\infty}\wtimes_{\Lambda_Z, x} \OO,  -)$ coincides with the restriction of the functor
$\Hom^{\cont}_{\Lambda[[K]]}(M_{\infty}, -)$ to $\mathfrak C(\OO)$; the exactness of the latter follows from Proposition \ref{prop:freeness over the centre}.
The reduction of $M_{\infty}\wtimes_{\Lambda_Z, x} \OO$ modulo $\varpi$ is isomorphic to $M_{\infty}\wtimes_{\Lambda_Z} \F$, which is non-zero, as
otherwise the topological Nakayama's lemma for pseudo-compact $\Lambda_Z$-modules would imply that $M_{\infty}$ is zero.
\end{proof}

\begin{rem} The same argument as in Remark \ref{rem: pseudo compact duality} shows that the dual $(M_{\infty}\wtimes_{\Lambda_Z, x} \OO)^{\vee}$ is injective in $\Mod^{\sm, \xi}_{ZK}(\OO)$.
\end{rem}

\subsection{Locally algebraic vectors}
We conclude this section by computing
the locally algebraic vectors in $V(r)$, in the case when $r$ is a generic,
potentially crystalline point of some $R_\infty(\sigma)$. (In other words, we identify the locally algebraic vectors at such points on automorphic
components of type $\sigma$.) 

Let $\Omega$ be  a Bernstein component corresponding to the inertial
type $\tau$ and let $(J, \lambda(\tau))$ be a type for this component,
as in \S \ref{subsec:BK theory}. The representation
$\sigma(\tau)$ in Theorem \ref{thm: inertial local Langlands, N=0} is a quotient of $\Ind_J^K \lambda(\tau)$.

\begin{rem}To orientate the reader not familiar with types, the example to keep
in mind is the following: if $\tau$ is a direct sum of copies of the
trivial representation, then $J$ is the Iwahori subgroup,
$\lambda(\tau)$ is the trivial representation of~$J$, and $\sigma(\tau)$ is the trivial
representation of $K$. The Steinberg representation $\St$ lies in $\Omega$, but $\Hom_K(\sigma(\tau), \St)=0$. If we worked only with $\sigma(\tau)$, we would not be able to control
copies of $\St$ tensored with an algebraic representation inside the
locally algebraic vectors of our patched modules.
This explains the need to work with $\lambda(\tau)$ instead of $\sigma(\tau)$.
\end{rem}

We will redo some of the lemmas in \S \ref{sec:local-global compatibility}
with $\lambda$ instead of $\sigma$. We denote by $\lambda_{\alg}$ the
representation denoted by $\sigma_{\alg}$ in \S \ref{sec:local-global
  compatibility}. We let  $\lambda:=\lambda(\tau)\otimes
\lambda_{\alg}$ and fix a $J$-stable $\cO$-lattice $\lambda^\circ$ in $\lambda$.
Set \[M_\infty(\lambda^\circ):=\left(\Hom^{\mathrm{cont}}_{\cO[[J]]}(M_\infty,(\lambda^\circ)^d)\right)^d.\]

\begin{lem}\label{free_over_S_infty} $M_{\infty}(\lambda^{\circ})$ is a free $S_{\infty}$-module of finite rank.
\end{lem}
\begin{proof} This follows from the fact, proved in Proposition \ref{prop:Minfty is projective and has a G-action}, that $M_{\infty}$ is projective as an
$S_{\infty}[[K]]$-module; see the proof of Lemma \ref{lem:basic facts about def rings acting on patched modules}.
\end{proof}

Let $R_\infty(\lambda)$ be the quotient of $R_\infty$ which acts
faithfully on $M_\infty(\lambda^\circ)$,
and set $R_\infty(\lambda)':=R_\infty\otimes_{R_{\tp}^\square}R_{\tp}^\square(\lambda)$, where $R_{\tp}^\square(\lambda)$ is the unique reduced and
$p$-torsion free quotient of $R_{\tp}^\square$ corresponding to
potentially  semi-stable  lifts of weight $\lambda_{\alg}$ and inertial type $\tau$.

\begin{lem}\label{quotient_control} $R_{\infty}(\lambda)$  is a $p$-torsion free quotient of $R_{\infty}(\lambda)'$.
\end{lem}
\begin{proof} The proof is the same as the proof of part (1) of  Lemma \ref{lem:classical lg compatibility}.
\end{proof}

\begin{lem}\label{union_irred} The support of $M_\infty(\lambda^\circ)$ in $\Spec R_{\infty}(\lambda)'$
is a union of irreducible components of $R_{\infty}(\lambda)'$. In particular, $R_{\infty}(\lambda)$ is reduced.
\end{lem}
\begin{proof} The first assertion follows from Lemma \ref{free_over_S_infty} and the fact that $S_{\infty}$ and $R_{\infty}(\lambda)'$
have the same Krull dimension. Since $R_{\infty}(\lambda)'$ is reduced, any non-reduced quotient of the same dimension will have
an associated prime, which is not minimal. It follows from Lemma \ref{free_over_S_infty} that $M_{\infty}(\lambda^{\circ})$ is a faithful Cohen--Macaulay module
over $R_{\infty}(\lambda)$, thus this cannot happen, and so $R_{\infty}(\lambda)$ is reduced.
\end{proof}

 \begin{prop}\label{control_smooth} Let $x$ be a closed $E$-valued point of $\Spec R_{\infty}(\lambda)[1/p]$, let $r_x$ be the corresponding Galois representation and let
 $V(r_x)$ be the unitary Banach space representation defined in \S \ref{subsec:admissible unitary Banach}, and let $V(r_x)^{\lalg}$ be the subspace of
 locally algebraic vectors in $V(r_x)$. Then $V(r_x)^{\lalg}\cong \pi\otimes \pi_{\alg}(r_x)$, where $\pi$ is a smooth admissible representation which lies in $\Omega$.
 \end{prop}
 \begin{proof} Let $\Pi^{\lalg}$ be any locally algebraic representation of $G$. Let $W$ be an irreducible algebraic representation of $G$. We assume that $E$ is large enough, so that
 any such $W$ is absolutely irreducible.  Then $W$ is also an absolutely irreducible representation of the Lie algebra of $G$ and this category is semi-simple, as $E$ has characteristic $0$. This induces an
 isomorphism
 $$ \Pi^{\lalg}\cong \bigoplus_{W} \Hom_E(W, \Pi^{\lalg})^{\sm}\otimes_E W,$$
 where the sum is taken over all irreducible algebraic representations $W$ of $G$ and $\Hom_E(W, \Pi^{\lalg})^{\sm}$ denotes the smooth vectors for the conjugation action of $G$ on
 $\Hom_E(W, \Pi^{\lalg})$. The theory of the Bernstein centre asserts that any smooth representation $\pi$ of $G$ decomposes as
 $$ \pi\cong \bigoplus_{\Omega} \pi[\Omega],$$
 where the sum is taken over all the Bernstein components and $\pi[\Omega]$ is the maximal subquotient of $\pi$ lying in $\Omega$. Thus
 $$ \Pi^{\lalg}\cong \bigoplus_{W, \Omega} \Hom_E(W, \Pi^{\lalg})^{\sm}[\Omega] \otimes_E W.$$
We claim that  $V(r_x)^{\lalg}\cong \pi\otimes \pi_{\alg}(r_x)$ with
$\pi$ in $\Omega$. If this was not the case then by the above there would be $\lambda'=\lambda_{\sm}(\tau')\otimes \lambda'_{\alg}$,
such that either $\tau\not\cong \tau'$ or $\lambda'_{\alg}\not\cong \lambda_{\alg}$ and $\Hom_{J'}(\lambda', V(r_x)^{\lalg})\neq 0$. But Lemma \ref{quotient_control} implies that the inertial type of $r_x$ is
$\tau'$ and the Hodge--Tate weights correspond to the highest weight
of $\lambda_{\alg}'$. This is a contradiction. Since $\pi$ lies in $\Omega$, $\pi$ is admissible if and only if
$\Hom_J(\lambda(\tau), \pi)$ is finite dimensional. We have
$$\dim_E \Hom_J(\lambda(\tau), \pi)=\dim_ E \Hom_J(\lambda, V(r_x))=\dim_E M_{\infty}(\lambda^{\circ})\otimes_{R_{\infty}, x} E,$$
 where the last equality  follows from~\cite[Prop.~2.20]{paskunasBM}. Since $M_{\infty}(\lambda^{\circ})$ is a finitely generated $R_{\infty}$-module, we deduce that the
 above dimensions are finite and hence $\pi$ is admissible.
 \end{proof}

\begin{prop}\label{bound_smooth}Let $x,y$ be closed, $E$-valued points of $\Spec R_{\infty}(\lambda)[1/p]$, lying on the same irreducible component. If $x$ is smooth then
$$\dim _E \Hom_J (\lambda, V(r_x)^{\lalg})\le \dim_E \Hom_J(\lambda, V(r_y)^{\lalg}).$$
\end{prop}
\begin{proof} Since $\lambda$ is locally algebraic  we have
$$\Hom_J(\lambda, V(r_y)^{\lalg})\cong \Hom_J(\lambda, V(r_y)).$$
It follows from Proposition 2.20 of~\cite{paskunasBM} that
$$ \dim_E \Hom_J(\lambda, V(r_y))= \dim_E M_{\infty}(\lambda^{\circ})\otimes_{R_{\infty}, y} E.$$
If $x$ is a smooth closed point of $\Spec
R_{\infty}(\lambda^{\circ})[1/p]$ then the localisation
$R_{\infty}(\lambda^{\circ})_{\mm_x}$ at $x$ is a regular ring. Since
$M_{\infty}(\lambda^{\circ})$ is a Cohen--Macaulay module, so is its localisation
$M_{\infty}(\lambda^{\circ})_{\mm_x}$ at $x$. Since $R_{\infty}(\lambda^{\circ})_{\mm_x}$ is regular, the standard  argument using the Auslander--Buchsbaum theorem allows us to conclude
$M_{\infty}(\lambda^{\circ})_{\mm_x}$ is a free $R_{\infty}(\lambda^{\circ})_{\mm_x}$-module of rank equal to $ \dim_E M_{\infty}(\lambda^{\circ})\otimes_{R_{\infty}, x} E$.
Let $V(\qqq)$ be the irreducible component of $\Spec R_{\infty}(\lambda^{\circ})$ containing $x$.   By further localising $M_{\infty}(\lambda^{\circ})_{\mm_x}$  at $\qqq$ we deduce  that
$$  \dim_E M_{\infty}(\lambda^{\circ})\otimes_{R_{\infty}, x} E=\dim_{\kappa(\qqq)}  M_{\infty}(\lambda^{\circ})\otimes_{R_{\infty}} \kappa(\qqq).$$
Since the function $\pp\mapsto \dim_{\kappa(\pp)} M_{\infty}(\lambda^{\circ})\otimes_{R_{\infty}} \kappa(\pp)$ is upper semi-continuous on $\Spec R_{\infty}$, we conclude that for any $E$-valued point $y\in V(\qqq)$ we have
\[\dim_{\kappa(\qqq)}  M_{\infty}(\lambda^{\circ})\otimes_{R_{\infty}} \kappa(\qqq)\le   \dim_E M_{\infty}(\lambda^{\circ})\otimes_{R_{\infty}, y} E.\qedhere\]
\end{proof}

\begin{thm}\label{computation of locally algebraic vectors} Let $x$ be a closed $E$-valued point of $\Spec
  R_{\infty}(\sigma)[1/p]$, such that $\pi_{\sm}(r_x)$ is
  generic. Then $$ V(r_x)^{\lalg}\cong \pi_{\sm}(r_x) \otimes\pi_{\lalg}(r_x).$$
\end{thm}
\begin{proof} We claim that $x$ is a smooth point of $\Spec
  R_{\infty}(\lambda)$. Lemma \ref{union_irred} implies that it is enough to check that $x$ is a smooth point of $\Spec R_{\infty}(\lambda)'$. It follows from \cite[Thm 3.3.8]{kisindefrings} that
$R^{\square, \xi, \tau}_{\tilde{v}}[1/p]$ is a regular ring for all $v\in S_p\setminus\{\mathfrak p\}$. Moreover,  both $R^{\square}_{\tilde{v}_1}$ and
$\OO[[x_1, \ldots, x_{q-[\tilde{F}^+:\Q] n(n-1)/2}]]$ are regular
rings (the former by Lemma~\ref{lem: at v1, we have smooth unramified
    lifts}). We let

$$B:= \left(\widehat{\otimes}_{v \in
  S_p\setminus\{\p\}}R_\tv^{\square,\xi,\tau}\right)\wtimes {R}^{\square}_{\tv_1}\wtimes  \OO[[x_1,\dots,x_{q-[\tilde F^+:\Q]n(n-1)/2}]].$$
Corollary \ref{cor_risi} implies that $B[1/p]$ is a regular ring.  It
follows from~\cite[Thm.\ D]{allen2014deformations} that the restriction of $r_x$ to $G_F$ defines a smooth point $x_{\tp}$ of
$\Spec R_{\tp}^\square(\lambda)[1/p]$. Since complete local noetherian rings are excellent and the localisations of excellent rings are excellent, the smooth locus
is open in $\Spec R_{\tp}^\square(\lambda)[1/p]$.
Thus there is $f\in R_{\tp}^\square(\lambda)$, such that
$\Spec R_{\tp}^\square(\lambda)[1/pf]$ is an open neighborhood of $x_{\tp}$ contained in the smooth locus. It follows from Corollary \ref{cor_risi} that
$(R_{\tp}^\square(\lambda)\wtimes_{\OO} B)[1/pf]$ is a regular ring. Since $R_{\infty}(\lambda)'= R_{\tp}^\square(\lambda)\wtimes_{\OO} B$, this proves the claim.

 Let
 $y$ be any closed point in $V(\mathfrak a) \cap \Spec R_{\infty}(\sigma)[1/p]$, where \[\mathfrak a=(y_1, \ldots, y_h)\subset S_{\infty}.\] It follows from the Corollary \ref{cor: comparison with completed cohomology} via Proposition 2.20 of~\cite{paskunasBM}, that
 $V(r_y)$ is identified with the closed subspace of the completed cohomology $\tilde S_{\xi,\tau}(U^\p,\cO)_{\m}\otimes_{\cO} E$, consisting of vectors annihilated by the maximal ideal $\mm_y$ corresponding to $y$.
 Thus
 $$ V(r_y)^{\lalg} \cong (\tilde S_{\xi,\tau}(U^\p,\cO)_{\m}\otimes_{\cO} E)^{\lalg}[\mm_y]\cong \pi_{\sm}(r_y)\otimes \pi_{\alg}(r_y).$$ The last isomorphism follows from Prop.~3.2.4 of~\cite{MR2207783}, which shows that locally algebraic vectors of any given weight are precisely the algebraic automorphic forms of that weight, together with classical local-global compatibility (Thm 1.1 of~\cite{1202.4683}). A priori, $\pi_{\sm}(r_y)\otimes\pi_{\alg}(r_y)$ may appear with some multiplicity, but this multiplicity is seen to be $1$ by our choice of $U^\p$ (and the fact that we have fixed the action mod $p$ of the
  Hecke operators at $\tv_1$), and the irreducibility of $\rhobar$, together
  with~\cite[Thms.\ 5.4 and 5.9]{labesse}.

Proposition \ref{control_smooth} implies that $V(r_x)^{\lalg}\cong \pi\otimes \pi_{\alg}(r_x)$, where $\pi$ is a smooth representation lying in $\Omega$. Since $x$ lies in the support of $M_{\infty}(\sigma)$,
 $\Hom_K(\sigma(\tau), \pi)\neq 0$. It follows from Theorems \ref{thm: algebraic local Langlands},  \ref{thm:alpha equals beta} and Corollary \ref{cor: generic implies hecke iso} that $\pi_{\sm}(r_x)$ is a subrepresentation
 of $\pi$. Since $\lambda(\tau)$ is a type for $\Omega$, it is enough to show that
 $$\dim_E \Hom_J(\lambda(\tau), \pi)\le \dim_E \Hom_J(\lambda(\tau), \pi_{\sm}(r_x)) .$$
Since $r_x$ and $r_y$ have the same Hodge--Tate weights,
$\pi_{\alg}(r_y)=\pi_{\alg}(r_x)$, and the restriction of these representations to $J$ is equal to $\lambda_{\alg}$.  Since $\lambda_{\alg}$ is an irreducible representation of the Lie algebra of $G$,
we have isomorphisms
$$\Hom_J(\lambda(\tau), \pi)\cong \Hom_J(\lambda, V(r_x)^{\lalg})$$
and
$$\Hom_J(\lambda(\tau), \pi_{\sm}(r_y))\cong \Hom_J(\lambda, V(r_y)^{\lalg}).$$
Proposition \ref{bound_smooth} implies that
$$\dim_E \Hom_J(\lambda(\tau), \pi)\le \dim_E \Hom_J(\lambda(\tau), \pi_{\sm}(r_y)).$$
Thus it is enough to show that
 $$\dim_E \Hom_J(\lambda(\tau), \pi_{\sm}(r_y))\le \dim_E \Hom_J(\lambda(\tau), \pi_{\sm}(r_x)).$$
Since both $r_x$ and $r_y$ are potentially crystalline Theorem  \ref{thm: inertial local Langlands, N=0} together with Proposition \ref{stronger_dat} implies that there
is a surjection
$$\pi_1'\times\ldots \times \pi_r'\twoheadrightarrow \pi_{\sm}(r_y).$$
Since $\pi_{\sm}(r_x)$ is assumed to be generic, the same argument together with Corollary \ref{cor: generic implies hecke iso} gives an isomorphism
$$\pi_1\times\ldots \times \pi_r\cong \pi_{\sm}(r_x).$$
Moreover,  in this case we have $\pi_1\times\ldots \times \pi_r\cong \pi_{\alpha(1)}\times\ldots\times \pi_{\alpha(r)}$ for any permutation $\alpha\in S_r$. Since $\pi_{\sm}(r_x)$ and $\pi_{\sm}(r_y)$ lie in the
same Bernstein component, they have the same inertial support. Thus we may assume that each $\pi'_i$ is a twist of $\pi_i$ by an unramified character.
 This implies  that there is a $J$-equivariant surjection
$\pi_{\sm}(r_x)|_J \twoheadrightarrow \pi_{\sm}(r_y)|_J$, which implies the desired inequality.
 \end{proof}

\begin{rem} When $r$ is ordinary (more precisely when $r$ satisfies the assumptions on $r_w$ in Theorem 4.4.8 of~\cite{breuil-herzig}), it should be possible to prove that
$V(r)$ contains the locally-defined representation $\Pi(r)^{\mathrm{ord}}$ of~\cite{breuil-herzig}. This should follow precisely the same strategy of proof as Theorem 4.4.8 of \emph{op. cit.}, which roughly shows that, when $r$ is the restriction of a global automorphic representation, the representation $\Pi(r)^{\mathrm{ord}}$ occurs in completed cohomology. The global ingredients used for this are the computation of locally algebraic vectors in completed cohomology and the fact that the reduction mod $\varpi$ of completed cohomology is an injective object in the category of smooth $K$-representations. In our case, $V(r)$ is obtained by taking the fibre of $M_\infty^d[1/p]$ at a point corresponding to $r$, and $M_\infty^d[1/p]$ can be thought of as a patched version of completed cohomology. The analogous ingredients are the computation of locally algebraic vectors in Theorem~\ref{computation of locally algebraic vectors} and the projectivity of $M_\infty$ in Proposition~\ref{prop:Minfty is projective and has a G-action}.
\end{rem}

\begin{rem} The computation of $V(r_x)^{\lalg}$, when  $x$ is a closed point of $R_{\infty}(\lambda)[1/p]$ for which the corresponding representation $r_x$
is not necessarily potentially crystalline,
and related questions connected to the Breuil--Schneider conjecture,
will be discussed in the forthcoming thesis of Alexandre Pyvovarov.
\end{rem}

\section{The Breuil--Schneider conjecture}\label{sec:breuilschneider} 

Continue to assume that $p\nmid 2n$, and that $F$ is a finite
extension of $\Qp$. If $r:G_F\to\GL_n(E)$ is a de Rham
representation of regular weight then we say that $r$ is generic
 if $\pi_\sm(r)$ is generic. In this case, we
set \[\BS(r):=\pi_\alg(r)\otimes\pi_\sm(r).\]
(In fact, our $\BS(r)$ differs from the definition made in~\cite{MR2359853} in that $\pi_\alg(r)$ and $\pi_\sm(r)$ are their analogues in~\cite{MR2359853} times the characters $\det^{n-1}$ and $|\det|^{n-1}$, respectively. Since $(\det|\det|)^{n-1}$ is a unitary character, this makes no difference to the following conjecture.
See also Section 2.4 of~\cite{sorensenBS2} for a discussion of the difference between
these conventions.) The following is \cite[Conjecture 4.3]{MR2359853} (in the
open direction, in the generic case).

\begin{conj}\label{conj: BS}If $r:G_F\to\GL_n(E)$ is de
  Rham and has regular weight, then $\BS(r)$ admits a
  nonzero unitary Banach completion.
  \end{conj}
    \begin{rem}
    In fact, it seems reasonable (particularly in the light of
    Corollary~\ref{cor: our result on BS} below) to conjecture that
    there is even a nonzero \emph{admissible} completion. (We recalled
    the definition of admissibility in Section~\ref{subsec:admissible unitary Banach}.) Indeed, completed cohomology always gives rise to admissible Banach space representations, so this is a reasonable expectation from the point of view of the global $p$-adic Langlands correspondence.
Further motivation for focussing on admissible completions is provided by
the functor constructed in~\cite{scholze}, which takes as input admissible $\mathbb{F}$-representation of $G$.
  \end{rem}

  Fix a representation $r:G_F\to\GL_n(E)$, and assume from now on that
  $r$ is potentially crystalline of regular weight, and that $r$ is
  generic. By Remark~\ref{rem: getting a Banach space from a general r
    without having to worry about a pot diag lift of rbar} (and
  possibly replacing $E$ with a finite extension if necessary), we may
  replace $r$ with a conjugate representation so that
  $r:G_F\to\GL_n(\cO)$, and $\rbar$ satisfies the hypotheses of
  Section~\ref{sec:patching}. We can therefore carry out the
  construction of Section~\ref{sec:patching}, obtaining the patched
  module $M_\infty$.  Recall that $r$ is induced from an $\cO$-algebra homomorphism
$x:R_{\tp}^\square\to\cO$, which we  extended  to an $\cO$-algebra
homomorphism $y:R_\infty\to\cO$. Then $V(r)$ is obtained from the fibre
of $(M_\infty)^d[1/p]$ above the closed point of $R_\infty[1/p]$
determined by $y$.

 Write $\sigma_\sm(r)$ for $\sigma(\tau)$,
$\sigma_\alg(r)$ for $\pi_\alg(r)|_K$, and let
$\sigma:=\sigma_\alg(r)\otimes\sigma_\sm(r)$, keeping in mind the convention at the
end of Section \ref{god_save_the_queen}. (Also enlarge $E$ to another finite extension
if necessary as explained in that section.)
As above, we write $\cH(\sigma)$
for $\End_G(\cInd_K^G\sigma)$, which is isomorphic to $\End_G\bigl(\cInd_K^G\sigma(\tau)\bigr)$ via $\iota_\sigma$, so that $\pi_\sm(r)$ determines a character
$\chi_{\pi_\sm(r)}:\cH(\sigma)\to E$. Since $r$ is generic, we see from
Corollary~\ref{cor: generic implies hecke iso} that
$\pi_\sm(r)\cong\bigl(\cInd_K^G\sigma_\sm(r)\bigr)\otimes_{\cH(\sigma),\chi_{\pi_\sm(r)}}E$. Tensoring
with $\pi_\alg(r)$, we
have \[\BS(r)\cong(\cInd_K^G\sigma)\otimes_{\cH(\sigma),\chi_{\pi_\sm(r)}}E.\]

Since our representations $V(r)$
are unitary Banach representations, and since $\BS(r)$ is irreducible by \cite[Appendix]{MR1835001}, in order to prove Conjecture~\ref{conj: BS} it would be enough to
check that $\Hom_G(\BS(r),V(r))\ne 0$. While we cannot at present do
this in general, we are able to reinterpret the problem in terms of
automorphy lifting theorems, and deduce new cases of
Conjecture~\ref{conj: BS}. In particular, Corollary~\ref{cor: specifc
  cases of BS} below gives the first general results in the principal
series case.

\begin{thm}
\label{thm: BS for points in support of patched modules}
Suppose that $p\nmid 2n$, and that $r:G_F\to\GL_n(E)$ is a generic
potentially crystalline representation of regular weight. If $r$
corresponds to a closed point on an automorphic component of
$R_{\tp}^\square(\sigma)[1/p]$
{\em (}in the sense of Remark~{\em \ref{rem: definition of an automorphic component})},
then $\BS(r)$ admits a non-zero unitary admissible Banach completion.
\end{thm}
\begin{proof}[Proof of Theorem~\ref{thm: BS for points in support of
    patched modules}]As remarked above, since $\BS(r)$ is irreducible it suffices to show that
  $\Hom_G\bigl(\BS(r),V(r)\bigr)\ne 0$. This follows  immediately from Theorem \ref{computation of locally algebraic vectors}. We also observe that it admits a simpler direct proof.
   Proposition 2.20 of~\cite{paskunasBM}  implies that
  $$ \dim_E   \Hom_K\bigl(\sigma,V(r)\bigr)= \dim_E M_\infty(\sigma^\circ)\otimes_{R_\infty,y} E.$$
(More
specifically, in the notation of that paper we take $R=R_\infty$,
$\Theta=\sigma^\circ$, $V=\sigma$, $N=M_\infty$, and
$\mathrm{m}^\circ=\cO$, regarded as an $R_\infty$-module via
$y$. Note that~\cite{paskunasBM} assumes that $N$ is finitely generated as an $R[[K]]$-module, which is satisfied in our case: $M_\infty$ is a finitely generated $R_\infty[[K]]$-module, since it is finite over $S_\infty[[K]]$ and the $S_\infty$-action on $M_\infty$ factors through a map $S_\infty \to R_\infty$.)

Together with Frobenius reciprocity, this shows that
$$\Hom_G\bigl(\cInd_K^G\sigma,V(r)\bigr)=\Hom_K\bigl(\sigma,V(r)\bigr)\ne 0,$$ as $y$ is in the support of
$M_\infty(\sigma^\circ)[1/p]$ by assumption. Since
$\BS(r)$ is isomorphic to $(\cInd_K^G\sigma)\otimes_{\cH(\sigma),\chi_{\pi_\sm(r)}}E$,
we need only show that the action of $\cH(\sigma^\circ)$ on
$M_\infty(\sigma^\circ)\otimes_{R_\infty,y}\cO$ factors through the
character $\chi_{\pi_\sm(r)}$; but this is immediate from
Theorem~\ref{thm:alpha equals beta}. \end{proof}

  \begin{cor}
    \label{cor: our result on BS}Suppose that $p\nmid 2n$ and that
    $r:G_F\to\GL_n(E)$ is de Rham of regular weight and potentially diagonalisable in the sense of~{\em \cite{BLGGT}}. Suppose
    also that $r$  is generic. Then $\BS(r)$ admits a
    nonzero unitary admissible Banach completion.
  \end{cor}
\begin{proof} By Theorem~\ref{thm: BS for points in support of patched
    modules}, we need only prove that $r$ corresponds to a point  on
  an automorphic component of $R_{\tp}^\square(\sigma)[1/p]$.
Recalling that $y$
was chosen to correspond to the potentially diagonalisable
representation $r_{\textrm{pot.diag}}$ at the places $v\mid p$, $v\ne
\p$, this follows from Theorem A.4.1 of~\cite{blggU2}, which constructs a
global automorphic Galois representation corresponding to a point on the same
component of  $R_{\tp}^\square(\sigma)[1/p]$ as $r$ (\emph{cf.}\ the
proof of Corollary 4.4.3 of~\cite{geekisin}).

Indeed, the existence of a global automorphic Galois representation corresponding to a point on $R_{\tp}^\square(\sigma)[1/p]$ shows that $S_{\xi,\tau}(U_0, (\sigma^\circ)^d)_{\m}[1/p]$ is supported at this point. We have $R_{\tp}^\square(\sigma)[1/p]$-equivariant isomorphisms \[S_{\xi,\tau}(U_0, (\sigma^\circ)^d)_{\m}[1/p]\simeq\Hom_K\bigl(\sigma^\circ, \tilde S_{\xi, \tau}(U^\p,\cO)_{\m}\bigr)[1/p]\]\[\simeq \Hom_{\cO[[K]]}^\mathrm{cont}\bigl(M_\infty/\ga M_\infty,(\sigma^\circ)^d\bigr)[1/p],\] where the former comes by identifying locally algebraic vectors in completed cohomology and the latter follows from Schikhof duality and Corollary~\ref{cor: comparison with completed cohomology}. These isomorphisms imply that $\bigl(M_\infty(\sigma^\circ)/\ga M_\infty(\sigma^\circ)\bigr)^d[1/p]$ is supported at the same point of $R_{\tp}^\square(\sigma)[1/p]$ coming from a global automorphic Galois representation. Therefore, $M_\infty(\sigma^\circ)^d[1/p]$ is supported at a point of $R_\infty(\sigma)[1/p]$ on the same component as $r$. Finally, we conclude that the component of $R_{\tp}^\square(\sigma)[1/p]$ corresponding to $r$ is automorphic in the sense of Remark~\ref{rem: definition of an automorphic component}. \end{proof}

 \begin{cor}
    \label{cor: specifc cases of BS}Suppose that $p>2$, that
    $r:G_F\to\GL_n(E)$ is de Rham of regular weight, and that $r$ is
    generic. Suppose further that either
    \begin{enumerate}
    \item $n=2$, and $r$ is potentially Barsotti--Tate, or
    \item $F/\Qp$ is unramified, $r$ is crystalline, $n\ne p$ and $r$ has Hodge--Tate weights in
      the extended Fontaine--Laffaille range; that is, for each
      $\kappa:F\into E$, any two elements of $\HT_\kappa(r)$ differ by at
      most $p-1$.
\end{enumerate}
Then  $\BS(r)$ admits a nonzero unitary admissible
    Banach completion.
  \end{cor}
  \begin{proof}
Note that in case~(2), the hypothesis on the Hodge--Tate weights
implies that $p\ge n$, so as $p>2$ and we are assuming that $n\ne p$,
we certainly have $p\nmid 2n$. By Corollary~\ref{cor: our result on BS}, it is enough to check that our hypotheses
imply that $r$ is potentially diagonalisable; in case~(1), this is
Lemma 4.4.1 of~\cite{geekisin}, and in case~(2), it is the main result
of~\cite{GaoLiu12}.
  \end{proof}
  \begin{rem}
    \label{rem: removing extension of coefficients in BM results}The
    attentive reader will have noticed that since throughout the paper
    we assumed that $E$ is sufficiently large (and allowed it to be
    enlarged in the course of making our argument), we have
    not proved cases of Conjecture~\ref{conj: BS} as it is written,
    but rather an apparently weaker version, which allows a finite
    extension of scalars. However, Conjecture~\ref{conj: BS} is an
    immediate consequence of this version, in the following way: given
    an (admissible) unitary Banach completion of $\BS(r)\otimes_E E'$,
    where $E'/E$ is a finite extension, we may regard this
    completion as being a representation over $E$, and then the closure of
    $\BS(r)$ inside it gives the required representation.
  \end{rem}

\section{Relationship with a hypothetical $p$-adic local Langlands correspondence}
\label{sec:LL}

In this section we describe the relationship between our construction and a hypothetical
$p$-adic local Langlands correspondence.

\subsection{A hypothetical formulation of the $p$-adic local Langlands correspondence}
Perhaps the strongest hypothesis one might make regarding a $p$-adic local Langlands
correspondence is the following: that given $\rbar:G_F \to \GL_n(\F)$
with trivial endomorphisms\footnote{We make this assumption in what
follows for simplicity, since the discussion is purely hypothetical in any case.
If $\rbar$ admits non-trivial endomorphisms, then we would instead
work with the lifting ring $R_{\rbar}^{\square}$ in everything
that follows, and the representation $L_{\infty}$ would be endowed with a
further equivariant structure for the group $\GL_n$ acting
by ``change of basis''.}
and with associated deformation ring $R_{\rbar}$,
there exists a finitely generated $R_{\rbar}[[K]]$-module $L_{\infty}$
which is $\cO$-torsion free, and is equipped with an $R_{\rbar}$-linear
$G$-action which extends the $K$-action arising from its $R_{\rbar}[[K]]$-module
structure. \begin{remark}\label{L_infty for GL_2(Q_p)} In the case of
  $\GL_2(\mathbb{Q}_p)$, let $\bar{\pi}$ be the mod $p$ representation
  of $\GL_2(\mathbb{Q}_p)$ attached to $\rbar$ by the mod $p$ local
  Langlands correspondence of~\cite{MR2642409}. Then $L_\infty$ can be
  taken to be the projective envelope of $\bar{\pi}^\vee$ in an
  appropriate category $\mathfrak{C}(\cO)$ of
  $\GL_2(\mathbb{Q}_p)$-representations. See the discussion in Section
  1.2 of~\cite{paskunasimage} for more details. As it follows from the
  discussion on page 10 of \emph{op.\ cit.}, in the case when~$\rbar$, (and thus $\bar{\pi}$), has trivial endomorphisms, the local deformation ring $R_{\rbar}$ can be identified with the endomorphism ring of the projective envelope $L_\infty$ in $\mathfrak{C}(\cO)$.
\end{remark}

Given such an object $L_{\infty}$, then for any $r: G_F \to \GL_n(E)$
arising from an $\cO_E$-valued point $x$ of $\Spec R_{\rbar}$,
we may associate a unitary Banach space representation
$B(r) := (L_{\infty}\otimes_{R_{\rbar},x}\cO_E)^d[1/p]$ of $G$,
which should be ``the'' representation of $G$ associated
to $r$ via the $p$-adic local Langlands correspondence.

One might conjecture that such a
structure should exist and satisfy the following properties:

\subsubsection{Relationship with classical local Langlands}
\label{subsubsec:loc-alg}
For any potentially semistable
lift $r:G_F \to \GL_n(E)$ of $\rbar$
with regular Hodge--Tate weights,
the locally algebraic vectors of
$B(r)$ are isomorphic to $\BS(r)$. 

\subsubsection{Local-global compatibility}
\label{subsubsec:LG}
Using the notation for completed cohomology and the Hecke algebra
that acts on it introduced in the discussion preceding Corollary~\ref{cor: comparison
with completed cohomology},
there is an isomorphism of
$\T_{\xi,\tau}^{S_p}(U^{\mathfrak p},\mathcal O)_{\mathfrak m}[G]$-modules
$$\widetilde{S}_{\xi, \tau}(U^{\mathfrak p},\mathcal O)^d_{\mathfrak m}
\cong
\left(\T_{\xi,\tau}^{S_p}(U^{\mathfrak p},\mathcal O)_{\mathfrak m}
\otimes_{R_{\rbar}} L_{\infty}\right)^{\oplus m(U^\mathfrak{p})}$$ (here
$\T_{\xi,\tau}^{S_p}(U^{\mathfrak p},\mathcal O)_{\mathfrak m}$ is regarded
as an $R_{\rbar}$-algebra via the natural maps
$R_{\rbar} \to R^{\univ}_S \to
\T_{\xi,\tau}^{S_p}(U^{\mathfrak p},\mathcal O)_{\mathfrak m},$
where the first morphism corresponds to restricting global
Galois representations to the decomposition group at $\mathfrak p$,
and the second morphism is induced by the universal automorphic
deformation of $\rhobar$),
as well as analogous isomorphisms when we add the auxiliary $Q_N$-level
structure, as in the patching process.\footnote{We remark that in the local-global compatibility isomorphism above there is a multiplicity $m(U^{\mathfrak{p}})$, which depends on the level away from $\mathfrak{p}$. However, it should be possible to impose certain global conditions, as we do in Section~\ref{subsec:globalization}, which will ensure that this multiplicity can be taken to be $1$. This multiplicity should not increase when we add auxiliary $Q_N$-level, since we also apply the projection operators defined in~\cite{jack} at primes in $Q_N$.}

\subsection{The relationship between $L_{\infty}$ and $M_{\infty}$.}
If Condition~\ref{subsubsec:LG} held, then we would find that our patched $R_{\infty}[G]$-module $M_{\infty}$
is obtained by restricting $R_{\infty}\otimes_{R_{\rbar}} L_{\infty}$
to the smallest closed subscheme of $\Spec R_{\infty}$ which contains
the support of all the modules $(M_1,Q_{N'(N)}^{\square}\otimes_{S_{\infty}}S_{\infty}/\mathfrak b_N)_{K_{2N}}$
that arise in the patching process.
Optimistic conjectures (of ``big $R$ equals big $\T$''-type) might suggest that
this support is all of $R_{\infty}$, and thus that $M_{\infty}$ is obtained
from $L_{\infty}$ simply by pulling it back along the natural map
$\Spec R_{\infty} \to \Spec R_{\rbar}$.
For this reason, we are hopeful that our patched representation $M_{\infty}$
{\em is} a good candidate for (the pull-back to $R_{\infty}$ of) $p$-adic local
Langlands. (We note that, in the case of $\GL_2(\mathbb{Q}_p)$, we can prove that $L_\infty$, constructed as in Remark~\ref{L_infty for GL_2(Q_p)}, and $M_\infty$ have the desired relationship.) 

\subsection{The relationship with the Fontaine--Mazur conjecture}
Note that if both Conditions~\ref{subsubsec:loc-alg} and~\ref{subsubsec:LG} held,
together with an appropriate
``big $R$ equals big $\T$'' result, then we would find that
if $\rho$ is a de Rham deformation of $\rhobar$ corresponding to a point
of $R^{\univ}_S$, then it would contribute to the locally algebraic
vectors of completed cohomology.   However, locally algebraic vectors
in the completed cohomology arise
precisely from algebraic automorphic forms (see Prop.~3.2.4 of~\cite{MR2207783}),
and hence we would conclude that $\rho$ would be an automorphic Galois representation.
(This is an abstraction of the strategy used to deduce the Fontaine--Mazur conjecture
for most odd two-dimensional representations of $G_{\mathbb Q}$ in \cite{emerton2010local}.)

\subsection{Concluding remarks}
The preceding discussion shows that the question of whether one can in fact relate $M_{\infty}$ to a purely
local correspondence which satisfies the above two conditions is closely
related to the Fontaine--Mazur conjecture for deformations of $\rhobar$.
Given this, we expect it to be a difficult question to precisely
determine the support of $M_{\infty}$ and the related modules $M_{\infty}(\sigma^{\circ})$
in general, and we likewise expect it to be difficult to analyse the extent
to which $M_{\infty}$ arises from a purely local construction over $\Spec R_{\rbar}$.

Furthermore, it seems quite possible that our hypothetical formulation of the $p$-adic
local Langlands correspondence is too naive; even if a local correspondence of some
kind exists, it may be of a more subtle nature.  In this case, we would still expect it to
have a strong relationship to our patched modules $M_{\infty}$, but perhaps not as
direct as the one considered in the above discussion.

In any case,
whatever the eventual truth might be, the preceding discussion suggests that the
further investigation of the patched representation $M_{\infty}$ is a problem of substantial
interest, which we hope to return to in future work.

\appendix

\section{Completed tensor product and Serre's conditions}
Let $L$ be a finite extension of $\Qp$ with the ring of integers $\OO$ and residue field $k$. Let $\mathcal{C}$ be the category of complete local noetherian $\OO$-algebras,
which are $\OO$-flat and have residue field $k$. If $A$ and $B$ are objects in $\mathcal{C}$ then the completed tensor product over $\OO$ is defined as
$$A\wtimes_{\OO} B:= \varprojlim_n A/\mm_A^n \otimes_{\OO} B/\mm_B^n.$$
It is easy to see that $A\wtimes_{\OO} B$ is again in $\mathcal{C}$. For example, if $A=\OO[[x_1,\ldots, x_n]]$ and $B=\OO[[y_1, \ldots, y_m]]$, then $A\wtimes_{\OO} B\cong \OO[[x_1, \ldots, x_n, y_1,\ldots, y_m]]$, and every ring in $\mathcal{C}$ can be obtained as a quotient of such rings.
The aim  of this note is the following Proposition.

\begin{prop}\label{risi} Let $A, B$ be objects in $\mathcal{C}$, let $x\in A$ and let $y\in B$. If $A[1/px]$ and $B[1/py]$ satisfy Serre's condition $(R_i)$ {\em (}resp.\ $(S_i)${\em )}  then so does $(A\wtimes_{\OO} B)[1/pxy]$.
\end{prop}

We note that the completed tensor product does not commute with localisation. Indeed, $A[1/p]$ is not even a local ring. A standard application
of Serre's conditions $(R_i)$ and $(S_i)$, see \cite[\S 23]{matsumura},  yields the following result.

\begin{cor}\label{cor_risi} Let \textbf{P} be one of the following properties: reduced, regular, Cohen--Macaulay, normal. If $A[1/px]$ and $B[1/py]$ have \textbf{P} then so does
$(A\wtimes_{\OO} B)[1/pxy]$.
\end{cor}

We will split the proof of Proposition \ref{risi} into several steps.

\begin{lem}\label{risi_lem_1} Let $K$ be a field, $A$ a noetherian $K$-algebra and $K'$ a finite separable field extension of $K$. If $A$  satisfies $(R_i)$
{\em (}resp.\ $(S_i)${\em )} then so does $K'\otimes_K A$.
\end{lem}
\begin{proof} Let $B= K'\otimes_K A$, let $\mathfrak P$ be a prime ideal of $B$ and let $\mathfrak p=A \cap \mathfrak P$.
 Then $B$ is a free $A$-module of finite rank. This implies that $B_{\mathfrak P}$ is flat over $A_{\mathfrak p}$. Since the conditions
 $(R_i)$ and $(S_i)$ are local, $A_{\mathfrak p}$ satisfies $(R_i)$ (resp.\ $(S_i)$). It is enough to show that for every prime ideal $\mathfrak q$
 of $A_{\mathfrak p}$ the fibre ring $B_{\mathfrak P} \otimes_{A_{\mathfrak p}} \kappa(\mathfrak q)$ satisfies $(R_i)$ (resp.\ $(S_i)$),
 \cite[Thm.23.9]{matsumura}.  The fibre ring
 is a localisation at $\mathfrak P$ of $B\otimes_A \kappa(\mathfrak q)\cong K'\otimes_{K} \kappa(\mathfrak q)$. Since $K'$ is a finite separable
 extension of $K$, this ring is isomorphic to a finite product of fields, and hence is regular.
\end{proof}
\begin{lem}\label{risi_lem_2} Let $A, B\in \mathcal{C}$ be integral domains and let $K(A)$ and $K(B)$ be the quotient fields of $A$ and $B$, respectively. Then $K(A)\otimes_A (A \wtimes_{\OO} B) \otimes_B K(B)$
is a regular ring.
\end{lem}
\begin{proof} We first note
\begin{equation}\label{risi_eq_1}
K(A)\otimes_A (A \wtimes_{\OO} B) \otimes_B K(B)\cong S^{-1}_A ( S^{-1}_B( A \wtimes_{\OO} B)),
\end{equation}
where $S_A$ and $S_B$ denote the multiplicative sets $A\setminus\{0\}$ and $B\setminus \{0\}$, respectively.

If both $A$ and $B$
are formally smooth, then $A\wtimes_{\OO} B$ is formally smooth, as explained above, and hence regular. Since a localisation of a regular ring
is again regular, \cite[Thm. 19.3]{matsumura}, we deduce from \eqref{risi_eq_1} that the assertion holds if both $A$ and $B$ are formally smooth.

In general,  by Cohen's structure theorem for complete local rings there are subrings $A'\subset A$, $B'\subset B$, such that $A'$ and $B'$ are formally smooth objects of $\mathcal{C}$ and $A$ is a finite $A'$-module, $B$ is a finite $B'$-module. The last property  implies that
$$ A\wtimes_{\OO} B \cong A \otimes_{A'} ( A'\wtimes_{\OO} B')\otimes_{B'} B.$$
This induces an isomorphism between $K(A)\otimes_A (A \wtimes_{\OO} B) \otimes_B K(B)$ and
$$ K(A)\otimes_{K(A')} ( K(A')\otimes_{A'} (A'\wtimes_{\OO} B')\otimes_{B'} K(B'))\otimes_{K(B')} K(B).$$
Since $A'$ and $B'$ are formally smooth, and $K(A')$, $K(B')$ are of  characteristic $0$,   Lemma \ref{risi_lem_1} implies the assertion.
\end{proof}

\begin{lem}\label{risi_lem_3} Let $A, B \in \mathcal{C}$ with $A$ an integral domain with quotient field $K(A)$ let $y\in B$. If $B[1/py]$ satisfies $(R_i)$
{\em (}resp.\ $(S_i)${\em )},
then so does $K(A)\otimes_A (A\wtimes_{\OO} B)[1/y]$.
\end{lem}
\begin{proof} Since $A$ is $\OO$-flat, $A\wtimes_{\OO} B$ is $B$-flat. Moreover,
\[C:=K(A)\otimes_A (A\wtimes_{\OO} B)[1/y]\] is $A\wtimes_{\OO} B$-flat, since it is
a localisation of $A\wtimes_{\OO} B$ at $S_A$ and $\{1, y, y^2,\ldots\}$. Thus $C$ is $B$-flat. Let $\mathfrak P$ be a prime ideal of $C$  and let $\mathfrak p= B \cap \mathfrak P$. Then $C_{\mathfrak P}$
is flat over $B_{\mathfrak p}$. We note that $p, y\not \in \mathfrak p$. Since $B[1/py]$ is assumed to satisfy $(R_i)$ (resp.\ $(S_i)$) and
these conditions are local, $B_{\mathfrak p}$ satisfies $(R_i)$ (resp.\ $(S_i)$). It follows from \cite[Thm. 23.9]{matsumura}  that it is
enough to show that for every prime $\mathfrak q$ of $B_{\mathfrak p}$ the fibre ring $C_{\mathfrak P} \otimes_{B_{\mathfrak p}} \kappa(\mathfrak q)$ satisfies $(R_i)$ (resp.\ $(S_i)$). We claim that $C_{\mathfrak P} \otimes_{B_{\mathfrak p}} \kappa(\mathfrak q)$ is regular, so that these conditions
are satisfied. Since it  is the localisation
of $C\otimes_B \kappa(\qqq)$ at $\mathfrak P$, it is enough to show that  $C\otimes_B \kappa(\qqq)$ is regular. Since
$$ C\otimes_B \kappa(\qqq)\cong K(A) \otimes_A (A \wtimes_{\OO} B/\qqq)\otimes_{B/\qqq} K(B/\qqq)$$
the assertion follows from Lemma \ref{risi_lem_2}.
\end{proof}

\begin{proof}[Proof of Proposition \ref{risi}] The idea is the same as in the proof of Lemma \ref{risi_lem_3}. Let $C:=A\wtimes_{\OO} B$
and let $\mathfrak P$ be a prime ideal of $C$ not containing $p$, $x$ and $y$. Let $\mathfrak p:=A \cap \mathfrak P$. Then
$A_{\mathfrak p}$ satisfies $(R_i)$ (resp.\ $(S_i)$) and $C_{\mathfrak P}$ is flat over $A_{\mathfrak p}$. It is enough to
show that for all prime ideals $\mathfrak q$ of $A_{\mathfrak p}$ the fibre ring $C_{\mathfrak P}\otimes_{A_{\mathfrak p}} \kappa(\qqq)$
satisfies $(R_i)$ (resp.\ $(S_i)$). The ring $K(A/\mathfrak q)\otimes_{A/\mathfrak q} ( A/\mathfrak q \wtimes_{\OO} B)[1/y]$ satisfies $(R_i)$
(resp.\ $(S_i)$) by Lemma \ref{risi_lem_3}. Since the fibre ring is the localisation of this ring at $\mathfrak P$, the assertion follows.
\end{proof}

\emergencystretch=3em
\printbibliography
\end{document}